\documentclass[11pt]{article}

\usepackage[margin=1in]{geometry}
\usepackage{amsmath,amssymb,amsthm}
\usepackage{mathtools}
\usepackage{bm}
\usepackage{graphicx}
\usepackage{booktabs}
\usepackage{multirow}
\usepackage{enumitem}
\usepackage{natbib}
\usepackage{algorithm}
\usepackage{algpseudocode}
\usepackage{xcolor}
\usepackage{hyperref}
\usepackage{amsmath,amssymb,amsthm,amsfonts,mathtools}
\usepackage{bm}
\usepackage{graphicx}
\usepackage{booktabs}
\usepackage{enumitem}
\usepackage{microtype}
\usepackage{xcolor}
\usepackage{hyperref}
\usepackage[nameinlink,noabbrev]{cleveref}
\usepackage{array}
\usepackage{multirow}
\theoremstyle{plain}
\newtheorem{theorem}{Theorem}
\newtheorem{proposition}{Proposition}
\newtheorem{corollary}{Corollary}
\newtheorem{lemma}{Lemma}
\newtheorem{assumption}{Assumption}
\newtheorem{remark}{Remark}

\newcommand{\Var}{\operatorname{Var}}

\newcommand{\cF}{\mathcal{F}}

\newcommand{\IF}{\mathrm{IF}}
\newcommand{\PoF}{\mathrm{PoF}}
\newcommand{\PCS}{\mathrm{PCS}}
\newcommand{\PFS}{\mathrm{PFS}}
\newcommand{\CVaR}{\mathrm{CVaR}}

\newcommand{\SL}{\mathrm{SL}}

\newcommand{\Halmos}{\qed}

\usepackage{graphicx}
\usepackage{subfigure, epsfig}
\usepackage{natbib}

\hypersetup{
    colorlinks=true,
    linkcolor=blue,
    citecolor=blue,
    urlcolor=blue
}

\title{
\textbf{The Implementation Cost of Fairness in Service Policy Selection}
}

\author{
Junjie Liu$^{1}$, Mingjie Hu$^{2}$, Kejia Hu$^{3}$, Siyang Gao$^{1}$, Jianqiang Hu$^{2}$\\[0.8em]\small
$^{1}$Department of systems engineering, City University of Hong Kong, Hong Kong SAR, China\\ \small
$^{2}$School of Management, Fudan University, Shanghai, China\\ \small
$^{3}$Saïd Business School, University of Oxford, Oxford, United Kingdom
}

\date{}

\begin{document}

\maketitle

\begin{abstract}
Service organizations use simulation, pilot studies, and historical data
to select service policies that balance
aggregate performance and fairness. Existing work primarily evaluates the
operational cost of fairness, defined as the performance loss caused by
restricting the feasible policy set. In this research, we identify and study implementation cost as a distinct and equally important dimension of fairness, defined as the sampling effort required to verify fairness and reliably select the best fair policy. Specifically, we consider fixed-budget policy selection under fairness requirements based on
mean performance, under-service risk, quantiles, and upper tail outcomes, and show that different fairness metrics can induce the same fair policy set
and the same operational cost while requiring substantially different
amounts of evidence because their estimators have different local
verification difficulty. We further show that, near a critical fairness tolerance, the required sampling budget for implementation scales inversely with the square of the distance to the boundary, creating a finite-budget
implementation gap. These results imply that the fairness metric,
the fairness tolerance, and the available sampling budget should be designed
jointly, because a population-level fairness requirement may be operationally
attractive but not statistically implementable at the available budget.
To translate this perspective into practice, we develop a fairness-guided adaptive
allocation algorithm that directs samples toward the ranking and
fairness verification comparisons governing false selection. Experiments in
synthetic, call center, and emergency department settings demonstrate the implementation cost mechanisms identified in this work and show that the proposed allocation uses the sampling budget much more effectively.
\end{abstract}

\vspace{1em}


\section{Introduction}

Service organizations often use simulation studies, pilot
experiments, and historical data to select service policies, such as staffing, routing, scheduling,
triage, and priority policies. These policies determine not only aggregate
system performance but also how service outcomes are distributed across
customer groups. A fairness requirement therefore has an operational
consequence because it restricts the policies that an organization is
willing to deploy. At the same time, since fairness must be assessed from finite and uncertain data in practice, it also creates an evidence requirement, which in our setting is a requirement in the form of a sampling budget sufficient to estimate the relevant group-level outcomes and verify fairness reliably. Group-level means, threshold probabilities, quantiles, and tail outcomes are rarely known and must be estimated to determine whether a policy satisfies the intended fairness requirement.

Consider a hospital choosing among triage and fast-track rules using a
limited simulation study. The hospital may evaluate policies through
aggregate service performance while requiring comparable waiting time
burdens or service guarantee attainment across patient groups. Equalizing
mean waiting time, a high waiting time quantile, or an upper tail
waiting time measure may lead to the same or similar policy recommendation
in the population model. The corresponding estimators, however, use
different parts of the outcome distribution. Means aggregate information
from all observations, and extreme quantiles depend on local density and
tail measures rely disproportionately on relatively rare observations. The
same population-level policy recommendation can therefore require very
different amounts of evidence before it can be implemented with the same
statistical reliability.

This mechanism is widely observed in service operations. A call center evaluating priority and
cross-training rules may observe abundant information about average waits
while seeing relatively few extreme delays for a smaller customer group. An
emergency department may estimate overall service guarantee attainment
precisely while having much less information about severe waiting burdens
for patients requiring communication support. In both settings, managers
can specify a clear population-level parity requirement before the available
data can reliably distinguish a near-boundary fair policy from an infeasible
one. The resulting evidence requirement can delay deployment, motivate a
larger pilot or simulation study, or affect the choice among fairness
metrics that have similar operational implications.

Literature on the price of fairness formalizes the \emph{operational cost of fairness} as the loss in aggregate performance caused by restricting the feasible policy set \citep{bertsimas2011price,caragiannis2012efficiency}. That measure treats the quantities defining fairness as known. In data-driven service operations, the organization must also estimate those quantities, verify the fair policy set, and rank the policies within that set. We call the total sampling budget required to complete these tasks at a prescribed probability of correct selection the \emph{implementation cost of fairness}. This term refers to the total evidence requirement induced by a fairness specification. It is different from the incremental sample requirement relative to unconstrained policy selection, which we define separately in Section~\ref{sec:model}.

The two costs answer different practical questions. Operational cost asks
how much aggregate performance is surrendered once the fair policy set is
known. Implementation cost asks how much evidence is required to learn that
set and reliably identify its best member. These two costs do not need to be aligned.
A fairness requirement may preserve aggregate performance but be difficult
to verify, while another may impose a larger performance loss but be
statistically easier to distinguish. Price of fairness alone therefore does
not determine whether a fairness commitment can be credibly supported by a
given simulation, pilot, or historical data sampling budget.

This distinction leads to the organizing perspective of the paper. A fairness specification is \(\mathfrak f=(\psi,\kappa)\), where \(\psi\) is a distributional metric and \(\kappa\) is the tolerated cross-group disparity. \(\mathfrak f\) affects service policy selection through two distinct channels. Through the operational channel, the metric
\(\psi\) and tolerance \(\kappa\) determine the fair policy set, the best fair
policy, and the resulting price of fairness. Through the implementation
channel, they determine which distributional quantities must be estimated and
how difficult the relevant cross-group comparisons are to verify, while the
available sampling budget \(T\) determines whether those comparisons can be
resolved with the desired statistical reliability. The three design variables
therefore play distinct but interdependent roles, where the metric determines the
service outcome protected and the estimator used for verification, the
tolerance determines both the accepted disparity and the distance to the
fairness boundary, and the budget determines the evidence available to resolve
the relevant comparisons. Consequently, a fairness specification may be
operationally attractive but not statistically implementable at the available
budget, and operationally equivalent specifications may require substantially
different budgets. Accordingly, \(\psi\), \(\kappa\), and \(T\)
should be designed jointly instead of being treated as sequential and independent service design decisions.

Statistically, fair policy selection under finite evidence is a joint
ranking and verification problem. Ranking error arises when a policy is empirically classified as fair and appears preferable to the true best fair policy, and verification error arises when sampling noise changes the estimated
fairness feasibility of the best fair policy.
The metric determines the estimator used in the cross-group comparisons,
the tolerance determines their distance from the feasibility boundary, and
the allocation of the available budget determines how accurately the
relevant ranking and verification comparisons can be resolved. This
structure also explains why uniform sampling and allocation rules designed
only for objective ranking can be inefficient for fair policy selection.

Our central research question is how the choice of a fairness metric and
tolerance interacts with a limited sampling budget to determine whether the
best fair service policy can be reliably identified and implemented. We study a fixed-budget setting with finitely many candidate policies and multiple customer groups. The model covers mean performance, under-service risk, high quantile, and upper tail parity requirements. It also permits the output used for aggregate performance to differ from, and be observed jointly with, the output used for fairness assessment. This feature connects the theory to service systems whose objectives combine waiting, abandonment, service attainment, and policy complexity.

The paper makes three contributions. First, we explicitly analyze two costs associated with fairness in service policy selection. Existing research has focused primarily on the operational cost of fairness, captured by the price of fairness. We additionally define the implementation cost of fairness as the total sampling budget required to verify the fair policy set and reliably identify its best policy at a target probability of correct selection (PCS). This distinction allows managers to compare the performance loss induced by a fairness specification with the evidence required to support it. We also distinguish this total implementation requirement from the incremental evidence attributable to fairness relative to unconstrained policy selection. When the required evidence exceeds the available sampling budget, managers can respond by increasing the evaluation budget, adjusting the tolerance, or comparing substantively appropriate metrics with different verification difficulty.

Second, we identify two mechanisms that generate high implementation cost. Under common-shape location-shift models, translation-equivariant metrics can induce the same fair policy set and price of fairness, while their influence function variances generate different local verification hardness. Near a critical fairness tolerance, the budget required for a target PCS grows proportionally to \((\kappa-\kappa_0)^{-2}\), where \(\kappa\) denotes the operational fairness tolerance and \(\kappa_0\) denotes the corresponding critical fairness tolerance. Equivalently, a fixed budget leaves an implementation gap of order \(\sqrt{\log(1/\delta)/T}\) around the critical boundary, where \(\delta\) is the target probability of false selection and \(T\) is the total sampling budget. These results explain when operationally similar fairness requirements demand different amounts of evidence.

Third, we translate the fairness perspective into a
fairness-guided allocation algorithm for the joint ranking and verification
problem. Given a fairness metric, tolerance, and fixed sampling budget, the
algorithm directs samples toward comparisons that are difficult because
policies are close in objective value or because a cross-group fairness gap
is close to the tolerance. It thus provides a practical
mechanism for using the available sampling budget efficiently. The proposed method is also evaluated by numerical experiments, where the synthetic examples demonstrate the implementation cost mechanisms identified in this work, and the
call center and emergency department examples demonstrate how the proposed allocation algorithm performs in service systems.

Our analysis also has a direct managerial implication for achieving fairness. It suggests that a service
organization should design its fairness metric, fairness tolerance, and
available sampling budget jointly. The appropriate metric should reflect
the aspect of service experience the organization intends to protect, but
its statistical requirements must also be recognized. The tolerance should
reflect the disparity the organization is willing to accept, while also
accounting for the evidence needed to distinguish policies near the
fairness boundary. The available budget, in turn, determines which
commitments can be supported with the desired reliability. An attractive
population-level fairness requirement may therefore remain statistically
unreliable at the available budget, while an operationally similar
requirement may be implemented with substantially less evidence.

The remainder of the paper develops this perspective in sequence. Section~\ref{sec:literaure} reviews the related literature. Section~\ref{sec:model} introduces the fair policy selection problem. Section~\ref{sec:equiv} compares the operational and statistical implications of fairness specifications. Section~\ref{sec:wedge} studies the finite-budget implementation gap. Section~\ref{sec:pro} develops the sampling algorithm. Section~\ref{sec:numerics} reports the numerical experiments. Section~\ref{sec:conclu} concludes the paper.

\section{Related Literature}\label{sec:literaure}

This research is related to three streams of literature.

The first stream is fairness in service operations. Fairness is central in systems where customers, patients, or jobs experience congestion, priority rules, and unequal access to service capacity. \citet{larson1987perspective} explains how physical waiting, perceived justice, and the social context of delay jointly shape queueing experiences. \citet{aviitzhak2008quantifying} develops formal queueing fairness metrics, while \citet{gans2003telephone} and \citet{aksin2007modern} review call center decisions that affect waiting, service performance, and service attainment. In emergency department operations, \citet{mandelbaum2012fair} studies fairness in routing patients to hospital wards. The price-of-fairness framework \citep{bertsimas2011price,caragiannis2012efficiency} quantifies the aggregate performance loss induced by equity constraints. \citet{zhu2026selection} study best policy selection under subpopulation fairness constraints, requiring the selected policy to satisfy minimum performance thresholds across all prespecified subpopulations. Collectively, this literature clarifies how fairness requirements shape service priorities, feasible policies, and attainable performance. What remains unaddressed is the implementation cost of fairness as a distinct counterpart to its operational cost. We define this implementation cost as the sampling budget required to estimate the distributional quantities that define fairness, verify policy feasibility, and reliably identify the best fair policy. We study how the fairness specification determines this second cost of fairness and compare it with the operational cost.

The second stream is distributional performance metrics in service systems. Service organizations often evaluate system performance using more than averages. In call centers, service performance targets and waiting time distributions are main operational metrics, and customer experience may depend on information, delay variability, and tail outcomes in addition to mean delays \citep{koole2004performance, avramidis2005modeling, aksin2007modern, roubos2012service}. The statistical properties of such distributional metrics are well known to differ across metrics. Sample means are smooth functionals, threshold-based service metrics like under-service risk depend on the probability of service failure, and quantiles depend on the density near the target quantile \citep{bahadur1966note,koenker1978regression}. Upper tail performance metrics
such as CVaR have been studied in risk optimization and simulation
\citep{rockafellar2000optimization,hong2009cvar,hong2011varcvar,
gordy2010nested}. The broader algorithmic fairness literature also
shows that different fairness criteria encode different distributional
objects, as in individual-similarity-based notions of fairness
\citep{dwork2012fairness} and group parity notions for classification
errors \citep{hardt2016equality}. We consider this metric dependence in service policy selection. The selected fairness specification determines the estimator used in each cross-group comparison and therefore changes the PCS for the best feasible policy, even when the induced fair policy set is unchanged. Thus, the metric choice affects not only the service outcomes an organization seeks to protect, but also the amount of data needed to identify a fair policy with confidence.

The third stream is fixed-budget ranking and selection (R\&S) in
simulation optimization. This line of research studies how to allocate a
limited sampling budget across a finite set of alternatives to
maximize the PCS for the best alternative
\citep{chen2011stochastic}. The closely related fixed-budget best arm identification (BAI) studies the analogous
problem in multi-armed bandits. A representative fixed-budget R\&S
method is the optimal computing budget allocation (OCBA), which approximates
the PCS and allocates more samples to
alternatives that are close to the current best or have larger sample
variances \citep{chen2000simulation}. Other common methods include Bayesian value of 
information rules such as expected improvement and knowledge gradient
\citep{jones1998efficient,frazier2008knowledge,frazier2009knowledge},
and BAI algorithms such as successive rejects and top-two sampling
\citep{audibert2010best,russo2016simple}. Most of this literature
focuses on unconstrained selection with a single objective and known
feasibility, while constrained R\&S extends the problem to settings in
which feasibility is itself estimated from simulation
\citep{lee2012approximate}.

Our problem shares the estimated feasibility feature of constrained R\&S, but its fairness structure creates additional coupling. Feasibility here corresponds to cross-group comparisons of distributional functionals. A replication can jointly inform the aggregate objective and several fairness comparisons. The metric changes the local variance of those comparisons, and the tolerance is itself a design variable that changes the distance to the feasibility boundary. These features produce a coupled ranking and verification problem with a metric-specific statistical structure.

\section{Fair Policy Selection and the Costs of
Fairness}
\label{sec:model}

This section introduces the service policy selection model and fairness specification, and then defines the two costs analyzed in the paper, namely the operational cost and the implementation cost of fairness.

\subsection{Policy and System Performance}

We consider a decision maker who selects one policy from a finite candidate set \(\mathcal K=\{1,\ldots,K\}\) for a service system with customer groups \(\mathcal G=\{1,\ldots,G\}\). Group \(g\) receives a nonnegative system weight \(w_g\), with \(\sum_{g\in\mathcal G}w_g=1\).

For each policy-group pair \((k,g)\), one independent sampling replication produces a jointly observed pair \((Y_{kg},X_{kg})\). The objective output \(Y_{kg}\) records the contribution of that group to aggregate performance. It may combine several operational outcomes and policy penalties. The fairness output \(X_{kg}\) records the service outcome used to assess cross-group parity. Let \(R_{kg}\) and \(P_{kg}\) denote the marginal distributions of \(Y_{kg}\) and \(X_{kg}\), respectively. Dependence between \(Y_{kg}\) and \(X_{kg}\) within a replication is allowed. Across replications and policy-group pairs, the sampling streams are independent, and $(Y_{kg,1},X_{kg,1}),(Y_{kg,2},X_{kg,2}),\ldots$
are identically distributed copies of \((Y_{kg},X_{kg})\).

Denote \(\mu_{kg}=\mathbb E[Y_{kg}]\). The aggregate performance of policy \(k\) is
\begin{equation}\label{eq:objective}
\mu_k=\sum_{g\in\mathcal G}w_g\,\mathbb E[Y_{kg}].
\end{equation}
Larger values of \(\mu_k\) are preferred. A policy-specific complexity cost can be included in each \(Y_{kg}\). Since the group weights sum to one, this representation covers the additive complexity penalties used in the experiments. The scalar model with \(Y_{kg}=X_{kg}\) is an important special case. The joint output formulation also covers the service system experiments, where the objective combines waiting, abandonment, service attainment, and complexity while the fairness constraint is based on waiting time.

For analytical clarity, the main theory treats \(X_{kg,i}\) as a scalar observation. A simulation replication may instead return a finite within-replication batch, such as all group-specific customer waits in one simulated day. In that case, the independent sampling unit is the replication, and the scalar influence function term is replaced by its replication-level batch contribution. The same allocation analysis applies when these contributions satisfy Assumption~\ref{ass:lal}.

\subsection{Fairness Specifications}

Let \(\psi(P)\) be a distributional metric that maps a fairness outcome distribution \(P\) to a real-valued summary. The metric determines the operational feature compared across groups and the estimator used for verification. We consider four representative fairness metrics.
\begin{itemize}
    \item \textbf{Mean performance parity.} The metric is \(\psi(P)=\mathbb E[X]\).
    \item \textbf{Under-service risk parity.} For a service threshold \(s\), the metric is \(\psi(P)=\mathbb P(X\le s)\).
    \item \textbf{High quantile parity.} For \(\beta\in(0,1)\), the metric is
    \[
    \psi(P)=Q_\beta(P)=\inf\{x\mid\mathbb P(X\le x)\ge\beta\}.
    \]
    \item \textbf{Upper tail parity.} We use the upper tail CVaR convention
    \[
    \psi(P)=\CVaR_\beta^+(P)
    =\inf_{z\in\mathbb R}\left\{z+\frac{1}{1-\beta}\mathbb E[(X-z)_+]\right\}.
    \]
    For a continuous distribution, this equals \(Q_\beta(P)+(1-\beta)^{-1}\mathbb E[(X-Q_\beta(P))_+]\). It represents the upper tail average. Thus it describes the worst tail when \(X\) is a service burden such as waiting time and the best tail when \(X\) is a performance measure.
\end{itemize}

A \emph{fairness specification} is a pair \(\mathfrak f=(\psi,\kappa)\), where \(\psi\) is the selected metric and \(\kappa\ge0\) is the tolerated disparity. Policy \(k\) is fairly feasible under \(\mathfrak f\) when
\begin{equation}\label{eq:fairness_feasible}
\max_{g<g'}|\psi(P_{kg})-\psi(P_{kg'})|\le\kappa.
\end{equation}
Define its fairness score and the induced fair policy set by $
\varepsilon_k^{\mathfrak f}=\max_{g<g'}|\psi(P_{kg})-\psi(P_{kg'})|$ and $\mathcal F_{\mathfrak f}=\{k\in\mathcal K\mid\varepsilon_k^{\mathfrak f}\le\kappa\}$ respectively.
The best fair policy is
\begin{equation}\label{eq:kstar}
k_{\mathfrak f}^\star\in\arg\max_{k\in\mathcal F_{\mathfrak f}}\mu_k.
\end{equation}

To simplify exposition, we assume throughout that $\mathcal F_{\mathfrak f}\neq\emptyset$ and that the best fair policy is unique. 
Moreover, we assume that all pairwise disparities in the metric values are 
nonboundary, i.e., 
\(
|\psi(P_{kg}) - \psi(P_{kg'})| \neq \kappa
\)
for every policy $k$ and every pair of groups $(g,g')$ under consideration.

\begin{remark}
The pairwise parity constraints provide a transparent representation of group fairness and expose the comparisons that determine sampling difficulty. Network-based and multi-criterion constraints can also be incorporated by replacing the set of group-pair comparisons.
\end{remark}

\subsection{Policy Selection}

Let \(N_{kg}\) be the number of replications allocated to pair \((k,g)\), and let \(\bar Y_{kg}\) be the sample mean of the objective outputs. For \(N_{kg}\ge1\), define $\widehat P_{kg}=\frac{1}{N_{kg}}\sum_{i=1}^{N_{kg}}\delta_{X_{kg,i}}$ and $\widehat\mu_k=\sum_{g\in\mathcal G}w_g\bar Y_{kg}$.
The empirical fair policy set is
\begin{equation}\label{eq:emp-feasible}
\widehat{\mathcal F}_{\mathfrak f}
=\left\{k\in\mathcal K\mid\max_{g<g'}|\psi(\widehat P_{kg})-\psi(\widehat P_{kg'})|\le\kappa\right\}.
\end{equation}

The total budget is fixed at \(T\), with \(\sum_{k,g}N_{kg}=T\), and the allocation proportions are \(\alpha_{kg}=N_{kg}/T\). If \(\widehat{\mathcal F}_{\mathfrak f}\) is nonempty, the plug-in rule selects
\begin{equation}\label{eq:plugin}
\widehat k_{\mathfrak f,T}\in\arg\max_{k\in\widehat{\mathcal F}_{\mathfrak f}}\widehat\mu_k.
\end{equation}
If the empirical fair set is empty, the rule selects the policy with the smallest estimated fairness score, using estimated aggregate performance and then the policy index to break ties. This fallback makes the decision rule defined for every sample path and becomes irrelevant asymptotically under the previously stated regularity assumptions, namely a unique best fair policy and no pairwise metric disparity exactly equal to the tolerance.

Let \(\Pi(T)\) denote the admissible sampling allocation rules paired with the fixed plug-in selection rule in \eqref{eq:plugin} with budget \(T\). For \(\pi\in\Pi(T)\), define
\begin{align*}
\PCS_{\mathfrak f}^{\pi}(T)&=\mathbb P_{\pi}(\widehat k_{\mathfrak f,T}=k_{\mathfrak f}^\star),\\
\PCS_{\mathfrak f}(T)&=\sup_{\pi\in\Pi(T)}\PCS_{\mathfrak f}^{\pi}(T).
\end{align*}
Thus, the unadorned \(\PCS_{\mathfrak f}(T)\) is the best attainable PCS at budget \(T\). When a specific allocation or algorithm is analyzed, we use a superscript when the distinction matters and suppress it only after the rule has been fixed. We write the probability of false selection as \(\PFS=1-\PCS\) with matching superscripts.

\subsection{The Operational and Implementation Costs of Fairness}

The operational cost is the objective loss caused by the fair policy set. We measure it through the price of fairness
\begin{equation}\label{eq:pof}
\PoF_{\mathfrak f}=V_0-V_{\mathfrak f},
\end{equation}
where $V_0=\max_{k\in\mathcal K}\mu_k $ and $V_{\mathfrak f}=\max_{k\in\mathcal F_{\mathfrak f}}\mu_k$.

For a target PCS level \(1-\delta\), the \emph{implementation cost of fairness} is
\begin{equation}\label{eq:burden}
B_{\mathfrak f}(\delta)=\inf\{T\in\mathbb R_+\mid \PCS_{\mathfrak f}(T)\ge1-\delta\}.
\end{equation}
We use this term for the total evidence requirement induced by the fairness specification, including the samples needed for objective ranking. It is not restricted to the incremental requirement relative to unconstrained selection. If \(B_0(\delta)\) denotes the corresponding unconstrained implementation cost, the incremental evidence attributable to the specification is
\[
\Delta B_{\mathfrak f}(\delta)=B_{\mathfrak f}(\delta)-B_0(\delta).
\]
At target PCS level $1-\delta$, we say that a fairness specification is statistically implementable under the available budget \(T\) whenever $B_{\mathfrak f}(\delta)\le T$. We use the term \emph{implementable fairness} to describe this finite evidence perspective.

For a specific algorithm \(\mathcal A\), its algorithm-specific cost is
\[
B_{\mathfrak f}^{\mathcal A}(\delta)
=\inf\{T\in\mathbb R_+\mid \PCS_{\mathfrak f}^{\mathcal A}(T)\ge1-\delta\}.
\]
The numerical experiments report these algorithm-specific crossing budgets. This notation separates the theoretical evidence benchmark from the budget attained by FGA or a comparison procedure.

\subsection{The Local Error Structure of Fairness Verification}
\label{subsec:local_geometry}

To analyze the statistical difficulty of fairness verification, we use
the standard local asymptotic linear (LAL) representation of plug-in
estimators. Denote the influence function of a distributional metric \(\psi\)
at distribution \(P_{kg}\) by
\(\IF_\psi(\cdot;P_{kg})\). It describes the first-order effect of a single
outcome on the value of the estimated metric. Formally, when the
derivative exists,
\begin{equation}
\label{eq:if_def}
\IF_\psi(x; P_{kg})
=
\lim_{\varepsilon \to 0}
\frac{
\psi\left((1-\varepsilon)P_{kg}+\varepsilon\delta_x\right)
-
\psi(P_{kg})
}{\varepsilon}.
\end{equation}

\begin{assumption}
\label{ass:lal}
For each policy-group pair \((k,g)\), if the distribution \(P_{kg}\) lies
in a metric-based regularity class specified in Section \ref{app:lal-proof}, the plug-in estimator
\(\psi(\widehat P_{kg})\) admits the expansion
\begin{equation}
\label{eq:lal}
\sqrt{N_{kg}}
\left(
\psi(\widehat P_{kg})-\psi(P_{kg})
\right)
=
\frac{1}{\sqrt{N_{kg}}}
\sum_{i=1}^{N_{kg}}
\IF_\psi(X_{kg,i};P_{kg})
+
o_p(1),
\end{equation}
where $\mathbb{E}
\left[
\IF_\psi(X_{kg};P_{kg})
\right]
=
0$ and $\sigma_{\psi,kg}^2
=
\Var
\left(
\IF_\psi(X_{kg};P_{kg})
\right)
\in(0,\infty)$.
\end{assumption}

Assumption \ref{ass:lal} requires that the plug-in estimator of the
fairness metric have a first-order influence function expansion with
finite nonzero asymptotic variance. This result holds for the mean, under-service risk, high quantile, and upper tail performance metrics considered in this research when the distribution lies in the corresponding regularity class, as verified in Section \ref{app:lal-proof}.

We next define two fundamental quantities that will be used in subsequent analysis. The \textit{fairness gap}
\(d_{k,gg'}^\psi\) is the true difference in the fairness metric between groups
\(g\) and \(g'\) under policy \(k\).
\begin{equation}\label{eq:gap}
d_{k,gg'}^\psi = \psi(P_{kg}) - \psi(P_{kg'}).
\end{equation}
The \textit{fairness slack} \(\Delta_{k,gg'}^\psi(\kappa)\) quantifies the 
difference between the maximum allowable fairness disparity \(\kappa\) and the
absolute fairness gap \(|d_{k,gg'}^\psi|\) for policy \(k\) and group pair
\((g,g')\).
\begin{equation}\label{eq:slack}
\Delta_{k,gg'}^\psi(\kappa) = \kappa - |d_{k,gg'}^\psi|.
\end{equation}
Here positive slack means that the corresponding group-pair comparison
satisfies the fairness constraint, while negative slack indicates a
violation. Its magnitude captures the distance to the fairness boundary,
so near-zero slack identifies the comparisons that are hardest to verify.

For a fixed allocation \(\bm{\alpha}\), define the local asymptotic
variance of the estimated group disparity between groups \(g\) and
\(g'\) under policy \(k\) as
\begin{equation}
\label{eq:V-before-prop}
V_{k,gg'}^\psi(\bm{\alpha})
=
\frac{\sigma_{\psi,kg}^2}{\alpha_{kg}}
+
\frac{\sigma_{\psi,kg'}^2}{\alpha_{kg'}},
\end{equation}
whenever \(\alpha_{kg}>0\) and \(\alpha_{kg'}>0\). We use the convention
\[
V_{k,gg'}^\psi(\bm{\alpha})=+\infty
\quad
\text{if }
\alpha_{kg}=0
\text{ or }
\alpha_{kg'}=0.
\]

\begin{proposition}
\label{prop:local}
Fix a policy \(k\) and two groups \(g\) and \(g'\). Suppose Assumption \ref{ass:lal} holds. Then, as
\(T\to\infty\),
\begin{equation}
\label{eq:local-clt}
\sqrt{T}
\Big[
\big(\psi(\widehat P_{kg})-\psi(\widehat P_{kg'})\big)
-
d_{k,gg'}^\psi
\Big]
\Rightarrow
\mathcal N
\left(
0,
V_{k,gg'}^\psi(\bm{\alpha})
\right).
\end{equation}
If the pair is non-boundary, i.e.,
\(\Delta_{k,gg'}^\psi(\kappa)\neq 0\), then the local Gaussian
approximation gives
\begin{equation}
\label{eq:local-rate}
\mathbb P
\left(
\operatorname{sgn}\{\Delta_{k,gg'}^\psi(\kappa)\}
\left[
\kappa-
\left|
\psi(\widehat P_{kg})-\psi(\widehat P_{kg'})
\right|
\right]
<0
\right)
\approx
\exp
\left\{
-
T
\frac{
\left(\Delta_{k,gg'}^\psi(\kappa)\right)^2
}{
2V_{k,gg'}^\psi(\bm{\alpha})
}
\right\}.
\end{equation}
Accordingly, the local verification hardness index is
\begin{equation}
\label{eq:hardness}
H_{k,gg'}^\psi(\kappa;\bm{\alpha})
=
\frac{
V_{k,gg'}^\psi(\bm{\alpha})
}{
\left(\Delta_{k,gg'}^\psi(\kappa)\right)^2
}.
\end{equation}
\end{proposition}

The central limit theorem in Proposition~\ref{prop:local} establishes the local variance but does not, by itself, establish a fixed-slack large deviation exponent. When technical parts later (Sections~\ref{sec:wedge} and~\ref{sec:pro}) use \eqref{eq:local-rate} as an exponential rate, we impose the following local rate condition.

\begin{assumption}
\label{ass:local-rate}
For each nonboundary comparison \((k,g,g')\) entering the rate analysis, define
\[
A_{T,k,gg'}
=
\left\{
\mathbf 1\!\left(
\left|\psi(\widehat P_{kg})-\psi(\widehat P_{kg'})\right|
\le\kappa
\right)
\neq
\mathbf 1\!\left(
\left|\psi(P_{kg})-\psi(P_{kg'})\right|
\le\kappa
\right)
\right\}.
\]
As \(T\to\infty\),
\[
-\frac{1}{T}\log\mathbb P(A_{T,k,gg'})
=
\frac{1}{2H_{k,gg'}^\psi(\kappa;\bm\alpha)}
+o(1),
\]
uniformly for
\(\alpha_{kg},\alpha_{kg'}\ge\underline{\alpha}>0\), where \(T\) is the
total sampling budget and \(H_{k,gg'}^\psi(\kappa;\bm\alpha)\) is the
corresponding local verification hardness index.
\end{assumption}

This condition is exact for Gaussian mean comparisons and serves as a
high-level rate condition for the other metrics. The local hardness
comparisons rely only on Assumption~\ref{ass:lal}.

Proposition \ref{prop:local} converts a distributional fairness
requirement into a local testing problem. For a fixed policy-group pair, the relevant uncertainty is not the uncertainty of either group metric alone, but the uncertainty of their difference. 
The allocation vector enters only through the variance term
\(V_{k,gg'}^\psi(\bm{\alpha})\), while the fairness tolerance enters only through the distance to the boundary
\(\Delta_{k,gg'}^\psi(\kappa)\). This decomposition is useful because it
puts different fairness specifications, different group pairs, and different sampling allocations on a common local scale.

\section{Metric-Based Local Verification Hardness Comparisons}
\label{sec:equiv}

This section compares the fairness specifications considered in this paper under the location-shift model and establishes a key foundation for our analysis that fairness specifications with similar operational costs can have very different local verification hardness indices, which may lead to very different implementation costs. This establishes a setting in which equal operational costs lead to different implementation costs.

\subsection{Translation-Equivariant Metrics}

To facilitate the presentation, we first consider a broad class of fairness metrics that behave uniformly when
service performance outcomes shift by a constant, and refer to them as
\emph{translation-equivariant} metrics. Formally, a distributional metric
\(\psi\) is translation-equivariant if, for any distribution \(P\) and constant
\(a\),
\[
\psi(P+a)=\psi(P)+a,
\]
where \(P+a\) denotes the distribution of \(X+a\) when \(X\sim P\). Common
operational metrics such as the mean, quantiles, and
\(\mathrm{CVaR}\) all fall into this category.

To understand how these metrics compare, consider a standard operational
environment where the outcome for group \(g\) under policy \(k\) follows a location-shift model.
\begin{equation}
\label{eq:location-family}
X_{kg} = \theta_{kg} + Z_k.
\end{equation}
Here, \(\theta_{kg}\) represents the intrinsic group-specific location
parameter, and \(Z_k\) captures the basic stochastic noise of the system
under policy \(k\), which is assumed to have an identical shape across groups.

Our first major result establishes that under this common structural property,
all translation-equivariant metrics with specified fairness tolerances lead to the exact same operational
costs, but impose different local verification hardness indices.

\begin{theorem}
\label{thm:same-frontier}
Suppose the outcome follows the location-shift model in
\eqref{eq:location-family}. Let \(\psi_1\) and \(\psi_2\) be any two
translation-equivariant fairness metrics. Then, the measured fairness gap
between any two groups \(g\) and \(g'\) is exactly the difference in their
location parameters.
\begin{equation}
\label{eq:gap-location}
\psi_j(P_{kg}) - \psi_j(P_{kg'})
=
\theta_{kg} - \theta_{kg'},
\qquad
\text{for } j=1,2.
\end{equation}
Moreover, both metrics induce identical fair policy sets, identical best achievable objective values, and the same price of fairness.
\begin{equation}
\label{eq:same-feasible}
\mathcal{F}_{(\psi_1,\kappa)}
=
\mathcal{F}_{(\psi_2,\kappa)},
\quad
V_{(\psi_1,\kappa)}
=
V_{(\psi_2,\kappa)},
\quad
\text{PoF}_{(\psi_1,\kappa)}
=
\text{PoF}_{(\psi_2,\kappa)}.
\end{equation}
However, if Assumption \ref{ass:lal} holds, the relative local verification hardness index
between the two parities is determined by the ratio of their influence function
variances, \(\sigma_{\psi,kg}^2=\tau_{\psi,k}^2\).
\begin{equation}
\label{eq:ratio-hardness}
\frac{
H_{k,gg'}^{\psi_1}(\kappa;\bm{\alpha})
}{
H_{k,gg'}^{\psi_2}(\kappa;\bm{\alpha})
}
=
\frac{\tau_{\psi_1,k}^2}{\tau_{\psi_2,k}^2}.
\end{equation}
\end{theorem}

Theorem \ref{thm:same-frontier} shows that identical operational costs can be accompanied by substantially different local verification hardness indices. With infinite data, mean performance parity, high quantile performance parity, and upper tail performance parity with the same fairness tolerance are interchangeable from the perspective of the price of fairness in this environment, because they select the same best fair policy and impose the same performance loss. From a finite data perspective, however, they are far from equivalent. Each fairness specification estimates a different feature of the outcome distribution, and those features differ dramatically in estimation variance. A manager who compares fairness rules only by their price of fairness would entirely miss this gap.

To quantify this gap, we next compare the relative local verification hardness of mean performance parity versus high quantile performance parity.

\begin{corollary}
\label{cor:mean-quantile}
Under the conditions of Theorem \ref{thm:same-frontier}, let \(f_k\) denote the
probability density of the system noise \(Z_k\). For a target quantile
\(\beta\), the relative local verification hardness index of high quantile performance parity compared to
mean performance parity is given by
\begin{equation}
\label{eq:ratio-mq}
\frac{
H_{k,gg'}^{Q_\beta}(\kappa;\bm{\alpha})
}{
H_{k,gg'}^{\text{mean}}(\kappa;\bm{\alpha})
}
=
\frac{
\beta(1-\beta)
}{
\text{Var}(Z_k)\left[f_k(q_{k,\beta})\right]^2
},
\end{equation}
where \(q_{k,\beta}\) is the \(\beta\)-quantile of \(Z_k\).
\end{corollary}

Corollary \ref{cor:mean-quantile} identifies the source of relative local verification hardness index for high quantile performance parity compared to mean performance parity under given conditions.
High quantile parity is difficult to verify when the density around the target quantile is extremely low. In
service settings, this means that fairness specifications based on extreme waits, unusually
fast service, or other high quantile experiences may require much larger sampling budgets
than fairness specifications based on averages, even if they produce the same operational
costs.

\subsection{Under-Service Risk Metric and Service Thresholds}
\label{subsec:threshold-placement}

We next turn to under-service risk metrics. They are different from the translation-equivariant
metrics considered above. A shift in service performance does not simply
translate the under-service probability by the same amount. Instead, the
measured disparity depends on where the service threshold \(s\) is placed
relative to the outcome distribution.

We continue to work with the location-shift model
\eqref{eq:location-family}. Let \(F_k\) and \(f_k\) denote the CDF and
density of the noise \(Z_k\). For a minimum service threshold \(s\),
the under-service risk metric is
\[
\psi_s(P_{kg})=\mathbb{P}(X_{kg}\le s).
\]
For a fixed policy \(k\) and group pair \((g,g')\), define $d_{k,gg'}=\theta_{kg}-\theta_{kg'}$ and $\bar\theta_{k,gg'}=\frac{\theta_{kg}+\theta_{kg'}}{2}$. When group differences are small, the under-service risk disparity is
locally governed by the density around \(s-\bar\theta_{k,gg'}\). To
compare thresholds on the same scale, we calibrate the threshold-specific tolerance by
\begin{equation}
\label{eq:kappa-calib}
\kappa(s)
=
f_k(s-\bar\theta_{k,gg'})\,\bar\kappa,
\end{equation}
where \(\bar\kappa>0\) is a fixed allowable location disparity.
This calibration makes the under-service risk constraint locally
equivalent to the same bound \(\bar\kappa\) on the location
gap \(d_{k,gg'}\) for the fixed policy and group pair. Denote $p_k(s)=F_k(s-\bar\theta_{k,gg'})$.

\begin{theorem}
\label{thm:service-threshold}
Suppose the location-shift model \eqref{eq:location-family} holds and
\(f_k\) is differentiable near \(s-\bar\theta_{k,gg'}\). Then, as
\(d_{k,gg'}\to0\),
\begin{equation}
\label{eq:service-first-order}
|\psi_s(P_{kg})-\psi_s(P_{kg'})|
=
f_k(s-\bar\theta_{k,gg'})\,|d_{k,gg'}|
+
o(|d_{k,gg'}|).
\end{equation}
Under Assumption \ref{ass:lal} and the calibration
\eqref{eq:kappa-calib}, the local hardness of under-service risk
verification satisfies
\begin{equation}
\label{eq:service-hardness}
H_{k,gg'}^{\SL(s)}(\kappa(s);\bm{\alpha})
=
\frac{
p_k(s)(1-p_k(s))
}{
f_k(s-\bar\theta_{k,gg'})^2\bar\kappa^2
}
\left(
\frac{1}{\alpha_{kg}}
+
\frac{1}{\alpha_{kg'}}
\right)
+
o(1),
\end{equation}
whenever \(\alpha_{kg}>0\) and \(\alpha_{kg'}>0\). Consequently, for two
thresholds \(s_1\) and \(s_2\) calibrated to the same allowable location disparity
\(\bar\kappa\),
\begin{equation}
\label{eq:service-ratio}
\frac{
H_{k,gg'}^{\SL(s_1)}(\kappa(s_1);\bm{\alpha})
}{
H_{k,gg'}^{\SL(s_2)}(\kappa(s_2);\bm{\alpha})
}
=
\frac{
p_k(s_1)(1-p_k(s_1))/f_k(s_1-\bar\theta_{k,gg'})^2
}{
p_k(s_2)(1-p_k(s_2))/f_k(s_2-\bar\theta_{k,gg'})^2
}
+
o(1).
\end{equation}
\end{theorem}

Theorem \ref{thm:service-threshold} shows that, after an appropriate calibration of the tolerance, under-service risk parities can have the same fair policy set across different service thresholds, so they have the same price of fairness. In particular, it is locally equivalent to limiting the same underlying location gap \(d_{k,gg'}\). However, the statistical difficulty of verifying the fairness specification still depends on where the threshold is placed. Thresholds in low-density regions may lead to large local verification hardness, which depends on the associated event probability and, in turn, can result in high implementation costs. This again shows that fairness specifications with similar operational costs can differ substantially in implementation
difficulty with finite data.

The theorem suggests that sampling budget requirements should also be considered when choosing a fairness specification. When several specifications are substantively appropriate and operationally similar, these requirements provide an additional criterion for service organizations before committing to one.

\section{Implementation Cost Divergence Near Critical Fairness Tolerances}
\label{sec:wedge}

This section provides another instance that operational cost equivalence does not imply implementation cost equivalence. Unlike Section~\ref{sec:equiv}, which compares only local verification hardness indices, this section directly compares implementation costs for fairness specifications that have the same operational cost. We also show that local fairness verification hardness creates an implementation gap between the critical fairness tolerance boundary and the tightest fairness tolerance that can be credibly supported with finite data. A finite sampling budget can thus prevent a service organization from implementing fairness specifications that would be attainable under a larger sampling budget.

\subsection{Budget Requirements Near Critical Fairness Tolerances}

We now characterize the implementation gap more precisely by considering fairness tolerances near a critical boundary. At such a boundary, the identity of the best fair policy may change. Reliable fair policy selection then requires not only accurate policy ranking, but also verification of fairness comparisons with slack close to zero.

A fairness tolerance level \(\kappa_0\) is called \emph{critical fairness tolerance} if
the identity of the best fair policy can change at \(\kappa_0\). This can occur when the current best fair policy becomes infeasible as the fairness tolerance is tightened, or when a higher-performing unfair policy becomes feasible as the fairness tolerance is relaxed. Thus, the fair policy selection problem can behave differently on the two sides of \(\kappa_0\).

A fairness comparison, denoted by \((k,g,g')\), is active if its verification error decay rate governs the overall decay rate of \(\PFS^{\pi}_{\mathfrak f}(T)\), where the verification error refers to both the case where the corresponding absolute fairness gap \(|d_{k,gg'}^\psi|\) is estimated to be below \(\kappa\) while the true gap exceeds \(\kappa\), and the case where \(|d_{k,gg'}^\psi|\) is estimated to be above \(\kappa\) while the true gap falls below \(\kappa\).

For a fairness specification \(\mathfrak f\),
let \(\Phi^\star_{\mathfrak f}\) denote the optimal local rate. This quantity summarizes, under an efficient allocation of the sampling budget, the best achievable exponential decay rate of \(\PFS^{\pi}_{\mathfrak f}(T)\) given that \(\kappa\) is sufficiently close to \(\kappa_0\). Near a critical fairness tolerance, the difficulty of fair policy selection is governed by the fairness comparison whose corresponding fairness slack is close to zero, making this fairness comparison active. This result is a straightforward consequence of the results in Section~\ref{sec:pro}. In fact, as will be shown in Section \ref{proof:wedge}, the absolute fairness gap of the active fairness comparison equals the critical fairness tolerance. Such comparisons have optimal error decay rates that depend on the squared distance between \(\kappa\) and \(\kappa_0\), a property specified in the approximations presented below. We write the local expansion of \(\Phi^\star_{\mathfrak f}\) as
\[
\Phi^\star_{\mathfrak f}
=
C_{(\psi,\kappa_0)}(\kappa-\kappa_0)^2
+
o\!\left((\kappa-\kappa_0)^2\right),
\]
where \(C_{(\psi,\kappa_0)}\) captures the local statistical complexity of
the active fairness comparison. This constant depends on the fairness metric, the service performance distributions of the relevant policy-group pairs, and the allocation proportions of samples across these pairs.

For a \(\PCS_{\mathfrak f}(T)\) target \(1-\delta\), recall that the implementation cost is
\[
B_{\mathfrak f}(\delta)
=
\inf
\left\{
T\in\mathbb R_+\mid
\PCS_{\mathfrak f}(T)\ge 1-\delta
\right\}.
\]
Given a fixed budget \(T\), define the tightest statistically implementable tolerance above the critical fairness tolerance as
\[
\kappa_{\mathrm{impl},\mathfrak f}(T,\delta)
=
\inf
\left\{
\kappa>\kappa_0\mid
B_{\mathfrak f}(\delta)\le T
\right\}.
\]
The corresponding lower-side definition is analogous.

The following result translates the optimal local rate into the budget needed to implement a fairness tolerance with a specified PCS level, in the regime where the fairness tolerance is near the critical fairness tolerance.

\begin{theorem}
\label{thm:wedge}
Consider a fairness metric \(\psi\) and a tolerance \(\kappa\) near a critical fairness
tolerance \(\kappa_0\),
\[
\Phi^\star_{\mathfrak f}
=
C_{(\psi,\kappa_0)}(\kappa-\kappa_0)^2
+
o\!\left((\kappa-\kappa_0)^2\right),
\qquad C_{(\psi,\kappa_0)}>0,
\]
and \(1-\PCS_{\mathfrak f}(T)\) attains this optimal decay rate $\Phi^\star_{\mathfrak f}$.
\[
1-\PCS_{\mathfrak f}(T)
=
\exp\{-T\Phi^\star_{\mathfrak f}+o(T)\}.
\]
Then, for \(\kappa\) sufficiently close to \(\kappa_0\) on the relevant side,
\[
B_{\mathfrak f}(\delta)
=
\frac{
\log(1/\delta)
}{
C_{(\psi,\kappa_0)}(\kappa-\kappa_0)^2
}
\left(1+o(1)\right).
\]
Consequently, if \(\log(1/\delta)/T\) is sufficiently close to zero, then
\[
\kappa_{\mathrm{impl},\mathfrak f}(T,\delta)-\kappa_0
=
\sqrt{
\frac{
\log(1/\delta)
}{
C_{(\psi,\kappa_0)}T
}
}
\left(1+o(1)\right).
\]
\end{theorem}

Theorem~\ref{thm:wedge} shows that implementing a fairness tolerance close to a critical fairness tolerance boundary is statistically costly. The budget required to implement a tolerance \(\kappa\) with confidence scales with the inverse-square rate
\((\kappa-\kappa_0)^{-2}\). 
Equivalently, under a fixed budget \(T\), the implementable tolerance must remain separated from the critical tolerance by an implementation gap of order \(\sqrt{\log(1/\delta)/T}\). 
This implementation gap reflects the difficulty of verifying whether a policy corresponding to the active fairness comparison satisfies the fairness constraint under a finite budget. The constant \(C_{(\psi,\kappa_0)}\) also controls the size of this gap. A larger \(C_{(\psi,\kappa_0)}\) means that the active fairness comparison
provides a stronger statistical signal as \(\kappa\) moves away from
\(\kappa_0\), so fewer samples are needed to verify feasibility with the
same PCS level. A smaller \(C_{(\psi,\kappa_0)}\) means that the
active fairness comparison is harder to verify. In that case, one must
either increase the sampling budget or use a more conservative tolerance.

For a given fairness metric, consider two tolerances \(\kappa_1\) and \(\kappa_2\) 
that are both greater than a critical tolerance \(\kappa_0\) or both less than 
\(\kappa_0\), and sufficiently close to \(\kappa_0\). They define two fairness specifications 
\(\mathfrak{f}_1\) and \(\mathfrak{f}_2\) that have the same operational cost. 
Theorem~\ref{thm:wedge} shows that their implementation costs 
\(B_{\mathfrak{f}_1}(\delta)\) and \(B_{\mathfrak{f}_2}(\delta)\) differ, 
illustrating an implementation cost divergence, which demonstrates that operational equivalence is not equivalent to implementation equivalence.

\subsection{Comparing Budget Requirements Across Fairness Specifications}

The last subsection gives a local implementation gap result for a
single fairness specification \(\mathfrak f\). We now specialize this perspective to compare the fairness specifications considered in this research. The purpose
is to identify how the
constant \(C_{(\psi,\kappa_0)}\) changes across metrics and, therefore, how
different fairness specifications translate into different finite-budget requirements.

For a fairness metric $\psi$, the associated constant $C_{(\psi,\kappa_0)}$ is determined by policy-group-level structural factors. A low-variance influence function for the metric \(\psi\) at the policy-group distributions in the active comparison tends to give a larger $C_{(\psi,\kappa_0)}$. By contrast, a high-variance influence function tends to give smaller $C_{(\psi,\kappa_0)}$ values. This feature can be directly observed from the closed-form expression of $C_{(\psi,\kappa_0)}$ derived in Section \ref{proof:wedge}.

The following corollary summarizes the resulting finite-budget implication.

\begin{corollary}
\label{cor:separation}
Suppose two fairness metrics \(\psi_1\) and
\(\psi_2\) have the same fair policy set near critical tolerance \(\kappa_0\). According to our definition of critical tolerance, for two fairness tolerances \(\kappa_1\) and \(\kappa_2\) that are both greater than \(\kappa_0\) or both less than \(\kappa_0\), and sufficiently close to \(\kappa_0\), we can conclude that the two fairness specifications \((\psi_1, \kappa_1)\) and \((\psi_2, \kappa_2)\) also have the same fair policy set. Suppose the local statistical complexity constants satisfy $C_{(\psi_1, \kappa_0)}>C_{(\psi_2, \kappa_0)}>0$. Then, if \(\log(1/\delta)/T\) is sufficiently close to zero,
there exists a nonempty interval of tolerances
\[
I(T,\delta)
=
\left[
\kappa_0+
\sqrt{
\frac{\log(1/\delta)}{C_{(\psi_1, \kappa_0)}T}
},
\,
\kappa_0+
\sqrt{
\frac{\log(1/\delta)}{C_{(\psi_2, \kappa_0)}T}
}
\right)
\]
such that, for any \(\kappa_1,\kappa _2\in I(T,\delta)\),
\[
B_{(\psi_1, \kappa_1)}(\delta)\le T(1+o(1))
<
B_{(\psi_2, \kappa_2)}(\delta).
\]
The result also holds when \(\kappa_1 = \kappa_2\).
\end{corollary}

Corollary \ref{cor:separation} shows that finite budget can separate
fairness specifications that are indistinguishable at operational costs. Both
rules may recommend the same fair policy set and the same
price of fairness, but only the rule with the larger complexity
constant $C_{(\psi, \kappa_0)}$ can be verified at the desired PCS within the given
budget. This also demonstrates that operational equivalence is not equivalent to implementation equivalence. In this sense, the sampling budget can change the practical
choice of fairness metrics. A parity that is theoretically attractive
but statistically fragile may be infeasible to implement, while another
parity with the same operational
costs may be usable because its variance of the active fairness comparison can be reduced more easily.

For example, a hospital may find that mean waiting time parity and high quantile waiting time parity recommend nearly the same triage policy in a calibrated simulation model. If the fairness specification based on the high quantile waiting time metric has a smaller local statistical complexity constant $C_{(\psi, \kappa_0)}$, however, the same sampling budget may verify the mean waiting time parity rule but not the high quantile one. The two parities are operationally similar, but only one can be implemented while maintaining the desired PCS. Fairness specifications should therefore not be compared only by their
operational meaning or by the price of fairness they induce. They should also be compared in terms of implementation under a realistic sampling budget.

\section{Fairness-Guided Budget Allocation}
\label{sec:pro}
This section develops the fairness-guided budget allocation rule. The fairness-guided allocation rule developed in this section serves as an operational mechanism within the fairness framework. Fair policy selection contains objective ranking comparisons and cross-group feasibility comparisons, so the useful value of an additional sample depends on both. We first decompose false selection into ranking and verification events. We then construct
tractable lower bounds on the corresponding error exponents and use the
minimum of these bounds as an allocation criterion. The resulting criterion supplies the target proportions used by the adaptive algorithm.

Since \(\PFS^{\pi}_{\mathfrak f}(T)=1-\PCS^{\pi}_{\mathfrak f}(T)\), under a fixed allocation, the goal is to make \(\PFS^{\pi}_{\mathfrak f}(T)\) decay as quickly as possible as the total budget \(T\) increases. 
We therefore use the exponential decay rate of \(\PFS^{\pi}_{\mathfrak f}(T)\) as the optimization objective.

Let \(k^\star=k^\star_{\mathfrak f}\). The policy selection can fail in two basic ways. First, the true best fair policy may be incorrectly rejected as unfair (fairness verification error event).
Second, some competing policy may be empirically classified as fair and appear to perform as well as or better than \(k^\star\) (objective ranking
error event). Define
\[
E_0
=
\left\{
k^\star \notin \widehat{\mathcal F}_{\mathfrak f}
\right\}
\]
to denote the fairness verification error event. For each \(k\neq k^\star\), define
\[
E_k
=
\left\{
k^\star \in \widehat{\mathcal F}_{\mathfrak f},\;
k \in \widehat{\mathcal F}_{\mathfrak f},\;
\widehat \mu_k \ge \widehat \mu_{k^\star}
\right\}
\]
to denote the objective ranking error event, which is formed by the intersection of three conditions.
Then
\[
\PFS^{\pi}_{\mathfrak f}(T)
=
\mathbb P
\left(
E_0
\cup
\bigcup_{k\neq k^\star} E_k
\right).
\]

\begin{lemma}
\label{lemma2.2}
The following rate result holds.
\begin{equation}
\label{eq:decay}
\lim_{T\to\infty}
-\frac{1}{T}\log \PFS^{\pi}_{\mathfrak f}(T)
=
\min
\left\{
-\lim_{T\to\infty}\frac{1}{T}\log \mathbb P(E_0),
\;
\min_{k\neq k^\star}
\left(
-\lim_{T\to\infty}\frac{1}{T}\log \mathbb P(E_k)
\right)
\right\},
\end{equation}
when the component exponential decay rates $-\lim_{T\to\infty}\frac{1}{T}\log \mathbb P(E_k)$ exist.
\end{lemma}

Equation \eqref{eq:decay} shows that the budget allocation problem must account
for both types of statistical error decay rates corresponding to \(E_0\) and \(E_k\).
Since false selection occurs whenever at least one fairness verification error event or objective ranking error event occurs, the overall decay
rate is governed by the slowest decaying error event. 
Therefore, improving fair policy selection requires allocating the sampling budget to
increase the smallest of these error decay rates.

\subsection{Fairness-Guided Allocation Rule}

Based on Lemma \ref{lemma2.2}, we approximate the decay rates of the two basic errors. 
These approximations will then be combined into the overall \(\PFS_{\mathfrak f}^{\pi}(T)\) decay rate through \eqref{eq:decay} for constructing the allocation rule.
The error event \(E_k\) has different structures depending on the
true feasibility and objective value of policy \(k\). To capture
them, we partition the set of \(K\) policies into four mutually exclusive and
collectively exhaustive subsets. Let
\[
k^\star_{\mathfrak f}
=
\arg\max_k
\left\{
\mu_k\mid\varepsilon_k^{\mathfrak f}\le \kappa
\right\}
\]
denote the unique best fair policy. The remaining policies are divided into
the following three classes.
\begin{align*}
&\Gamma
=
\left\{
k\mid
\varepsilon_k^{\mathfrak f}\le \kappa,\ 
k\neq k^\star_{\mathfrak f}
\right\},\\
&\mathcal{S}_b
=
\left\{
k\mid
\mu_{k^\star_{\mathfrak f}}\leq\mu_k
\text{ and }
\varepsilon_k^{\mathfrak f}>\kappa
\right\},\\
&\mathcal{S}_w
=
\left\{
k\mid
\mu_{k^\star_{\mathfrak f}}> \mu_k
\text{ and }
\varepsilon_k^{\mathfrak f}>\kappa
\right\}.
\end{align*}
Policies in \(\Gamma\) are fair but worse than
\(k^\star_{\mathfrak f}\). Policies in \(\mathcal{S}_b\) have objective values no worse than \(k^\star_{\mathfrak f}\) but violate the fairness specification. Policies in
\(\mathcal{S}_w\) are unfair and worse than
\(k^\star_{\mathfrak f}\). We make this partition because in this way, each class
can be falsely selected for a different reason. A policy in \(\Gamma\) is already fair, so it can be
mistakenly selected only if it appears to perform as well as or better than the best fair policy.
A policy in \(\mathcal{S}_b\) already has an objective value no worse, so false selection occurs if it is mistakenly classified as fair. A policy in
\(\mathcal{S}_w\) is unfair and worse than the best fair policy, so it
can be falsely selected only if it is both mistakenly classified as fair and
incorrectly ranked as well as or better than the best fair policy.

We use the following sets to distinguish fairness comparisons that are
satisfied or violated by policy \(k\).
\begin{align*}
\mathcal{C}_F^k
&=
\left\{
(g,g')\mid
|\psi(P_{kg})-\psi(P_{kg'})|\leq \kappa,g<g'
\right\},\\
\mathcal{C}_I^k
&=
\left\{
(g,g')\mid
|\psi(P_{kg})-\psi(P_{kg'})|>\kappa,g<g'
\right\}.
\end{align*}

\begin{theorem}
\label{thm1}
Let \(k^\star=k^\star_{\mathfrak f}\). Under Assumptions \ref{ass:lal} and \ref{ass:local-rate}, the decay rate of fairness verification error \(E_0\) is
\begin{equation}
\label{eq:kstar-rate}
-\lim_{T\to\infty}
\frac{1}{T}
\log
\mathbb{P}(E_0)
=
R_{k^\star}(\bm{\alpha})
=
\min_{g<g'}
\frac{
1
}{
2H_{k^\star,gg'}^\psi(\kappa;\bm{\alpha})
}.
\end{equation}
\end{theorem}

Theorem \ref{thm1} shows that the fairness verification for the best fair policy is only as
reliable as its most difficult fairness comparison. Even if most
group-pair fairness constraints are easy to verify, a single
near-binding or high-variance comparison can dominate the probability
that the true best fair policy is incorrectly discarded. The allocation vector
\(\bm{\alpha}\) therefore affects the \(\PFS_{\mathfrak f}^{\pi}(T)\) through the biggest verification hardness index of the best fair policy.

Let $\mathcal G_+ = \{g\in\mathcal G \mid w_g>0\}$ be groups with positive objective weights and $I_R(m)$ be the Cramér transform of objective output distribution $R$. The standard Cramér regularity condition requires the cumulant generating function to be finite in an open neighborhood of zero \citep{dembo2010large}. 

\begin{theorem}
\label{thm:mean-rate1}
Let \(k^\star=k^\star_{\mathfrak f}\), and consider a policy
\(k\) satisfying \(\mu_k<\mu_{k^\star}\). Under the Cramér regularity condition above, the
decay rate of event $\widehat{\mu}_k\geq \widehat{\mu}_{k^\star}$
is
\begin{equation}
\label{eq:mean-rate1}
J_{k,k^\star}(0;\boldsymbol{\alpha})
=
\inf_{\substack{
\{m_{kg}\}_{g\in\mathcal G_+},
\{m_{k^\star g}\}_{g\in\mathcal G_+}
\\
\sum_{g\in\mathcal G_+}w_g m_{kg}
=
\sum_{g\in\mathcal G_+}w_g m_{k^\star g}
}}
\left(
\sum_{g\in\mathcal G_+}
\alpha_{kg} I_{R_{kg}}(m_{kg})
+
\sum_{g\in\mathcal G_+}
\alpha_{k^\star g}
I_{R_{k^\star g}}(m_{k^\star g})
\right).
\end{equation}
Hence, for \(k\neq k^\star\), the decay rate of objective ranking error \(E_k\) satisfies
\[
-\lim_{T\to\infty}
\frac{1}{T}
\log
\mathbb{P}(E_k)\ge
R_k(\bm{\alpha})
=
\begin{cases}
\displaystyle
J_{k,k^\star}(0;\bm{\alpha}),
&
k\in\Gamma,
\\[8pt]
\displaystyle
\max_{(g,g')\in\mathcal C_I^k}
\frac{
1
}{
2H_{k,gg'}^\psi(\kappa;\bm{\alpha})
},
&
k\in\mathcal S_b,
\\[10pt]
\displaystyle
\max\left\{
\max_{(g,g')\in\mathcal C_I^k}
\frac{
1
}{
2H_{k,gg'}^\psi(\kappa;\bm{\alpha})
},
\;
J_{k,k^\star}(0;\bm{\alpha})
\right\},
&
k\in\mathcal S_w.
\end{cases}
\]
\end{theorem}

Theorem \ref{thm:mean-rate1} gives, under \(\mu_k<\mu_{k^\star}\), the decay rate of event \(\widehat{\mu}_k\geq \widehat{\mu}_{k^\star}\) and establishes a lower bound on the error decay rate for the objective ranking error event \(E_k\) across different classes of suboptimal or unfair policies. It shows that the confidence of discarding a suboptimal or unfair policy \(k\) is at least as strong as the largest decay rate among the three conditions defining \(E_k\), as the easiest condition to detect as false dominates the decay rate of overall probability of the joint rare event. The allocation vector \(\bm{\alpha}\) therefore affects the overall selection confidence through both the objective estimation precision and the fairness verification for each policy.

We now combine the decay rates of the fairness verification error event and the objective ranking error events to characterize the error decay rate of \(\PFS_{\mathfrak f}^{\pi}(T)\). The error-specific decay rate approximation \(R_k(\bm{\alpha})\) allows us to formulate a tractable allocation criterion. By Lemma \ref{lemma2.2}, the overall \(\PFS^{\pi}_{\mathfrak f}(T)\) decay rate equals the smallest decay rate of the error events. To maximize a lower bound on the overall decay rate of \(\PFS^{\pi}_{\mathfrak f}(T)\), we solve a max-min optimization problem over allocation vectors, which is formalized in the theorem below.

\begin{theorem}
\label{thm:rate}
Under Assumptions \ref{ass:lal} and \ref{ass:local-rate} and the Cramér regularity condition above,
\begin{equation}
\label{eq:final}
\lim_{T\to\infty}
-\frac{1}{T}\log \PFS^{\pi}_{\mathfrak f}(T)\ge\min_{k\in\mathcal K}
R_k(\bm{\alpha}),
\end{equation}
which provides the following tractable
fairness-guided allocation rule.
\begin{equation}
\label{eq:optimization}
\max_{\bm{\alpha}\in\Delta}
\min_{k\in\mathcal K}
R_k(\bm{\alpha}),
\qquad
\Delta=
\left\{
\bm{\alpha}\in\mathbb R_+^{K\times G}\mid
\sum_{k,g}\alpha_{kg}=1
\right\}.
\end{equation}
\end{theorem}

Theorem \ref{thm:rate} gives a fairness-guided allocation rule maximizing the minimum of these error-specific decay rate lower bounds, which forms the theoretical foundation of the sequential fairness-guided sampling algorithm detailed in the next subsection. The resulting allocation rule balances error decay rates across all four policy classes. It ensures reliable verification of the best fair policy, correct elimination of unfair policies with no worse objective value, accurate ranking of inferior fair alternatives, and reduction of joint errors for unfair policies with worse objective values.

\subsection{An Adaptive Allocation Algorithm}

The allocation rule in Theorem \ref{thm:rate} depends on quantities that are unknown prior to sampling, so we operationalize it through an adaptive allocation algorithm. We begin with an initial stage, where a fixed number of samples is allocated to every policy-group pair to get preliminary estimations of all objective and fairness-related quantities. These initial estimations are used to construct the empirical fair policy set, identify a provisional estimated best fair policy, and compute the proxy error-specific decay rate approximations \(\widehat R_k(\bm{\alpha})\). In each subsequent allocation iteration, we solve the max-min problem to allocate additional sampling budget to the policy-group pairs. A minimum share constraint is enforced to avoid prematurely discarding policy-group pairs that may become binding as more samples are allocated.

\begin{algorithm}[t]
\caption{Fairness-Guided Allocation (FGA) Algorithm}
\label{alg:two-stage}
\begin{algorithmic}[1]
\Require Total budget \(T\), per-round budget \(T_0\), initial sample size
\(n_0\ge 10\), minimum share \(0\le \alpha_{\min}\le 1/(KG)\), fairness metric \(\psi\),
tolerance parameter \(\kappa\); assume \(T=KGn_0+mT_0\) for some positive
integer \(m\); denote \(\Delta_{min} = \{\bm{\alpha}\in\mathbb{R}_+^{K\times G}\mid \sum_{k,g}\alpha_{kg}=1,\alpha_{kg}\geq \alpha_{min}\}\).
\State \textbf{Initial stage.} sample each \((k,g)\) pair \(n_0\) times, and set
\(N_{kg}\gets n_0\) for all \(k,g\).
\State Estimate \(\widehat\mu_k\), fairness gaps
\(\widehat d_{k,gg'}^\psi\), local slacks
\(\widehat\Delta_{k,gg'}^\psi(\kappa)\), distribution \(\widehat P_{kg}\) and local variances
\(\widehat\sigma_{\psi,kg}^2\) from the initial samples.
\State Construct the empirical fair policy set
\[
\widehat{\mathcal{F}}_{\mathfrak f}
=
\left\{
k\mid
\max_{g<g'}
\left|
\psi(\widehat P_{kg})-\psi(\widehat P_{kg'})
\right|
\le \kappa
\right\}.
\]
\State If $\widehat{\mathcal F}_{\mathfrak f}$ is nonempty, identify an estimated best policy $\widehat{k}_{\mathfrak f}
\in
\arg\max_{k\in\widehat{\mathcal{F}}_{\mathfrak f}}
\widehat\mu_k$; otherwise apply the minimum violation fallback rule from Section~\ref{sec:model}.
\State \textbf{For} round \(r=1,\dots,m\) \textbf{do}
\State Build proxy error-specific decay rate approximations
\(\widehat R_k(\bm{\alpha})\) for each policy \(k\) using
\(\widehat{k}_{\mathfrak f}\), \(\widehat\mu_k\),
\(\widehat d_{k,gg'}^\psi\),
\(\widehat\Delta_{k,gg'}^\psi(\kappa)\), \(\widehat P_{kg}\) and
\(\widehat\sigma_{\psi,kg}^2\).
\State Solve the max-min allocation problem.
\[
\widehat{\bm{\alpha}}
\in
\arg\max_{\bm{\alpha}\in\Delta_{\min}}
\min_k
\widehat R_k(\bm{\alpha}).
\]
\State Allocate \(T_0\) samples according to proportion
\(\widehat{\bm{\alpha}}\), and update the counters \(N_{kg}\).
\State Recompute \(\widehat\mu_k\),
\(\widehat d_{k,gg'}^\psi\),
\(\widehat\Delta_{k,gg'}^\psi(\kappa)\),
\(\widehat\sigma_{\psi,kg}^2\), the empirical fair policy set
\(\widehat{\mathcal{F}}_{\mathfrak f}\), and the estimated best policy
\(\widehat{k}_{\mathfrak f}\) using all samples observed so far.
\State \textbf{end for}
\State \Return the final estimated best policy
\(\widehat k_{\mathfrak f,T}\).
\end{algorithmic}
\end{algorithm}

In implementation, we solve the empirical max-min problem with a constrained multistart numerical optimizer. Mean, threshold, quantile, and CVaR influence function variances are estimated from the current samples using the formulas in the appendix. The target proportions are converted to integer batch counts by largest remainder rounding, which preserves the batch total. The minimum share parameter is set below \(1/(KG)\) and acts as an exploration safeguard.

\section{Numerical Experiments}
\label{sec:numerics}

In this section, we conduct numerical experiments to illustrate the costs of fairness and to evaluate the performance of the
proposed FGA algorithm. We consider two synthetic experiments and two
realistic service system experiments. The first synthetic experiment compares the costs of fairness and algorithmic performance across different fairness specifications, and
the second examines policy selection near a fairness tolerance boundary. The call
center and emergency department experiments evaluate the proposed algorithm in
realistic service settings. The numerical experiments have two purposes. First, they illustrate how
the fairness metric and tolerance interact with a finite sampling budget to
determine implementation difficulty. Second, they evaluate FGA as an operational mechanism for using that budget efficiently.

For comparison, we adopt several representative
R\&S procedures.
\begin{itemize}
\item \textbf{Equal Allocation (EA).}
EA uniformly allocates the total sampling budget across all
policy-group pairs. It is a naive allocation method and serves as a basic benchmark against which improvements might be measured.
\item \textbf{Objective-Only OCBA (OCBA-O).}
OCBA-O is adapted from the classical optimal computing budget allocation
framework for unconstrained R\&S \citep{chen2000simulation}. It allocates
samples based on the objective value. We extend it to our multi-group policy setting by aggregating group-level noise into a policy-level objective variance to retain the standard OCBA rule, and further split each policy's budget across groups proportionally to group-weighted estimation uncertainty.
\item \textbf{Constrained OCBA (OCBA-C).}
OCBA-C is adapted from an algorithm proposed for constrained R\&S
\citep{lee2012approximate}. It incorporates feasibility estimated from simulation, but does not
explicitly account for the policy-group structure or the
metric-specific verification difficulty of fairness constraints. We tailor its constrained allocation logic to the multi-group policy setting by constructing a two-dimensional priority score for objective ranking and fairness verification, and decomposing policy-level priorities into group-level allocation proportions. 
\item \textbf{FGA.}
FGA is our proposed fairness-guided allocation algorithm, which allocates
sampling iteratively based on the result of Theorem \ref{thm:rate}.
\end{itemize}

In
the initial stage, each policy-group pair of the compared algorithms receives \(n_0\) pilot samples,
which are used to estimate unknown quantities. In the subsequent
allocation iterations, the remaining budget is allocated according to
each algorithm's allocation rule. The algorithm performance is evaluated using three measures, with results estimated
from \(M=500\) independent macro replications for each total sampling
budget \(T\).

First, we report the empirical PCS, which is the proportion of macro replications in which the algorithm selects the true best fair policy, i.e.,
\(\widehat k_{\mathfrak f,T}=k_\mathfrak f^\star\). This is our target measure and directly reflects the goal of fair policy selection. In the experimental results, the term PCS generally refers to the empirical PCS.
Second, we report the probability of incorrect fairness verification $\mathrm{PIF}
=
\mathbb{P}\left(
\varepsilon_{\widehat k_{\mathfrak f,T}}^{\mathfrak f}
>
\kappa
\right)$, which measures the probability that the selected policy violates at
least one fairness constraint. This measure captures whether the final
selection is truly fair. Third, we assess sample efficiency through the algorithm-specific empirical implementation cost defined after \eqref{eq:burden}. For each algorithm, we record the total sampling budget at which its empirical PCS first reaches the target
level \(1-\delta\), with \(\delta=0.05\) set as the default.

Additional details on the numerical experiment setup and implementation of the compared algorithms are provided in Section \ref{sec: numerical details}.

\subsection{Synthetic Experiments}

The first synthetic experiment examines, under different fairness specifications, the costs of fairness and the corresponding algorithm performance. We consider \(K=10\) candidate policies and \(G=2\) customer groups. For policy \(k\) and group \(g\), the outcome \(W_{kg}\), representing the waiting time of a group-\(g\) customer under policy \(k\), is given by
$
    W_{kg}=\theta_{kg}+Z_k .
$

The objective value of policy \(k\) is the negative average waiting time across groups, adjusted by a small policy-specific operating cost $c_k$.
\[
    \mu_k
    =
    -
    \frac{1}{2}
    \left(
    \mathbb{E}[W_{k1}]
    +
    \mathbb{E}[W_{k2}]
    \right)
    -
    c_k .
\]
Thus, larger \(\mu_k\) indicates better aggregate performance. In the joint output notation of Section~\ref{sec:model}, one may set \(Y_{kg}=-W_{kg}-c_k\) and \(X_{kg}=W_{kg}\); the absolute disparity constraint is unchanged by a sign transformation.
We compare two noise regimes to examine how distributional shape affects implementation cost. In the light-tailed regime,
$
    Z_k\sim N(0,\sigma_k^2).
$
In the heavy-tailed regime,
\[
    Z_k
    \sim
    \mathrm{Lognormal}(m_k,s_k^2)
    -
    \mathbb{E}\!\left[
    \mathrm{Lognormal}(m_k,s_k^2)
    \right].
\]
The variance parameters are calibrated so that the two regimes have comparable overall variability but different tail behavior, which is especially relevant for high quantile performance and upper tail performance-based fairness metrics.

We compare three fairness metrics that target different regions of the waiting time distribution. Mean waiting time parity controls disparities in average waiting times, 90th quantile waiting time parity controls disparities in high waiting time outcomes, and upper tail waiting time parity using CVaR controls disparities in the expected waiting time among the worst 10\% of outcomes. These metrics allow us to examine how fairness constraints based on mean, high quantile, and upper tail performance affect the costs of fairness for policy selection.

Each algorithm is evaluated with \(n_0=30\) initial samples per policy-group pair. In the main comparison, the fairness tolerance \(\kappa\) is selected so that identifying the best fair policy remains nontrivial under all three metrics. We also vary \(\kappa\) to trace the price of fairness for each metric.
Since this experiment focuses on the costs induced by different fairness metrics, rather than on the full algorithmic
comparison, we include EA, OCBA-O, and FGA.

\begin{figure}[t]
    \centering
    \includegraphics[width=0.85\linewidth]{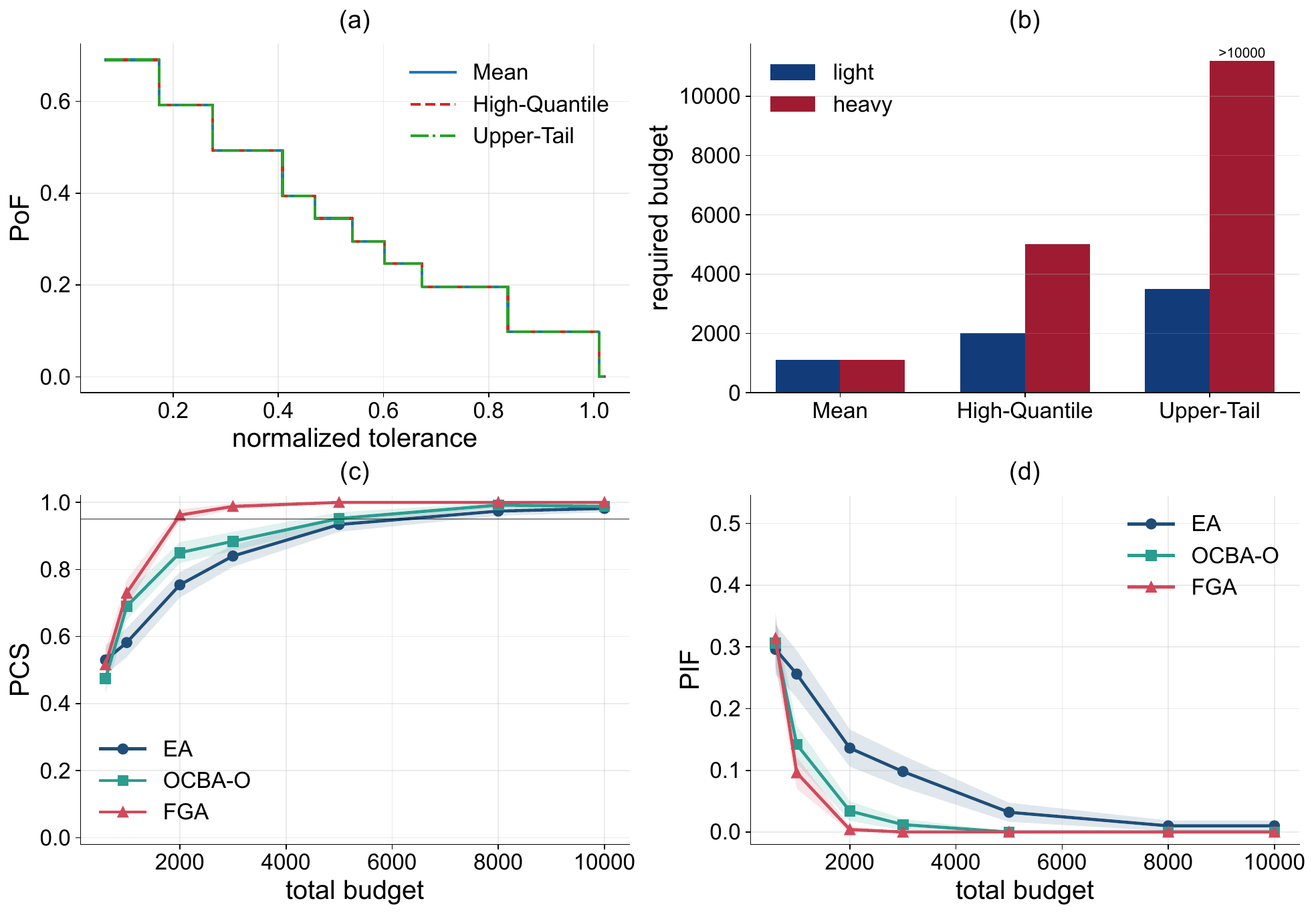}
    \caption{Synthetic Experiment I Results.
    (a) Price of fairness for three fairness specifications under the heavy-tailed noise regime.
    (b) Minimum required total budget for FGA to achieve target PCS 0.95 for each
     parity, comparing light-tailed and heavy-tailed regimes.
    (c) PCS under the 90th quantile parity constraint in
    the light-tailed regime.
    (d) PIF under the same 90th quantile parity constraint in
    the light-tailed regime.}
    \label{fig:controlled1}
\end{figure}

Figure~\ref{fig:controlled1} summarizes the comparison results. 
Panel~(a) shows that mean waiting time parity, 90th quantile waiting time parity, and upper tail waiting time parity lead to the same price-of-fairness frontier. Thus, with infinite data, the three metrics with given fairness tolerances impose the same population-level operational cost.
Panel~(b) shows that mean parity requires the smallest budget to attain the target PCS for FGA, while tail-sensitive parities require much larger budgets. The gap is most pronounced under heavy-tailed noise, where upper tail waiting time parity requires more than three times the budget of mean waiting time parity, and 90th quantile waiting time parity also incurs a substantially larger cost. The crossing budget reported for FGA is its algorithm-specific implementation cost, \(B_{\mathfrak f}^{\mathrm{FGA}}(0.05)\). It provides a common implementation-oriented proxy across metrics without asserting that FGA attains the theoretical optimum. The observed budget differences therefore show that the three specifications impose substantially different evidence requirements under the same allocation procedure.

These results demonstrate that operational cost equivalence in the asymptotic limit does not translate into implementation cost equivalence in practice. The substantial budget discrepancies underscore the importance of explicitly accounting for verification difficulty when selecting a fairness metric, especially under heavy-tailed distributions. Ignoring these finite-sample challenges risks fair policy selections failing to meet the target PCS level.

Panels~(c) and~(d) focus on 90th quantile parity under light-tailed noise. Panel~(c) shows that FGA achieves higher PCS than the benchmarks throughout the tested budget range. Panel~(d) shows that FGA also reduces the PIF more rapidly.

Figure~\ref{fig:cvar_results} in Section~\ref{sec: first synthetic} shows the comparison results for upper-tail waiting-time parity under both light-tailed and heavy-tailed regimes, which illustrates that tail-sensitive fairness metrics are particularly difficult to verify under heavy-tailed outcome distributions.

The second synthetic experiment examines a setting near a critical fairness tolerance boundary, where the best fair policy is hard to distinguish because it lies close to the boundary and is difficult to distinguish from high-performing but unfair alternatives. This setting is motivated by the critical fairness tolerance discussed in Section \ref{sec:wedge}.

We consider a selection problem with \(K=12\) candidate policies and \(G=2\)
groups. With fairness tolerance \(\kappa=0.1025\) and \(\kappa=0.098\), policies 1-3 have the highest
objective values but are unfair. Policy 4 is the true best fair policy,
lying just below the fairness boundary. Policies 5-12 are fair but have
lower objective values. 

For each policy-group pair, outcomes are generated as
$
    X_{kg}
    \sim
    N(m_{kg},\sigma_{kg}^2).
$
The group-specific means are calibrated so that $(m_{k1}+m_{k2})/2
    =
    \mu_k$ and $ |m_{k1}-m_{k2}|
    =
    D_k$.
Thus, \(\mu_k\) controls the average performance of policy \(k\), while \(D_k\)
controls its group-level fairness disparity. The setting uses a common
variance \(\sigma_{kg}^2=0.2^2\). The fairness specification is mean performance parity between the two groups.

We evaluate each algorithm with \(n_0=20\) initial samples per policy-group
pair. To identify the sources of performance differences, we decompose
selection errors into three categories, which may overlap because a single
false selection can reflect multiple failure events. \textbf{Type I errors}
select a better but unfair policy (policies 1-3), reflecting failure to
verify the fairness constraints; \textbf{Type II errors} classify the true
best fair policy (policy 4) as unfair and discard it, reflecting failure to
verify the best fair policy; and \textbf{Type III errors} select a fair
but worse policy (policies 5-12), reflecting failure to rank fair policies
correctly.

\begin{figure}[t]
    \centering
    \includegraphics[width=0.85\linewidth]{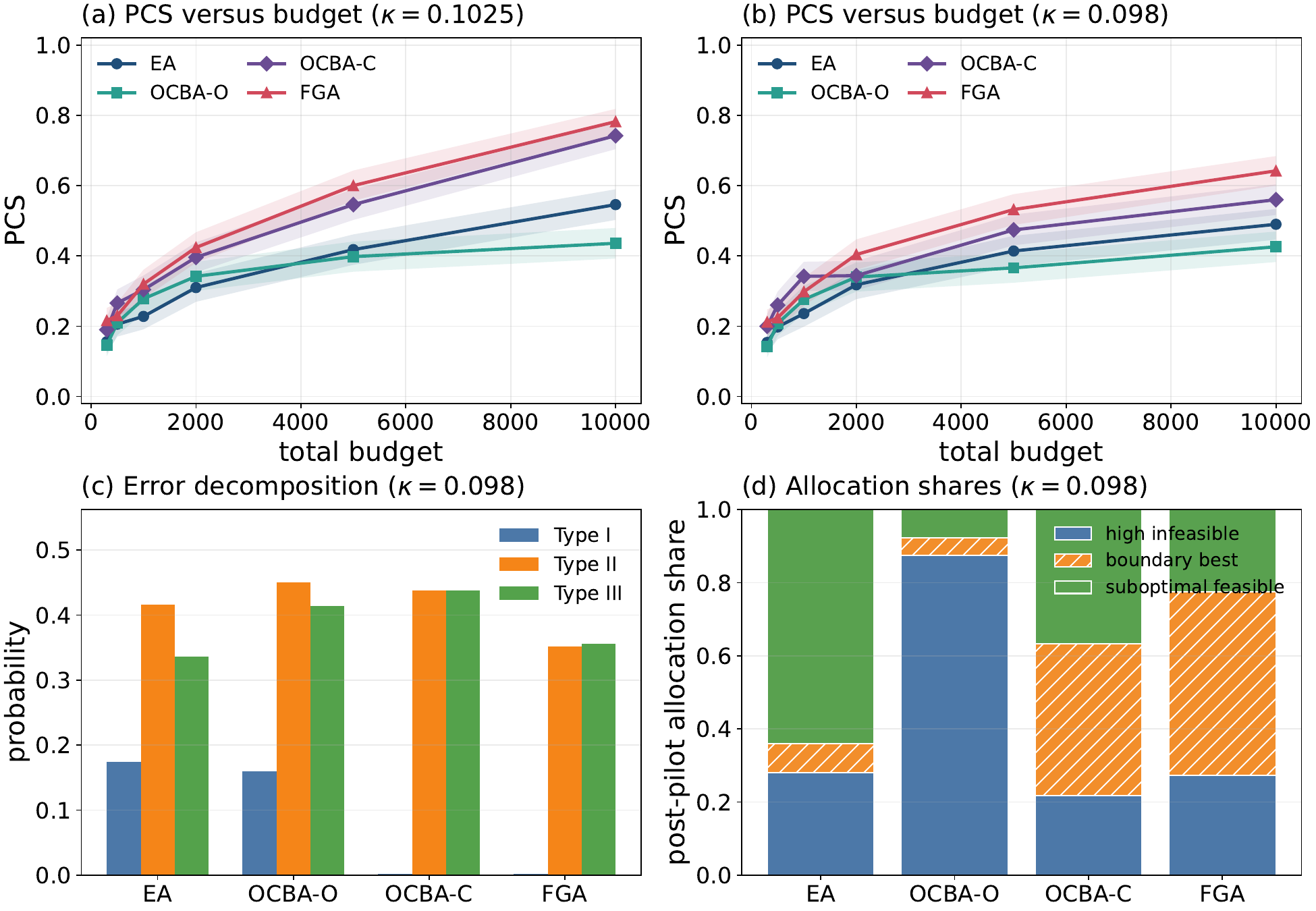}
    \caption{Synthetic Experiment II Results.
    (a) PCS at \(\kappa=0.1025\).
    (b) PCS at \(\kappa=0.098\).
    (c) Error type decomposition (Type I/II/III) across algorithms at \(T=10000\) with \(\kappa=0.098\).
    (d) Budget allocation decomposition at \(T=10000\) with \(\kappa=0.098\).}
    \label{fig:controlled2}
\end{figure}

Figure~\ref{fig:controlled2} summarizes the results. Panel (a) shows that FGA achieves the highest PCS across almost the entire sampling budget range, followed by OCBA-C. FGA and OCBA-C achieve PCS higher than \(0.7\) at budget \(T=10000\). Panel (b) also shows that FGA achieves the highest PCS across almost the entire sampling budget range, followed by OCBA-C. None of the algorithms achieve PCS higher than \(0.7\), even at budget \(T=10000\). This demonstrates that the budget required for explicitly selecting the best fair policy in fairness tolerance boundary settings is substantial. 

By comparing Panel (a) and Panel (b), we can substantiate the result proposed in Section \ref{sec:wedge}. Specifically, under the same fairness metric but with different tolerance levels \(\kappa\) sufficiently close to the critical fairness tolerance, the budgets required to achieve the PCS level \(1-\delta\) are considerably different when using the same algorithm. A tolerance closer to the critical value demands a larger budget. This observation indicates that the two fairness specifications incur different implementation costs, thereby illustrating the existence of an implementation gap. The tolerance is therefore not only a statement about the disparity an organization is willing to accept; it is also a decision that determines how difficult feasibility is to verify. Meanwhile, because both cases produce the same best fair policy, their operational costs remain identical, whereas their implementation costs differ.

Panel (c) decomposes the selection errors and helps explain the performance differences between algorithms. EA and OCBA-O have high Type I error probabilities. EA, OCBA-O and OCBA-C have relatively high Type II error probabilities and OCBA-O and OCBA-C have relatively high Type III error probabilities. In contrast, FGA reduces those error probabilities by directing sampling effort toward the
comparisons that are most relevant for fair selection.

Panel (d) decomposes allocated budget across three mutually
exclusive policy classes, namely better and unfair policies, the best fair
policy, and worse and fair policies. Worse and unfair policies are excluded
from this decomposition because they are not critical in this
experiment. EA spreads samples across policies
without regard to their relevance for the final fair selection decision, samples are allocated uniformly across all policy-group
pairs, so the allocation share of each class is determined by the number of
policies in that class. This provides a natural reference point for interpreting
the other methods. Compared with this baseline, OCBA-O allocates too much
budget to better and unfair policies. OCBA-C and FGA,
by contrast, allocate more budget to the best fair policy and its most
relevant competitors, where additional data has the greatest value.

These numerical results confirm that fairness-guided allocation is valuable when the best fair policy lies near the fairness boundary. In such regimes, algorithms that do not explicitly account for the difficulty of fairness verification are prone to costly finite-sample errors. Also, these results provide additional evidence that operational cost equivalence does not imply implementation cost equivalence.

\subsection{Call Center Experiment}

We study a multi-server queueing system with abandonment and \(G=2\)
customer groups. Customer arrivals are nonstationary, service times are
group-dependent, and customers may abandon the system if service has not
started before their patience time expires. The primary service burden outcome
is the waiting time \(W_{kg}\) for group \(g\) under policy \(k\). Detailed
distributional parameters and simulation settings are provided in
Section~\ref{sec: call center exp} of the appendix.

We consider \(K=12\) candidate policies, each combining a priority discipline
with a congestion threshold \(b\in\{5,10,15\}\). The priority disciplines
include first-come-first-served (FCFS), static priority for group 1, static
priority for group 2, and weighted priority. Under a threshold policy, the
system operates under FCFS when the total queue length is below \(b\) and
switches to the designated priority rule when the threshold is exceeded. This
policy class creates a tradeoff between average waiting time, group-level
waiting time disparity, and abandonment risk.

The objective function combines average waiting time, abandonment probability,
and policy complexity.

$$
    \mu_k
    =
    -
    \sum_{g=1}^{G}
    w_g
    \mathbb{E}[W_{kg}]
    -
    c_A
    \sum_{g=1}^{G}
    w_g
    \mathbb{P}(A_{kg}=1)
    -
    c_C C_k ,
$$
where \(w_1=w_2=0.5\), \(c_A=8.0\), and \(c_C=0.20\). Higher objective
values correspond to better aggregate performance. In the notation of
Section~\ref{sec:model},
\(Y_{kg}=-W_{kg}-c_A A_{kg}-c_C C_k\), while the fairness output is
\(X_{kg}=W_{kg}\).

The fairness constraint imposes upper tail waiting time parity based on the
CVaR of the worst \(10\%\) of waits. A policy is fair if the difference
between the two customer groups in upper tail CVaR waiting time does not exceed
\(\kappa=184.42\). For customers who abandon, \(W_{kg}\) is the time spent
waiting before abandonment, so abandonment affects both the aggregate
objective and the waiting time distribution used for fairness assessment.

Each algorithm begins with \(n_0=20\) replications per policy-group pair.
One independently seeded simulated day is one replication and serves as the
independent sampling unit; customer waits observed within a day are retained as
a batch. Additional implementation details are provided in the appendix.

\begin{figure}[t]
    \centering
\includegraphics[width=0.85\linewidth]{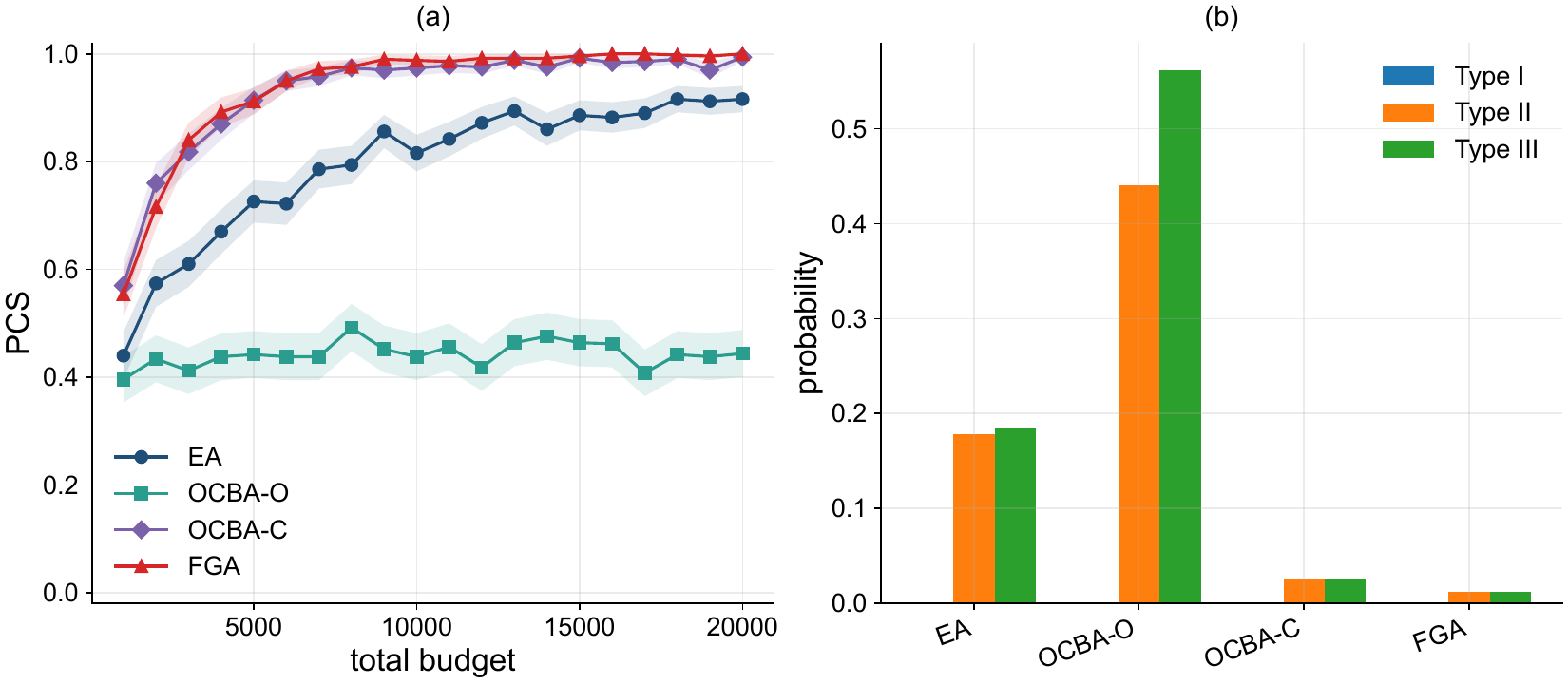}
    \caption{Service System Experiment Results.
    (a) PCS for different algorithms under the upper tail waiting time parity constraint.
    (b) Error type decomposition across algorithms at \(T=20000\).}
    \label{fig:service}
\end{figure}

Figure~\ref{fig:service} summarizes the results. Panel (a) shows that FGA achieves the highest or near-highest PCS in the sampling
budget range, closely followed by OCBA-C. EA improves more slowly and remains
less reliable at large budgets, while OCBA-O performs poorly because it focuses
only on objective performance and ignores the fairness constraint.
Panel (b) decomposes the selection errors at \(T=20000\). The main error types for EA and OCBA-O are discarding the best fair policy and misranking among fair policies. Both OCBA-C and FGA reduce these errors, with FGA achieving the best overall performance. According to these results, the implementation cost of fairness also arises in realistic service systems. In a congested system with abandonment and priority-based routing, tail-sensitive fairness constraints can be difficult to verify from a limited sampling budget. As a result, allocation rules that focus only on aggregate performance may favor policies that appear efficient but cannot be reliably verified as fair.

From a service design perspective, adopting a tail-sensitive fairness
requirement therefore creates a commitment of sampling budget as well as an
operational commitment. Before treating the selected policy as
deployment-ready, managers must assess whether the available
sampling budget is sufficient for the chosen metric and tolerance. FGA helps
use that budget efficiently, but it does not substitute for sufficient evidence.

\subsection{Emergency Department Experiment}
\label{sec:ed}

We consider an emergency department queueing system with \(G=2\) patient
groups over a 12-hour operating day. The system features nonstationary
arrivals, heterogeneous patient acuity, communication-support needs,
fast-track capacity, and abandonment. Patients pass through triage and
treatment resources, with some patients requiring additional communication
support. Detailed arrival, staffing, service time, patience, and
communication-support parameters are provided in
Section~\ref{sec: healthcare exp} of the appendix.

The service burden outcome is the waiting time \(W_{kg}\) for group \(g\)
under policy \(k\). For high-acuity patients, waiting time is measured as
door-to-provider time, and for low-acuity patients, it is measured from
arrival to completion. For patients who abandon, waiting time is the time
spent in the system before abandonment.

We consider \(K=12\) candidate policies that vary along two operational
dimensions, fast-track routing and priority assignment. The policy set
includes an FCFS baseline, acuity-based priority rules, static fast-track
rules, and fairness-aware fast-track rules. Static fast-track rules vary the
extent to which low-acuity patients are routed to the fast-track stream,
whereas fairness-aware fast-track rules further adjust same-acuity priorities
to reduce group-level waiting time disparities.

The objective combines service guarantee satisfaction, abandonment, and
policy complexity.

$$
    \mu_k
    =
    \sum_{g=1}^G w_g \mathbb{E}[Q_{kg}]
    -
    c_A \sum_{g=1}^G w_g \mathbb{P}(A_{kg}=1)
    -
    c_C C_k ,
$$
where \(Q_{kg}\) indicates whether the service guarantee is met,
\(A_{kg}\) indicates abandonment, \(w_1=w_2=0.5\), \(c_A=0.5\), and
\(c_C=0.01\). In the notation of Section~\ref{sec:model},
\(Y_{kg}=Q_{kg}-c_A A_{kg}-c_C C_k\), while the fairness output is
\(X_{kg}=W_{kg}\). The service guarantee requires door-to-provider time
within 30 minutes for high-acuity patients and length of stay within
120 minutes for low-acuity patients.

We consider three fairness metrics for group-level waiting time burdens: mean waiting time performance parity, the 90th quantile of waiting time parity, and upper tail parity using CVaR controls disparities in the expected waiting time among the worst 10\% of outcomes. For the main upper tail parity experiment, the fairness tolerance is \(\kappa=38\). The
candidate policies include high-performing policies that violate the
fairness constraint, a best fair policy close to the fairness boundary, and
more conservative fair policies with lower objective values. This structure
allows us to examine whether an allocation rule can simultaneously avoid
unfair high-performing policies and verify a near-boundary best fair policy.

As in the call center experiment, one independently seeded simulated day is
one outer replication and serves as the independent sampling unit. Additional
replication and estimation details are provided in the appendix.

\begin{figure}[htbp]
    \centering
    \includegraphics[width=0.85\textwidth]{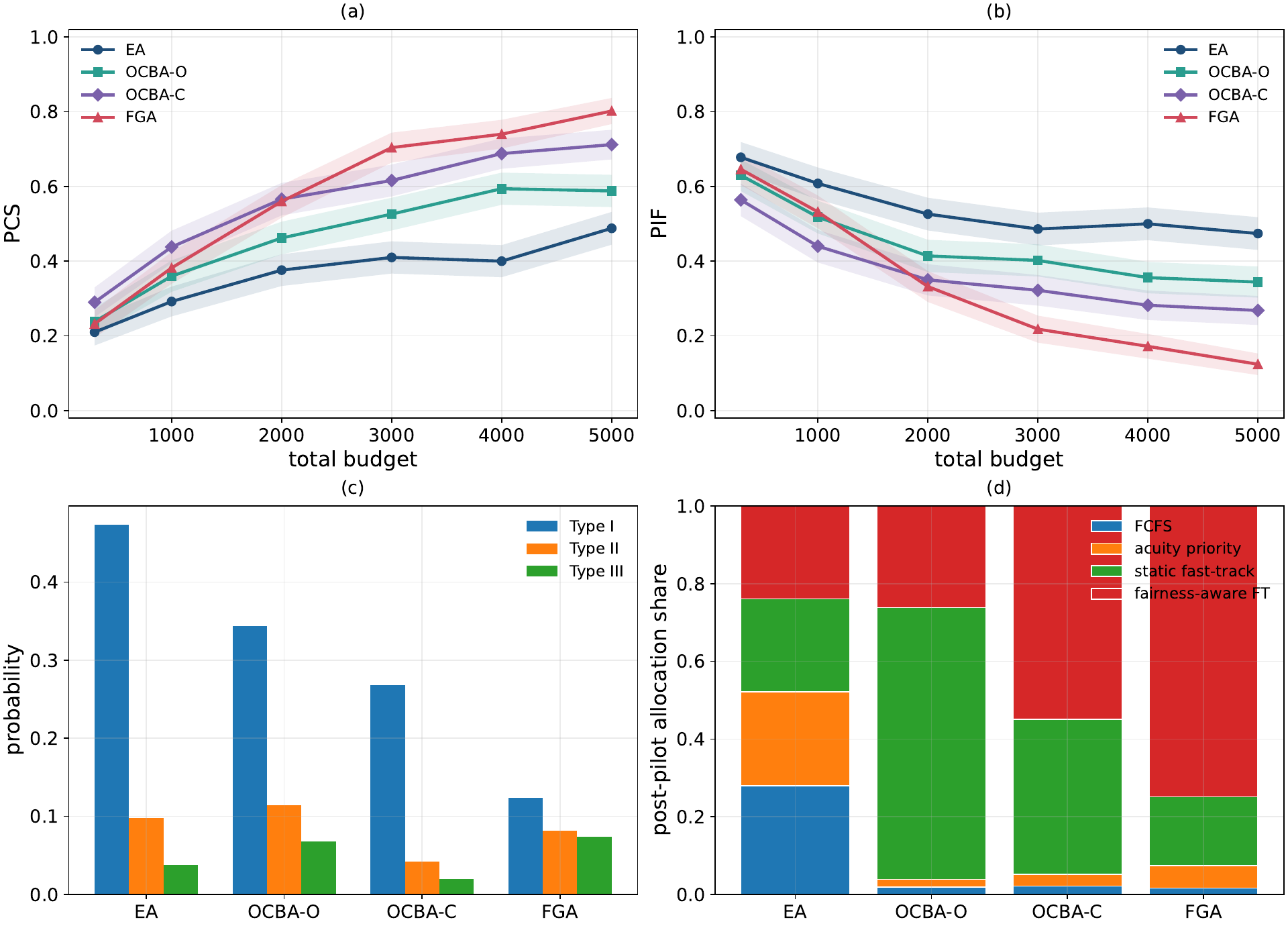}
    \caption{
    Emergency Department Experiment under Upper Tail Parity.
    (a) PCS of the algorithms.
    (b) PIF of the algorithms.
    (c) Error type decomposition across algorithms at \(T=5000\).
    (d) Budget allocation decomposition at \(T=5000\).
    }
    \label{fig:healthcare_cvar}
\end{figure}

Figure~\ref{fig:healthcare_cvar} reports the emergency department queueing experiment results under upper tail parity. Panel~(a) compares the tested algorithms. FGA attains the largest PCS and reaches about 0.8 at \(T=5000\). OCBA-C is the closest competitor, while EA and OCBA-O improve more slowly and remain below the two fairness-aware algorithms at large budgets. 
Panel~(b) shows that FGA also has the lowest PIF with budget above \(T=2000\). Its PIF decreases steadily as the budget increases, indicating that the algorithm more effectively rules out unfair policies. In contrast, EA and OCBA-O retain relatively high PIF even at larger budgets. 

Panel~(c) decomposes the  selection errors at \(T=5000\). EA and OCBA-O are dominated by Type~I errors, reflecting frequent selection of unfair policies. OCBA-C reduces this error but still exhibits a nonnegligible Type~I error probability. FGA achieves the smallest Type~I error and a balanced reduction across all error types, leading to the best overall selection performance.
Panel~(d) explains these differences through the decomposition of the allocated budget. EA acts as the proportional baseline here, where the share of each policy class is strictly determined by the number of policies in that category. FGA allocates a larger fraction of the sampling budget to fairness-aware fast-track, while OCBA-C also devotes substantial effort to both fairness-aware fast-track and static fast tracking. By contrast, OCBA-O concentrates heavily on static fast tracking. These allocation patterns show that explicitly targeting difficult fairness verification comparisons improves both fairness control and finite-budget PCS in the emergency department setting. 

We also report results for mean performance parity and high quantile parity in Section \ref{sec: healthcare exp}.
For comparison, the fairness tolerances are calibrated for each metric such that all three parities yield the same best fair policy and the same price of fairness (i.e., operational cost). Nevertheless, under a fixed sampling budget, the three parities achieve different finite-sample PCS values. This implies that different parities incur different implementation costs. With a fixed sampling budget \(T=5000\), a target PCS level of 0.85 is statistically implementable for the mean performance parity setting, but not for the high quantile or upper tail parity settings.
These results reinforce the core finding that operational cost equivalence does not guarantee implementation cost equivalence.

Thus, the choice of fairness metric determines whether an otherwise operationally equivalent fairness specification can be supported by the available budget. For an emergency department manager, this means that metric choice cannot be separated from budget planning. Tail-sensitive fairness specifications may be substantively desirable because they protect patients from severe waiting burdens, but they may require a larger sampling budget before the corresponding policy can be selected with the same statistical reliability. The appropriate design therefore depends jointly on the fairness metric, the tolerance, and the available sampling budget.

\section{Conclusions and Discussion}
\label{sec:conclu}

This paper studies two costs associated with fairness in service policy
selection when the distributional quantities defining fairness must be
estimated under a finite sampling budget. The operational cost is the familiar
price of fairness, which measures the aggregate performance loss induced by
restricting attention to fair policies. The implementation cost of fairness is
the minimum sampling budget required to verify policy feasibility and reliably
identify the best fair policy at a target PCS.

Our analysis identifies two mechanisms. In common-shape location-shift environments, translation-equivariant metrics can generate the same fair policy set and price of fairness while their influence function variances produce different local verification hardness. Near a critical tolerance, the budget required for reliable selection grows at an inverse-square rate in the distance to the boundary. These results connect the population fairness decision to the finite budget available for implementing it. The fairness-guided allocation algorithm uses this connection by directing samples toward the objective and feasibility comparisons that govern false selection.

Our analysis also points to three broader considerations for fair service
policy selection. First, fairness is not only a constraint on admissible service policies but also a commitment that requires sufficient evidence before implementation. When the quantities defining fairness must be estimated, the metric and tolerance determine both operational consequences and the evidence required for reliable verification.

Second, fairness specifications with similar operational consequences can have substantially different implementation requirements. Two specifications may induce the same fair policy set and price of fairness, yet differ sharply in the sampling budget required for reliable verification. Implementation cost therefore captures whether a fairness requirement can support reliable policy selection at the available budget.

Third, data collection, policy selection, and fairness design should be treated as parts of the same operational decision. A service organization chooses not only a policy and a fairness requirement, but also the outcomes to measure, the sampling budget to provide, and the comparisons toward which simulation or experimental effort should be directed. It should therefore design the fairness metric \(\psi\), the fairness tolerance \(\kappa\), and the available sampling budget \(T\) jointly. The metric determines what aspect of service experience is protected; the tolerance determines how much disparity is accepted and where the feasibility boundary lies; and the budget determines whether those distinctions can be established with the desired reliability. This joint view is not limited to the service settings considered here. As organizations increasingly use data-driven systems to allocate capacity, set priorities, and govern access to scarce services, responsible implementation requires both an appropriate fairness requirement and sufficient evidence to verify that the requirement is satisfied.


\bibliographystyle{plainnat}
\bibliography{ref}


\clearpage
\appendix

\section*{Appendix for ``The Implementation Cost of Fairness in Service Policy Selection''}

This appendix provides supplementary  materials for the main paper. It is organized as follows. Section \ref{app:lal-proof} verifies Assumption \ref{ass:lal} for all four fairness metrics considered in this research. Section \ref{calculation} derives the full procedure for solving \(J_{k,k^\star_{\mathfrak{f}}(\kappa)}(0)\) from Theorem \ref{thm:mean-rate1} and provides closed-form results under the normal distributions. Section \ref{app:proofs} presents proofs of the theoretical results in the main paper. Section \ref{sec: numerical details} provides full implementation details for the numerical experiments.

\section{Verification for Assumption~\ref{ass:lal}}
\label{app:lal-proof}

This section verifies Assumption~\ref{ass:lal} for the four fairness
metrics used in the paper. Throughout the section, we fix a policy-group
pair \((k,g)\) and omit the subscripts. Let \(X \sim P\) denote the random variable of a single service outcome. Let \(X_1, X_2, \dots, X_n\) be i.i.d. outcomes drawn from \(P\), with cumulative distribution function \(F\). Let \(\widehat P_n\) and \(\widehat F_n\) denote the
empirical distribution and empirical distribution function, respectively.
For any measurable function \(h\), write
\[
    Ph=\int h(x)\,dP(x),
    \qquad
    \widehat P_n h
    =
    \int h(x)\,d\widehat P_n(x)
    =
    \frac1n\sum_{i=1}^n h(X_i).
\]
All asymptotic statements are taken as \(n\to\infty\) for this fixed
policy-group pair.

We consider the following four metrics.
\[
    \psi_{\mathrm{mean}}(P)=\mathbb E_P[X],
\quad 
    \psi_{s}(P)=P(X\le s),
\]
\[
    \psi_{Q_\beta}(P)=Q_\beta(P)
    =
    \inf\{x\mid F(x)\ge \beta\},
\]
and
\[
    \psi_{\mathrm{cvar}_\beta}(P)
    =
    \CVaR_\beta^+(P)
    =
    q_\beta
    +
    \frac{1}{1-\beta}\mathbb E_P[(X-q_\beta)_+],
    \qquad
    q_\beta=Q_\beta(P),
\]
where \(s\in\mathbb R\) is a fixed service threshold,
\(\beta\in(0,1)\) is fixed, and \((a)_+=\max\{a,0\}\).

For each policy-group pair \((k,g)\), consider the following regularity classes for distribution \(P\).

\begin{enumerate}
    \item \textbf{Mean performance metric.}
    If \(\psi(P)=\mathbb E_P[X]\), assume that $0<\operatorname{Var}_P(X)<\infty$.
    \item \textbf{Under-service risk metric.}
    If \(\psi(P)=P(X\le s)\), define $p_s=P(X\le s)$, and assume
    \(
        0<p_s<1.
    \)
    \item \textbf{High quantile metric.}
    If \(\psi(P)=Q_\beta(P)\), let
    \(
        q_\beta=Q_\beta(P).
    \)
    Assume that $F$ has at least two derivatives in some neighborhood of $q_\beta$,
    that $F''(x)$ is bounded in the neighborhood, and that $F'(q_\beta) = f(q_\beta) > 0$.
    These assumptions imply, in particular, that $0 < \beta < 1$ and that $q_\beta$ is the unique $\beta$-quantile of $F$.

    \item \textbf{Upper tail metric.}
    If
    \[
        \psi(P)
        =
        \CVaR_\beta^+(P)
        =
        q_\beta
        +
        \frac{1}{1-\beta}\mathbb E_P[(X-q_\beta)_+],
    \]
    assume the quantile regularity class in high quantile performance metric at
    \(q_\beta=Q_\beta(P)\). In addition, assume $\mathbb E_P[(X-q_\beta)_+^2]<\infty$.
\end{enumerate}

\begin{proposition}
\label{class}
Given fairness metric \(\psi\), for distribution \(P\) in the corresponding regularity class,
\[
    \sqrt n\{\psi(\widehat P_n)-\psi(P)\}
    =
    \frac1{\sqrt n}\sum_{i=1}^n IF_\psi(X_i;P)+o_p(1),
\]
where
\[
    IF_{\psi_{\mathrm{mean}}}(x;P)
    =
    x-\mathbb E_P[X],
\]
\[
    IF_{\psi_{s}}(x;P)
    =
    \mathbf 1\{x\le s\}-P(X\le s),
\]
\[
    IF_{\psi_{Q_\beta}}(x;P)
    =
    \frac{\beta-\mathbf 1\{x\le q_\beta\}}{f(q_\beta)},
\]
\[
    IF_{\psi_{\mathrm{cvar}_\beta}}(x;P)
    =
    q_\beta
    +
    \frac{(x-q_\beta)_+}{1-\beta}
    -
    \CVaR_\beta^+(P).
\]
Moreover, $\mathbb E_P[IF_\psi(X;P)]=0$, $0<\operatorname{Var}_P\{IF_\psi(X;P)\}
    <\infty$, where 
\(\mathbf{1}_{A}\) denotes the indicator function of event \(A\).
\end{proposition}

\proof{Proof of proposition \ref{class}}
We verify the result case by case.

\paragraph{Mean performance metric}
Let
\(
    \psi(P)=\mathbb E_P[X]=\mu.
\)
Then $ \psi(\widehat P_n)
    =
    \frac1n\sum_{i=1}^n X_i$.
Therefore,
\[
\begin{aligned}
    \sqrt n\{\psi(\widehat P_n)-\psi(P)\}
    =
    \sqrt n
    \left\{
        \frac1n\sum_{i=1}^n X_i-\mu
    \right\} 
    =
    \frac1{\sqrt n}\sum_{i=1}^n (X_i-\mu).
\end{aligned}
\]
Thus the LAL holds exactly, and
\(
    IF_{\psi_{\mathrm{mean}}}(x;P)=x-\mu.
\)
Also,
\(
    \mathbb E_P[IF_{\psi_{\mathrm{mean}}}(X;P)]=0,
\)
and
\[
    \operatorname{Var}_P\{IF_{\psi_{\mathrm{mean}}}(X;P)\}
    =
    \operatorname{Var}_P(X)
    \in(0,\infty)
\]
by the stated condition.

\paragraph{Under-service risk metric}
Let
\(
    \psi(P)=P(X\le s).
\)

Define $A_s=(-\infty,s]$ and $p_s=P(A_s)=P(X\le s)=F(s)$. Then
\[
    \psi(\widehat P_n)
    =
    \widehat P_n(A_s)
    =
    \frac1n\sum_{i=1}^n \mathbf 1\{X_i\le s\}.
\]
Hence
\[
\begin{aligned}
    \sqrt n\{\psi(\widehat P_n)-\psi(P)\}
    &=
    \sqrt n\{\widehat P_n(A_s)-P(A_s)\} \\
    &=
    \sqrt n
    \left\{
        \frac1n\sum_{i=1}^n \mathbf 1\{X_i\le s\}
        -
        p_s
    \right\} \\
    &=
    \frac1{\sqrt n}\sum_{i=1}^n
    \left[
        \mathbf 1\{X_i\le s\}-p_s
    \right].
\end{aligned}
\]
Thus the LAL holds, and
\[
    IF_{\psi_{s}}(x;P)
    =
    \mathbf 1\{x\le s\}-P(X\le s).
\]
Furthermore,
\[
    \mathbb E_P[IF_{\psi_s}(X;P)]
    =
    p_s-p_s
    =
    0,
\]
and
\[
    \operatorname{Var}_P\{IF_{\psi_s}(X;P)\}
    =
    \operatorname{Var}_P(\mathbf 1\{X\le s\})
    =
    p_s(1-p_s)
    \in(0,\infty)
\]
under \(0<p_s<1\).

\paragraph{High quantile metric}
For the high quantile metric, under the assumption that $F$ is absolutely continuous in a neighborhood of $q_\beta$ with a continuous and positive density $f(q_\beta)>0$, the Bahadur representation \citep{bahadur1966note} yields
\[
\sqrt{n}(\widehat q_\beta - q_\beta) = \frac{1}{\sqrt{n}} \sum_{i=1}^n \frac{\beta - \mathbf{1}\{X_i \le q_\beta\}}{f(q_\beta)} + o_p(1).
\]
Hence the influence function is 
    \[
    IF_{\psi_{Q_\beta}}(x;P)
    =
    \frac{\beta-\mathbf 1\{x\le q_\beta\}}{f(q_\beta)}.
    \] 
Furthermore,
\[
    \mathbb E_P[IF_{\psi_{Q_\beta}}(X;P)]
    =\frac{\beta-\beta}{f(q_\beta)}=0.
\] 
Because \(\mathbf{1}\{X\le q_{\beta}\}\) is a Bernoulli random variable
with success probability \(\beta\),
\[
\mathbb{E}
\left[
\left(
\beta-\mathbf{1}\{X\le q_{\beta}\}
\right)^2
\right]
=
\beta(1-\beta).
\]
\[
\operatorname{Var}
\left(
\operatorname{IF}_{\psi_{Q_\beta}}(X;P)
\right)
=
\frac{1}{f(q_{\beta})^2}
\mathbb{E}
\left[
\left(
\beta-\mathbf{1}\{X\le q_{\beta}\}
\right)^2
\right]=\frac{\beta(1-\beta)}{f(q_\beta)^2}
\]
is finite and strictly positive under the given conditions.

\paragraph{Upper tail metric}
Let \(q=q_\beta\) and define \(g_t(x)=(x-t)_+\).
Write the upper tail metric as
\(\psi(P) = q + (1-\beta)^{-1} P g_q\).
The plug-in estimator is
\(\psi(\widehat P_n) = \widehat q + (1-\beta)^{-1} \widehat P_n g_{\widehat q}\),
where \(\widehat q = Q_\beta(\widehat P_n)\).
By the Bahadur representation above, \(\widehat q - q = O_p(n^{-1/2})\).

Consider the decomposition
\begin{equation}
\begin{aligned}
\psi(\widehat P_n) - \psi(P)
&= (\widehat q - q) + \frac{1}{1-\beta}\bigl(\widehat P_n g_{\widehat q} - P g_q\bigr) \\
&= (\widehat q - q) + \frac{1}{1-\beta}(\widehat P_n - P)g_q  + \frac{1}{1-\beta}\bigl[P(g_{\widehat q} - g_q)\bigr]
   + \frac{1}{1-\beta}(\widehat P_n - P)(g_{\widehat q} - g_q).
\end{aligned}
\label{eq:cvar-decomp}
\end{equation}

Define \(G(t)=\mathbb{E}_P[(X-t)_+]\). Under the high quantile performance metric regularity
and \(\mathbb{E}[(X-q)_+^2]<\infty\), an application of the Dominated Convergence Theorem yields
\(G\) is differentiable at \(q\)
with derivative \(G'(q) = -P(X>q) = -(1-\beta)\).
Because \(\widehat q - q = O_p(n^{-1/2})\), a Taylor expansion gives
\begin{equation}
P(g_{\widehat q} - g_q) = G(\widehat q) - G(q)
= -(1-\beta)(\widehat q - q) + o_p(n^{-1/2}).
\label{eq:G-expand}
\end{equation}

Consider the class of functions
\(\mathcal{G} = \{g_t \mid t\in [q-\delta, q+\delta]\}\) for some small
\(\delta>0\). Under \(\mathbb{E}[(X-q)_+^2]<\infty\) the envelope
\(G_\delta(x) = (x-(q-\delta))_+\) is square-integrable.
The class \(\mathcal{G}\) is a VC-subgraph class and hence
\(P\)-Donsker \citep{vaart1996}.
By the random equicontinuity of the empirical process,
\begin{equation}
\sup_{|t-q|\le \delta_n}
\bigl|\sqrt{n}(\widehat P_n-P)(g_t - g_q)\bigr| = o_p(1)
\qquad\text{whenever }\delta_n\to 0.
\label{eq:equi}
\end{equation}
Since \(\widehat q \xrightarrow{p} q\), we may set \(t=\widehat q\) and
obtain
\[
(\widehat P_n-P)(g_{\widehat q} - g_q) = o_p(n^{-1/2}).
\]

Insert~\eqref{eq:G-expand} and the bound on the empirical term
into~\eqref{eq:cvar-decomp}.
\[
\begin{aligned}
\psi(\widehat P_n) - \psi(P)
&= (\widehat q - q) + \frac{1}{1-\beta}(\widehat P_n - P)g_q
   + \frac{1}{1-\beta}\bigl[-(1-\beta)(\widehat q - q) + o_p(n^{-1/2})\bigr] \\
&= \frac{1}{1-\beta}(\widehat P_n - P)g_q + o_p(n^{-1/2}).
\end{aligned}
\]
Multiplying by \(\sqrt{n}\) yields
\[
\sqrt{n}\bigl(\psi(\widehat P_n)-\psi(P)\bigr)
= \frac{\sqrt{n}}{1-\beta}(\widehat P_n-P)g_q + o_p(1)
= \frac{1}{\sqrt{n}}\sum_{i=1}^n
   \frac{g_q(X_i) - \mathbb{E}[g_q(X)]}{1-\beta} + o_p(1).
\]
Rewriting the summand using
\(\CVaR_\beta^+(P) = q + \mathbb{E}[g_q(X)]/(1-\beta)\),
we obtain the influence function
\[
IF_{\psi_{\mathrm{cvar}_\beta}}(x;P) = q + \frac{(x-q)_+}{1-\beta} - \CVaR_\beta^+(P).
\]
Clearly \(\mathbb{E}_P[IF_{\psi_{\mathrm{cvar}_\beta}}(X;P)] = 0\). The variance is
\[
\operatorname{Var}_P(IF_{\psi_{\mathrm{cvar}_\beta}})
= \frac{\operatorname{Var}_P\{(X-q)_+\}}{(1-\beta)^2}.
\]
By assumption \(\mathbb{E}[(X-q)_+^2]<\infty\), this variance is finite.
It is strictly positive because
\(P(X>q) = 1-\beta > 0\) implies that \((X-q)_+\) is not almost surely
constant.

Combining the four cases proves the LAL for all fairness
metrics considered in the paper.

Restoring the policy-group subscripts, for every \(P_{kg}\) in the
corresponding metric-based regularity class, we have
\[
    \sqrt{N_{kg}}
    \left\{
        \psi(\widehat P_{kg})-\psi(P_{kg})
    \right\}
    =
    \frac1{\sqrt{N_{kg}}}
    \sum_{i=1}^{N_{kg}}
    IF_\psi(X_{kg,i};P_{kg})
    +
    o_p(1),
\]
with $\mathbb E_{P_{kg}}[IF_\psi(X_{kg};P_{kg})]=0$ and $0<
    \operatorname{Var}_{P_{kg}}\{IF_\psi(X_{kg};P_{kg})\}
    <\infty.$
\endproof
This verifies Assumption~\ref{ass:lal}.

\section{Procedure for Solving $J_{k,k^\star_{\mathfrak{f}}(\kappa)}(0)$ in Theorem \ref{thm:mean-rate1}}
\label{calculation}

\subsection{General Procedure for Solving $J_{k,k^\star_{\mathfrak{f}}(\kappa)}(0)$}
In this section, we solve the optimization problem for $J_{k,k^\star_{\mathfrak{f}}(\kappa)}(0)$ in Theorem \ref{thm:mean-rate1}, building directly on the definition of the Cramér transform (rate function) and its core convex conjugacy properties established in the preceding sections. The objective function of this problem is a weighted sum of the Cramér transforms corresponding to the group-level distributions, subject to a linear equality constraint on the weighted group means. This forms a convex program where strong duality holds. Leveraging the Lagrange multiplier method, combined with the conjugate reciprocal relationship between the Cramér transform and the cumulant generating function (CGF), we reduce this high-dimensional optimization problem to an efficiently solvable univariate root-finding problem. This computational framework applies to distributions that satisfy the regularity condition of Cramér's Theorem (i.e., the CGF is finite on an open interval containing 0).

We introduce the Lagrange multiplier $\lambda$ for the linear equality constraint in Theorem \ref{thm:mean-rate1}, and construct the Lagrangian function.
\begin{equation}
\mathcal{L} = \sum_{g\in\mathcal{G}} \alpha_{kg} I_{R_{kg}}(m_{kg}) + \sum_{g\in\mathcal{G}} \alpha_{k^\star_{\mathfrak{f}}(\kappa)g} I_{R_{k^\star_{\mathfrak{f}}(\kappa)g}}(m_{k^\star_{\mathfrak{f}}(\kappa)g}) + \lambda \left( \sum_{g\in\mathcal{G}} w_g m_{k^\star_{\mathfrak{f}}(\kappa)g} - \sum_{g\in\mathcal{G}} w_g m_{kg} \right)
\end{equation}
where $I_{R_{kg}}(\cdot)$ denotes the Cramér transform (rate function) of the distribution $R_{kg}$ for policy $k$ and group $g$, as defined earlier. Under additional regularity ensuring that the Cramér rate functions $I_R(\cdot)$ are differentiable and strictly convex on the relevant interiors, equivalently with the corresponding CGFs $\Lambda_R(\cdot)$ sufficiently smooth, and that the optimizer is interior, the KKT conditions reduce the convex program to a one-dimensional
root-finding problem.
\begin{itemize}
    \item Taking the partial derivative with respect to the optimization variable $m_{kg}$ for policy $k$ and group $g$ and setting it to zero.
    \begin{equation}
    \frac{\partial \mathcal{L}}{\partial m_{kg}} = \alpha_{kg} \cdot I'_{R_{kg}}(m_{kg}) - \lambda w_g = 0 \implies I'_{R_{kg}}(m_{kg}) = \lambda \cdot \frac{w_g}{\alpha_{kg}}
    \end{equation}
    \item Taking the partial derivative with respect to the optimization variable $m_{k^\star_{\mathfrak{f}}(\kappa)g}$ for policy $k^\star_{\mathfrak{f}}(\kappa)$ and group $g$ and setting it to zero.
    \begin{equation}
    \frac{\partial \mathcal{L}}{\partial m_{k^\star_{\mathfrak{f}}(\kappa)g}} = \alpha_{k^\star_{\mathfrak{f}}(\kappa)g} \cdot I'_{R_{k^\star_{\mathfrak{f}}(\kappa)g}}(m_{k^\star_{\mathfrak{f}}(\kappa)g}) + \lambda w_g = 0 \implies I'_{R_{k^\star_{\mathfrak{f}}(\kappa)g}}(m_{k^\star_{\mathfrak{f}}(\kappa)g}) = - \lambda \cdot \frac{w_g}{\alpha_{k^\star_{\mathfrak{f}}(\kappa)g}}
    \end{equation}
\end{itemize}

From the definition of the Cramér transform, the rate function $I_R(m)$ is the Legendre--Fenchel conjugate of the corresponding distribution's CGF $\Lambda_R(\theta) = \log \mathbb{E}_R[e^{\theta Y}]$. The two satisfy the core reciprocal property of convex conjugates (proven in the preceding sections).
$$
I'_R(m) = \theta \iff m = \Lambda'_R(\theta),
$$
where $\Lambda'_R(\theta)$ is the first derivative of the CGF. Using this property, we can convert the derivative of the rate function in the first-order optimality conditions into the inverse mapping of the CGF, and directly express the optimal group means as explicit functions of the Lagrange multiplier $\lambda$.
\begin{equation}
m_{kg}^*(\lambda) = \Lambda'_{R_{kg}}\left( \lambda \cdot \frac{w_g}{\alpha_{kg}} \right), \quad m_{k^\star_{\mathfrak{f}}(\kappa)g}^*(\lambda) = \Lambda'_{R_{k^\star_{\mathfrak{f}}(\kappa)g}}\left( - \lambda \cdot \frac{w_g}{\alpha_{k^\star_{\mathfrak{f}}(\kappa)g}} \right)
\end{equation}

Substituting the explicit expressions for the optimal means above into the global weighted mean equality constraint in Theorem \ref{thm:mean-rate1} eliminates all high-dimensional optimization variables. This gives a univariate root-finding equation in terms of $\lambda$ only.
\begin{equation}
F(\lambda) = \sum_{g\in\mathcal{G}} w_g \left[ \Lambda'_{R_{kg}}\left( \lambda \cdot \frac{w_g}{\alpha_{kg}} \right) - \Lambda'_{R_{k^\star_{\mathfrak{f}}(\kappa)g}}\left( - \lambda \cdot \frac{w_g}{\alpha_{k^\star_{\mathfrak{f}}(\kappa)g}} \right) \right] = 0
\end{equation}

Combined with the premise of Theorem \ref{thm:mean-rate1} that $\mu_k \leq \mu_{k^\star_{\mathfrak{f}}(\kappa)}$, we can prove that this equation has a unique nonnegative solution.
\begin{itemize}
    \item At $\lambda=0$, $F(0) = \sum_{g\in\mathcal{G}} w_g (\mu_{kg} - \mu_{k^\star_{\mathfrak{f}}(\kappa)g}) = \mu_k - \mu_{k^\star_{\mathfrak{f}}(\kappa)} \leq 0$;
    \item $F(\lambda)$ is a strictly increasing function of $\lambda$, and $F(\lambda) \to +\infty$ as $\lambda \to +\infty$;
    \item By the Intermediate Value Theorem, there exists a unique $\lambda^* \geq 0$ satisfying $F(\lambda^*)=0$, which can be solved efficiently via Newton's method.
\end{itemize}

Substitute the solved optimal multiplier $\lambda^*$ into the optimal mean expressions, then plug these optimal group means into the objective function to obtain the exact solution for $J_{k,k^\star_{\mathfrak{f}}(\kappa)}(0)$.

\subsection{Closed-Form Results for Normal Distributions}
The normal distribution is the standard modeling assumption for continuous outcomes (e.g., service times or performance scores), for which we have closed-form results for the CGF and Cramér transform. Assume that for all $g\in\mathcal{G}$, the within-group outcomes follow independent normal distributions, with $Y_{kg} \sim \mathcal{N}(\mu_{kg}, \sigma_{kg}^2)$ and $Y_{k^*g} \sim \mathcal{N}(\mu_{k^*g}, \sigma_{k^*g}^2)$. The CGF of a normal distribution is $\Lambda_R(\theta) = \mu\theta + \frac{\sigma^2 \theta^2}{2}$, with first derivative $\Lambda'_R(\theta) = \mu + \sigma^2 \theta$, consistent with our earlier derivations.

1.  Substitute the first derivative of the normal CGF into the optimal mean expressions.
    \begin{equation}
    m_{kg}^*(\lambda) = \mu_{kg} + \sigma_{kg}^2 \cdot \frac{\lambda w_g}{\alpha_{kg}}, \quad m_{k^*g}^*(\lambda) = \mu_{k^*g} - \sigma_{k^*g}^2 \cdot \frac{\lambda w_g}{\alpha_{k^*g}}
    \end{equation}
2.  Substitute into the constraint equation $F(\lambda)=0$ to solve for the optimal multiplier $\lambda^*$.
    \begin{equation}
    (\mu_k - \mu_{k^*}) + \lambda \cdot \underbrace{\sum_{g\in\mathcal{G}} w_g^2 \left( \frac{\sigma_{kg}^2}{\alpha_{kg}} + \frac{\sigma_{k^*g}^2}{\alpha_{k^*g}} \right)}_{S>0} = 0 \implies \lambda^* = \frac{\mu_{k^*} - \mu_k}{S}
    \end{equation}
3.  Substitute $\lambda^*$ and the optimal means into the objective function, and simplify using the closed-form Cramér transform for the normal distribution, to obtain the final analytical result.
    \begin{equation}
    {J_{k,k^*}(0) = \frac{(\mu_{k^*} - \mu_k)^2}{2 \cdot \sum_{g\in\mathcal{G}} w_g^2 \left( \frac{\sigma_{kg}^2}{\alpha_{kg}} + \frac{\sigma_{k^*g}^2}{\alpha_{k^*g}} \right)}}
    \end{equation}

It has the following properties.
\begin{itemize}
    \item When $\mu_k = \mu_{k^*}$, $J_{k,k^*}(0)=0$, consistent with large deviation theory because equal true means imply no exponential decay of the error probability;
    \item When $\mu_k < \mu_{k^*}$, $J_{k,k^*}(0)>0$, meaning the error probability decays exponentially with the total sample size $N$ at rate $J_{k,k^*}(0)$;
    \item When $G=1$, the result reduces to the classic Cramér large deviation rate function for two policies normal mean comparison, given by $J=\frac{(\mu_{k^*}-\mu_k)^2}{2(\frac{\sigma_{k}^2}{\alpha_{k}}+\frac{\sigma_{k^*}^2}{\alpha_{k^*}})}$.
\end{itemize}

\section{Proofs of Theoretical Results}
\label{app:proofs}

This section provides proofs for the theoretical results in the main text. 

\subsection{Proof of Proposition \ref{prop:local}}

\proof{}
Fix a policy \(k\) and a group pair \((g,g')\). Recall that
\(
d_{k,gg'}^\psi
=
\psi(P_{kg})-\psi(P_{kg'}).
\)
We first consider the case in which the two allocation shares are positive.
\(
\alpha_{kg}>0,
\) and \(
\alpha_{kg'}>0.
\)
By algebraic decomposition,
\[
\begin{split}
\sqrt{T}
\left[
\left(
\psi(\widehat P_{kg})-\psi(\widehat P_{kg'})
\right)
-
d_{k,gg'}^\psi
\right] =
\sqrt{T}
\left(
\psi(\widehat P_{kg})-\psi(P_{kg})
\right)
-
\sqrt{T}
\left(
\psi(\widehat P_{kg'})-\psi(P_{kg'})
\right).
\end{split}
\]
Since \(\alpha_{kg}=N_{kg}/T\), we have
\[
\sqrt{T}
\left(
\psi(\widehat P_{kg})-\psi(P_{kg})
\right)
=
\frac{1}{\sqrt{\alpha_{kg}}}
\sqrt{N_{kg}}
\left(
\psi(\widehat P_{kg})-\psi(P_{kg})
\right).
\]
By Assumption \ref{ass:lal},
\[
\sqrt{N_{kg}}
\left(
\psi(\widehat P_{kg})-\psi(P_{kg})
\right)
=
\frac{1}{\sqrt{N_{kg}}}
\sum_{i=1}^{N_{kg}}
\operatorname{IF}_\psi(X_{kg,i};P_{kg})
+
o_p(1),
\]
where $\mathbb{E}
\left[
\operatorname{IF}_\psi(X_{kg};P_{kg})
\right]
=
0$ and $\operatorname{Var}
\left[
\operatorname{IF}_\psi(X_{kg};P_{kg})
\right]
=
\sigma_{\psi,kg}^2
\in(0,\infty)$.
The central limit theorem gives
\[
\frac{1}{\sqrt{N_{kg}}}
\sum_{i=1}^{N_{kg}}
\operatorname{IF}_\psi(X_{kg,i};P_{kg})
\Rightarrow
\mathcal N(0,\sigma_{\psi,kg}^2).
\]
Therefore,
\[
\sqrt{T}
\left(
\psi(\widehat P_{kg})-\psi(P_{kg})
\right)
\Rightarrow
\mathcal N
\left(
0,
\frac{\sigma_{\psi,kg}^2}{\frac{N_{kg}}{T}}
\right).
\]
The same argument gives
\[
\sqrt{T}
\left(
\psi(\widehat P_{kg'})-\psi(P_{kg'})
\right)
\Rightarrow
\mathcal N
\left(
0,
\frac{\sigma_{\psi,kg'}^2}{\frac{N_{kg'}}{T}}
\right).
\]
Because the samples for the two policy-group pairs are independent, the two
limits are independent. Hence
\[
\sqrt{T}
\left[
\left(
\psi(\widehat P_{kg})-\psi(\widehat P_{kg'})
\right)
-
d_{k,gg'}^\psi
\right]
\Rightarrow
\mathcal N
\left(
0,
V_{k,gg'}^\psi(\bm{\alpha})
\right),
\]
where
\[
V_{k,gg'}^\psi(\bm{\alpha})
=
\frac{\sigma_{\psi,kg}^2}{\frac{N_{kg}}{T}}
+
\frac{\sigma_{\psi,kg'}^2}{\frac{N_{kg'}}{T}}
=
T S_{T,k,gg'}^\psi,
\quad
S_{T,k,gg'}^\psi
=
\frac{\sigma_{\psi,kg}^2}{N_{kg}}
+
\frac{\sigma_{\psi,kg'}^2}{N_{kg'}}
.
\]

Given
\(
\alpha_{kg}>0
\) and \(
\alpha_{kg'}>0,
\)
\[
V_{k,gg'}^\psi(\bm{\alpha})
=
\frac{\sigma_{\psi,kg}^2}{\alpha_{kg}}
+
\frac{\sigma_{\psi,kg'}^2}{\alpha_{kg'}}.
\]
It remains to justify the Gaussian approximation.
Let
\[
d=d_{k,gg'}^\psi,
\qquad
\Delta=\Delta_{k,gg'}^\psi(\kappa)
=
\kappa-|d|,
\qquad
S_T=S_{T,k,gg'}^\psi.
\]
The normal approximation above motivates the Gaussian surrogate
\[
\psi(\widehat P_{kg})-\psi(\widehat P_{kg'})
\approx
d+Z_T,
\qquad
Z_T\sim \mathcal N(0,S_T).
\]
Since \(\Delta\geq0\), we have \(|d|\leq\kappa\). Under this Gaussian approximation,
the false probability is
\(
\mathbb P
\left(
|d+Z_T|>\kappa
\right).
\)
The closest distance from \(d\) to the complement of the feasible interval
\([-\kappa,\kappa]\) is
\(
\kappa-|d|=\Delta.
\)
Therefore, by the standard Gaussian tail approximation,
\[
\mathbb P
\left(
|d+Z_T|>\kappa
\right)
\approx
\exp
\left\{
-
\frac{\Delta^2}{2S_T}
\right\}.
\]
Substituting back the definitions of \(S_T\), \(V_{k,gg'}^\psi(\bm{\alpha})\),
and \(\Delta\), we obtain
\[
\mathbb{P}\Big(
|\psi(\widehat{P}_{kg})-\psi(\widehat{P}_{kg'})|>\kappa
\Big)
\approx
\exp\left\{
-
\frac{
\left(\Delta_{k,gg'}^\psi(\kappa)\right)^2
}{
2S_{T,k,gg'}^\psi
}
\right\}
=
\exp\left\{
-
T
\frac{
\left(\Delta_{k,gg'}^\psi(\kappa)\right)^2
}{
2V_{k,gg'}^\psi(\bm{\alpha})
}
\right\}.
\]
This gives the approximation in \eqref{eq:local-rate}.

The local verification hardness index then follows directly from the exponent.
\[
H_{k,gg'}^\psi(\kappa;\bm{\alpha})
=
\frac{
V_{k,gg'}^\psi(\bm{\alpha})
}{
\left(\Delta_{k,gg'}^\psi(\kappa)\right)^2
}
=
\frac{
T S_{T,k,gg'}^\psi
}{
\left(\Delta_{k,gg'}^\psi(\kappa)\right)^2
}.
\]
A larger value of \(H_{k,gg'}^\psi(\kappa;\bm{\alpha})\) corresponds to a
smaller Gaussian exponent and hence a slower decay rate under the Gaussian approximation.

Now consider the boundary case. If
\(
\alpha_{kg}=0
\) or \(
\alpha_{kg'}=0,
\)
then the \(T\)-scaled variance term is not finite.
\[
V_{k,gg'}^\psi(\bm{\alpha})
=T S_{T,k,gg'}^\psi=
\frac{\sigma_{\psi,kg}^2}{\frac{N_{kg}}{T}}
+
\frac{\sigma_{\psi,kg'}^2}{\frac{N_{kg'}}{T}}
=
+\infty
\]
under the boundary convention. Consequently, whenever
\(\Delta_{k,gg'}^\psi(\kappa)\ne0\),
\[
H_{k,gg'}^\psi(\kappa;\bm{\alpha})
=
\frac{
V_{k,gg'}^\psi(\bm{\alpha})
}{
\left(\Delta_{k,gg'}^\psi(\kappa)\right)^2
}
=
+\infty.
\]
Thus any comparison involving an $\alpha_{kg}$-zero-allocation pair has infinite \(T\)-scale
local hardness. If the corresponding sample size $N_{kg}$ diverges, the Gaussian
approximation based on \(S_{T,k,gg'}^\psi\) remains meaningful, but its rate is
governed by the actual growth of \(N_{kg}\) and \(N_{kg'}\).

Finally, if \(\Delta_{k,gg'}^\psi(\kappa)<0\), then the pair is truly unfair.
The error event is incorrectly classifying the pair as fair. The closest
distance from \(d_{k,gg'}^\psi\) to the feasible interval
\([-\kappa,\kappa]\) is
\(
|\Delta_{k,gg'}^\psi(\kappa)|.
\)
Repeating the same Gaussian tail calculation gives the same approximation in the exponent. We obtain
\[
\mathbb{P}\Big(
|\psi(\widehat{P}_{kg})-\psi(\widehat{P}_{kg'})|\leq\kappa
\Big)
\approx
\exp\left\{
-
\frac{
\left(\Delta_{k,gg'}^\psi(\kappa)\right)^2
}{
2S_{T,k,gg'}^\psi
}
\right\}
=
\exp\left\{
-
T
\frac{
\left(\Delta_{k,gg'}^\psi(\kappa)\right)^2
}{
2V_{k,gg'}^\psi(\bm{\alpha})
}
\right\}.
\]

This completes the proof of Proposition \ref{prop:local}.
\Halmos
\endproof

\subsection{Proof of Theorem \ref{thm:same-frontier}}

\proof{}
Let \(P_{Z_k}\) denote the distribution of the common noise term \(Z_k\). Under
the location-shift model,
\(
X_{kg}=\theta_{kg}+Z_k,
\)
the distribution of \(X_{kg}\) is
\(
P_{kg}=P_{Z_k}\oplus \theta_{kg},
\)
where \(P\oplus c\) denotes the distribution of \(X+c\) when \(X\sim P\). For
any other group \(g'\),
\(
P_{kg'}=P_{Z_k}\oplus \theta_{kg'}.
\)

Since \(\psi_j\), \(j=1,2\), is translation-equivariant,
\(
\psi_j(P\oplus c)=\psi_j(P)+c
\)
for any distribution \(P\) and constant \(c\). Therefore,
\[
\psi_j(P_{kg})
=
\psi_j(P_{Z_k}\oplus\theta_{kg})
=
\psi_j(P_{Z_k})+\theta_{kg},
\]
and similarly,
\[
\psi_j(P_{kg'})
=
\psi_j(P_{Z_k}\oplus\theta_{kg'})
=
\psi_j(P_{Z_k})+\theta_{kg'}.
\]
Subtracting the two expressions cancels the common term
\(\psi_j(P_{Z_k})\),
\begin{equation}
\label{eq:proof-gap-location}
\psi_j(P_{kg})-\psi_j(P_{kg'})
=
\theta_{kg}-\theta_{kg'},
\qquad j=1,2.
\end{equation}
This proves \eqref{eq:gap-location}.

Equation \eqref{eq:proof-gap-location} implies that the fairness score of
policy \(k\) is identical under the two fairness metrics.
\[
\varepsilon_k^{(\psi_j,\kappa)}
=
\max_{g<g'}
\left|
\psi_j(P_{kg})-\psi_j(P_{kg'})
\right|
=
\max_{g<g'}
\left|
\theta_{kg}-\theta_{kg'}
\right|,
\qquad j=1,2.
\]
Hence the two metrics with the same fairness tolerance induce the same fair policy set $\mathcal{F}_{(\psi_1,\kappa)}
=
\mathcal{F}_{(\psi_2,\kappa)}$. Since the fair policy sets are identical, the best fair objective values are
also identical.
\[
V_{(\psi_1,\kappa)}
=
\max_{k\in \mathcal{F}_{(\psi_1,\kappa)}}\mu_k
=
\max_{k\in \mathcal{F}_{(\psi_2,\kappa)}}\mu_k
=
V_{(\psi_2,\kappa)}.
\]
The unconstrained value
\(
V_0=\max_{k\in\mathcal K}\mu_k
\)
does not depend on the fairness specification. Therefore,
\[
\operatorname{PoF}_{(\psi_1,\kappa)}
=
V_0-V_{(\psi_1,\kappa)}
=
V_0-V_{(\psi_2,\kappa)}
=
\operatorname{PoF}_{(\psi_2,\kappa)}.
\]
This proves \eqref{eq:same-feasible}.

It remains to establish the ratio of local verification hardness. Under
Assumption \ref{ass:lal}, the local hardness index for metric \(\psi_j\) with fairness tolerance \(\kappa\) is
\[
H_{k,gg'}^{\psi_j}(\kappa;\bm{\alpha})
=
\frac{
V_{k,gg'}^{\psi_j}(\bm{\alpha})
}{
\left(
\Delta_{k,gg'}^{\psi_j}(\kappa)
\right)^2
},
\]
where
\[
V_{k,gg'}^{\psi_j}(\bm{\alpha})
=
\frac{\sigma_{\psi_j,kg}^2}{\frac{N_{kg}}{T}}
+
\frac{\sigma_{\psi_j,kg'}^2}{\frac{N_{kg'}}{T}}.
\]
When \(\alpha_{kg}>0\), \(\alpha_{kg'}>0\),
\[
V_{k,gg'}^{\psi_j}(\bm{\alpha})
=
\frac{\sigma_{\psi_j,kg}^2}{\alpha_{kg}}
+
\frac{\sigma_{\psi_j,kg'}^2}{\alpha_{kg'}},
\]
otherwise, \(V_{k,gg'}^{\psi_j}(\bm{\alpha})
=+\infty\).

For any allocation under which these local hardness indices are well defined,
\eqref{eq:proof-gap-location} implies that the two metrics with the same fairness tolerance have the same fairness slack for the same policy and group pair.
\[
\Delta_{k,gg'}^{\psi_j}(\kappa)
=
\kappa-
\left|
\theta_{kg}-\theta_{kg'}
\right|,
\qquad j=1,2.
\]

We next show that, for translation-equivariant metrics, the
influence function variance is invariant to group-level location shifts. For a
translation-equivariant metric \(\psi\), the influence function satisfies
\[
\operatorname{IF}_{\psi}(x+c;P\oplus c)
=
\operatorname{IF}_{\psi}(x;P).
\]
Indeed, by the definition of the influence function,
\[
\operatorname{IF}_{\psi}(x+c;P\oplus c)
=
\lim_{\varepsilon\downarrow0}
\frac{
\psi\!\left((1-\varepsilon)(P\oplus c)+\varepsilon\delta_{x+c}\right)
-
\psi(P\oplus c)
}{\varepsilon}.
\]
Since
\[
(1-\varepsilon)(P\oplus c)+\varepsilon\delta_{x+c}
=
\left((1-\varepsilon)P+\varepsilon\delta_x\right)\oplus c,
\]
translation equivariance gives
\[
\psi\!\left((1-\varepsilon)(P\oplus c)+\varepsilon\delta_{x+c}\right)
=
\psi\!\left((1-\varepsilon)P+\varepsilon\delta_x\right)+c,
\]
\[
\psi(P\oplus c)=\psi(P)+c.
\]
Subtracting the two expressions cancels the constant \(c\), so
\[
\operatorname{IF}_{\psi}(x+c;P\oplus c)
=
\operatorname{IF}_{\psi}(x;P).
\]

Applying this identity with \(P=P_{Z_k}\), \(c=\theta_{kg}\), and \(x=Z_k\),
we obtain
\[
\operatorname{IF}_{\psi_j}(X_{kg};P_{kg})
=
\operatorname{IF}_{\psi_j}(Z_k+\theta_{kg};P_{Z_k}\oplus\theta_{kg})
=
\operatorname{IF}_{\psi_j}(Z_k;P_{Z_k}).
\]
Therefore,
\[
\sigma_{\psi_j,kg}^2
=
\operatorname{Var}
\left(
\operatorname{IF}_{\psi_j}(X_{kg};P_{kg})
\right)
=
\operatorname{Var}
\left(
\operatorname{IF}_{\psi_j}(Z_k;P_{Z_k})
\right)
=
\tau_{\psi_j,k}^2,
\]
which is independent of \(g\). Thus,
\(
\sigma_{\psi_j,kg}^2
=
\sigma_{\psi_j,kg'}^2
=
\tau_{\psi_j,k}^2.
\)

Substituting this expression and the common slack into the local hardness index
gives, for \(j=1,2\),
\[
H_{k,gg'}^{\psi_j}(\kappa;\bm{\alpha})
=
\frac{
\tau_{\psi_j,k}^2
}{
\left(
\kappa-
|\theta_{kg}-\theta_{kg'}|
\right)^2
}
\left(
\frac{T}{N_{kg}}
+
\frac{T}{N_{kg'}}
\right).
\]
When \(\alpha_{kg}>0\), \(\alpha_{kg'}>0\), we have
\[
H_{k,gg'}^{\psi_j}(\kappa;\bm{\alpha})
=
\frac{
\tau_{\psi_j,k}^2
}{
\left(
\kappa-
|\theta_{kg}-\theta_{kg'}|
\right)^2
}
\left(
\frac{1}{\alpha_{kg}}
+
\frac{1}{\alpha_{kg'}}
\right).
\]
Taking the ratio of the two hardness indices, the common slack term and the
common allocation term cancel.
\[
\frac{
H_{k,gg'}^{\psi_1}(\kappa;\bm{\alpha})
}{
H_{k,gg'}^{\psi_2}(\kappa;\bm{\alpha})
}
=
\frac{
\tau_{\psi_1,k}^2
}{
\tau_{\psi_2,k}^2
}.
\]
This proves \eqref{eq:ratio-hardness} and completes the proof.
\Halmos
\endproof

\subsection{Proof of Corollary \ref{cor:mean-quantile}}

\proof{}
We first verify that both the mean functional and the quantile functional are
translation-equivariant. For the mean functional
\(\psi_{\mathrm{mean}}(P)=\mathbb{E}_P[X]\), if
\(P'=P\oplus c\), then \(X'=X+c\sim P'\) and
\[
\psi_{\mathrm{mean}}(P')
=
\mathbb{E}[X+c]
=
\mathbb{E}[X]+c
=
\psi_{\mathrm{mean}}(P)+c.
\]
Thus \(\psi_{\mathrm{mean}}\) is translation-equivariant. For the quantile
functional \(\psi_{Q_\beta}(P)=Q_\beta(P)\), if \(P'=P\oplus c\), then, under
the assumed continuity of the distribution at the target quantile,
\(
Q_\beta(P')
=
Q_\beta(P)+c.
\)
Hence \(\psi_{Q_\beta}\) is also translation-equivariant. Therefore, Theorem
\ref{thm:same-frontier} applies to both metrics under the location-shift model.

It remains to compute the influence function variances of the two functionals
under the distribution \(P_{Z_k}\) of \(Z_k\). For the mean functional, the
influence function at any distribution \(P\) with mean
\(\mu=\mathbb{E}_P[X]\) is
\(
\operatorname{IF}_{\psi_{\mathrm{mean}}}(x;P)
=
x-\mu.
\)
Applying this to \(P=P_{Z_k}\), we obtain
\(
\operatorname{IF}_{\psi_{\mathrm{mean}}}(z;P_{Z_k})
=
z-\mathbb{E}[Z_k].
\)
Thus,
\begin{equation}
\label{eq:mean-var}
\tau^2_{\mathrm{mean},k}
=
\operatorname{Var}
\left(
\operatorname{IF}_{\psi_{\mathrm{mean}}}(z;P_{Z_k})
\right)
=
\operatorname{Var}(Z_k).
\end{equation}

Next consider the quantile functional. Let
\(q_{k,\beta}=Q_\beta(P_{Z_k})\), and assume that \(P_{Z_k}\) has a continuous
density \(f_k\) that is positive at \(q_{k,\beta}\). For a general distribution
\(P\) with positive density \(f\) at \(q=Q_\beta(P)\), the influence function
of the \(\beta\)-quantile is
\[
\operatorname{IF}_{\psi_{Q_\beta}}(x;P)
=
\frac{\beta-\mathbf{1}\{x\le q\}}{f(q)}.
\]

Applying this formula to \(P=P_{Z_k}\), we have
\[
\operatorname{IF}_{\psi_{Q_\beta}}(z;P_{Z_k})
=
\frac{
\beta-\mathbf{1}\{z\le q_{k,\beta}\}
}{
f_k(q_{k,\beta})
}.
\]
Since
\(
\mathbb{P}(Z_k\le q_{k,\beta})=\beta,
\)
the influence function has mean zero. Therefore,
\begin{equation}
\label{eq:quant-var}
\tau^2_{Q_\beta,k}
=
\operatorname{Var}
\left(
\operatorname{IF}_{\psi_{Q_\beta}}(z;P_{Z_k})
\right)
=
\frac{
\beta(1-\beta)
}{
f_k(q_{k,\beta})^2
}.
\end{equation}

By Theorem \ref{thm:same-frontier}, for two translation-equivariant metrics the
ratio of local hardness indices equals the ratio of their influence function
variances.
\[
\frac{
H^{\psi_1}_{k,gg'}(\kappa;\bm{\alpha})
}{
H^{\psi_2}_{k,gg'}(\kappa;\bm{\alpha})
}
=
\frac{
\tau^2_{\psi_1,k}
}{
\tau^2_{\psi_2,k}
}.
\]
Setting \(\psi_1=\psi_{Q_\beta}\) and
\(\psi_2=\psi_{\mathrm{mean}}\), and substituting
\eqref{eq:mean-var} and \eqref{eq:quant-var}, we obtain
\[
\frac{
H^{Q_\beta}_{k,gg'}(\kappa;\bm{\alpha})
}{
H^{\mathrm{mean}}_{k,gg'}(\kappa;\bm{\alpha})
}
=
\frac{
\tau^2_{Q_\beta,k}
}{
\tau^2_{\mathrm{mean},k}
}
=
\frac{
\beta(1-\beta)
}{
\operatorname{Var}(Z_k)
\left[f_k(q_{k,\beta})\right]^2
}.
\]
This proves \eqref{eq:ratio-mq}.
\Halmos
\endproof

\subsection{Proof of Theorem \ref{thm:service-threshold}}

\proof{}
Fix a policy \(k\) and a group pair \((g,g')\). Under the location-shift model,
\(
X_{kg}=\theta_{kg}+Z_k,
\)
where \(Z_k\) has distribution function \(F_k\) and density \(f_k\), common
across groups for the fixed policy \(k\). For a threshold \(s\), the
threshold-based functional is
\(
\psi_s(P)=\mathbb{P}(X\le s).
\)
Thus,
\[
\psi_s(P_{kg})
=
\mathbb{P}(X_{kg}\le s)
=
\mathbb{P}(Z_k+\theta_{kg}\le s)
=
F_k(s-\theta_{kg}),
\]
and similarly
\(
\psi_s(P_{kg'})
=
F_k(s-\theta_{kg'}).
\)

Let $d_{k,gg'}=\theta_{kg}-\theta_{kg'}$ and $\bar\theta_{k,gg'}=\frac{\theta_{kg}+\theta_{kg'}}{2}$. Then $\theta_{kg}
=
\bar\theta_{k,gg'}+\frac{d_{k,gg'}}{2}$ and $\theta_{kg'}
=
\bar\theta_{k,gg'}-\frac{d_{k,gg'}}{2}$.
Writing
\(
x_s=s-\bar\theta_{k,gg'},
\)
we have
\[
\psi_s(P_{kg})
=
F_k\left(x_s-\frac{d_{k,gg'}}{2}\right),
\qquad
\psi_s(P_{kg'})
=
F_k\left(x_s+\frac{d_{k,gg'}}{2}\right).
\]
If \(f_k\) is continuous in a neighborhood of \(x_s\), then as
\(d_{k,gg'}\to0\),
\[
F_k\left(x_s-\frac{d_{k,gg'}}{2}\right)
=
F_k(x_s)
-
\frac{d_{k,gg'}}{2}f_k(x_s)
+
o(|d_{k,gg'}|),
\]
and
\[
F_k\left(x_s+\frac{d_{k,gg'}}{2}\right)
=
F_k(x_s)
+
\frac{d_{k,gg'}}{2}f_k(x_s)
+
o(|d_{k,gg'}|).
\]
Subtracting the two expansions gives
\[
\psi_s(P_{kg})-\psi_s(P_{kg'})
=
-d_{k,gg'} f_k(x_s)
+
o(|d_{k,gg'}|).
\]
Taking absolute values and using \(x_s=s-\bar\theta_{k,gg'}\), we obtain
\begin{equation}
\label{eq:service-first-order-proof}
|\psi_s(P_{kg})-\psi_s(P_{kg'})|
=
f_k(s-\bar\theta_{k,gg'})|d_{k,gg'}|
+
o(|d_{k,gg'}|),
\end{equation}
which proves \eqref{eq:service-first-order}.

We next consider the calibrated tolerance
\(
\kappa(s)
=
f_k(s-\bar\theta_{k,gg'})\bar\kappa.
\)
Combining this calibration with \eqref{eq:service-first-order-proof}, the
threshold-based under-service risk fairness restriction
\(
|\psi_s(P_{kg})-\psi_s(P_{kg'})|\le \kappa(s)
\)
becomes
\[
f_k(s-\bar\theta_{k,gg'})|d_{k,gg'}|
+
o(|d_{k,gg'}|)
\le
f_k(s-\bar\theta_{k,gg'})\bar\kappa.
\]
If \(f_k(s-\bar\theta_{k,gg'})>0\), this is first-order equivalent to
\(
|d_{k,gg'}|\le \bar\kappa.
\)
Thus, after calibration, threshold \(s\) imposes, to first order, the same
allowable location disparity \(\bar\kappa\).

It remains to derive the local hardness expression. The influence function of
\(\psi_s(P)=\mathbb{P}(X\le s)\) is
\[
\operatorname{IF}_{\psi_s}(x;P)
=
\mathbf{1}\{x\le s\}-\psi_s(P).
\]
Therefore,
\[
\sigma_{\psi_s,kg}^2
=
\operatorname{Var}
\left(
\operatorname{IF}_{\psi_s}(X_{kg};P_{kg})
\right)
=
\operatorname{Var}
\left(
\mathbf{1}\{X_{kg}\le s\}
\right).
\]
Since
\(
\mathbb{P}(X_{kg}\le s)
=
F_k(s-\theta_{kg}),
\)
we have
\[
\sigma_{\psi_s,kg}^2
=
F_k(s-\theta_{kg})
\left[
1-F_k(s-\theta_{kg})
\right].
\]
Similarly,
\[
\sigma_{\psi_s,kg'}^2
=
F_k(s-\theta_{kg'})
\left[
1-F_k(s-\theta_{kg'})
\right].
\]
In the small-gap regime \(d_{k,gg'}\to0\),
\[
F_k(s-\theta_{kg})
=
F_k(s-\bar\theta_{k,gg'})
+
o(1),
\]
and the same holds for group \(g'\). Hence, defining
\(
p_k(s)=F_k(s-\bar\theta_{k,gg'}),
\)
we obtain
\[
\sigma_{\psi_s,kg}^2
=
p_k(s)(1-p_k(s))+o(1),
\qquad
\sigma_{\psi_s,kg'}^2
=
p_k(s)(1-p_k(s))+o(1).
\]

By Proposition \ref{prop:local}, the local hardness index for the functional
\(\psi_s\) is
\[
H_{k,gg'}^{\SL(s)}(\kappa(s);\bm{\alpha})
=
\frac{
\frac{\sigma_{\psi_s,kg}^2}{{\frac{N_{kg}}{T}}}
+
\frac{\sigma_{\psi_s,kg'}^2}{{\frac{N_{kg'}}{T}}}
}{
\left(
\kappa(s)
-
|\psi_s(P_{kg})-\psi_s(P_{kg'})|
\right)^2
}.
\]
When \(\alpha_{kg}>0\), \(\alpha_{kg'}>0\), we have
\[
H_{k,gg'}^{\SL(s)}(\kappa(s);\bm{\alpha})
=
\frac{
\sigma_{\psi_s,kg}^2/{\alpha_{kg}}
+
\sigma_{\psi_s,kg'}^2/\alpha_{kg'}
}{
\left(
\kappa(s)
-
|\psi_s(P_{kg})-\psi_s(P_{kg'})|
\right)^2
}.
\]
Using \eqref{eq:service-first-order-proof} and the calibration
\(\kappa(s)=f_k(s-\bar\theta_{k,gg'})\bar\kappa\), the denominator satisfies
\[
\begin{split}
\kappa(s)
-
|\psi_s(P_{kg})-\psi_s(P_{kg'})|
&=
f_k(s-\bar\theta_{k,gg'})\bar\kappa
-
f_k(s-\bar\theta_{k,gg'})|d_{k,gg'}|
+
o(|d_{k,gg'}|)
\\
&=
f_k(s-\bar\theta_{k,gg'})
\big(\bar\kappa-|d_{k,gg'}|\big)
+
o(|d_{k,gg'}|).
\end{split}
\]
In the local small-gap regime \(d_{k,gg'}\to0\) with fixed
\(\bar\kappa>0\), this becomes
\[
\kappa(s)
-
|\psi_s(P_{kg})-\psi_s(P_{kg'})|
=
f_k(s-\bar\theta_{k,gg'})\bar\kappa
+
o(1).
\]
Therefore,
\[
\left(
\kappa(s)
-
|\psi_s(P_{kg})-\psi_s(P_{kg'})|
\right)^2
=
f_k(s-\bar\theta_{k,gg'})^2\bar\kappa^2
+
o(1).
\]
Substituting the numerator and denominator approximations into the local
hardness index gives
\begin{equation}
\label{eq:service-hardness-proof}
H_{k,gg'}^{\SL(s)}(\kappa(s);\bm{\alpha})
=
\frac{
p_k(s)(1-p_k(s))
}{
f_k(s-\bar\theta_{k,gg'})^2\bar\kappa^2
}
\left(
\frac{T}{N_{kg}}
+
\frac{T}{N_{kg'}}
\right)
+
o(1),
\end{equation}
which proves \eqref{eq:service-hardness}.
When \(\alpha_{kg}>0\), \(\alpha_{kg'}>0\), we have
\begin{equation}
H_{k,gg'}^{\SL(s)}(\kappa(s);\bm{\alpha})
=
\frac{
p_k(s)(1-p_k(s))
}{
f_k(s-\bar\theta_{k,gg'})^2\bar\kappa^2
}
\left(
\frac{1}{\alpha_{kg}}
+
\frac{1}{\alpha_{kg'}}
\right)
+
o(1),
\end{equation}

Finally, for two thresholds \(s_1\) and \(s_2\) calibrated to the same
\(\bar\kappa\), applying \eqref{eq:service-hardness-proof} to each threshold
yields
\[
\frac{
H_{k,gg'}^{\SL(s_1)}(\kappa(s_1);\bm{\alpha})
}{
H_{k,gg'}^{\SL(s_2)}(\kappa(s_2);\bm{\alpha})
}
=
\frac{
p_k(s_1)(1-p_k(s_1))/f_k(s_1-\bar\theta_{k,gg'})^2
}{
p_k(s_2)(1-p_k(s_2))/f_k(s_2-\bar\theta_{k,gg'})^2
}
+
o(1),
\]
which proves \eqref{eq:service-ratio}. This completes the proof.
\Halmos
\endproof

\subsection{Proof of Theorem \ref{thm:wedge}}
\label{proof:wedge}
\proof{}
We first show how the local quadratic decay rate arises near a critical
fairness tolerance. We then use this exponent to derive the implementation cost
and the implementation gap.

The key local object is a critical tolerance \(\kappa_0\) at which the best fair policy changes. There are two one-sided
configurations that can generate such a critical tolerance. In the first, \(\kappa\) decreases toward
\(\kappa_0\), and the current best fair policy is about to become unfair. In the second, \(\kappa\) increases toward \(\kappa_0\), and a high-performing but unfair
policy is about to become fair.
In both cases, the relevant local distance to the boundary is
\(
\rho(\kappa)=|\kappa-\kappa_0|.
\)

\paragraph{Case 1. the current best fair policy becomes nearly unfair as \(\kappa\downarrow\kappa_0\).}

Now suppose the current best fair policy
\(k^\star_{\mathfrak f}(\kappa)\) has a unique active group pair
\((g_0,g_0')\) whose fairness comparison becomes binding at \(\kappa_0\). That
is,
\(
\left|
d_{k^\star_{\mathfrak f},g_0g_0'}^\psi
\right|
=
\kappa_0.
\)
For \(\kappa>\kappa_0\) close to \(\kappa_0\), this policy is still truly fair,
but its slack for the active fairness comparison is
\[
\kappa
-
\left|
d_{k^\star_{\mathfrak f},g_0g_0'}^\psi
\right|
=
\kappa-\kappa_0
=
\rho(\kappa).
\]

The dominant event of false selection is that the true best fair policy is
empirically classified as unfair.
\[
\left|
\psi(\widehat P_{k^\star_{\mathfrak f}g_0})
-
\psi(\widehat P_{k^\star_{\mathfrak f}g_0'})
\right|
>
\kappa.
\]
Again, we can assume that all other fairness and ranking error decay rates are
bounded away from zero in the local neighborhood. The \(\PFS_{\mathfrak f}(T)\) decay rate is therefore
governed by this active fairness comparison.
By Proposition \ref{prop:local} and Theorem \ref{thm:rate}, the local \(\PFS_{\mathfrak f}(T)\) decay rate is
\[
\frac{
\left(
\kappa-
\left|
d_{k^\star_{\mathfrak f},g_0g_0'}^\psi
\right|
\right)^2
}{
2V_{k^\star_{\mathfrak f},g_0g_0'}^\psi(\bm{\alpha})
}
+
o\!\left((\kappa-\kappa_0)^2\right).
\]
Using
\(
\left|
d_{k^\star_{\mathfrak f},g_0g_0'}^\psi
\right|
=
\kappa_0,
\)
this becomes
\[
\frac{
(\kappa-\kappa_0)^2
}{
2V_{k^\star_{\mathfrak f},g_0g_0'}^\psi(\bm{\alpha})
}
+
o\!\left((\kappa-\kappa_0)^2\right)
=
\frac{
\rho(\kappa)^2
}{
2V_{k^\star_{\mathfrak f},g_0g_0'}^\psi(\bm{\alpha})
}
+
o\!\left(\rho(\kappa)^2\right).
\]
Thus, for a fixed allocation \(\bm{\alpha}\), the local error decay rate in this
case has the form
\[
C_{\mathfrak f}^{(1)}(\bm{\alpha})\rho(\kappa)^2
+
o\!\left(\rho(\kappa)^2\right),
\qquad
C_{\mathfrak f}^{(1)}(\bm{\alpha})
=
\frac{
1
}{
2V_{k^\star_{\mathfrak f},g_0g_0'}^\psi(\bm{\alpha})
},
\]
when \(\alpha_{k^\star_{\mathfrak f},g_0}>0\) and \(\alpha_{k^\star_{\mathfrak f},g_0'}>0\).

\paragraph{Case 2. a superior unfair policy becomes nearly fair as \(\kappa\uparrow\kappa_0\).}

Suppose there exists a policy \(k_0\) with objective value strictly larger than
that of the current best fair policy, but \(k_0\) violates the fairness
constraint when \(\kappa<\kappa_0\). Assume that, locally, this violation is
driven by a unique active group pair \((g_0,g_0')\), with
\(
\left|
d_{k_0,g_0g_0'}^\psi
\right|
=
\kappa_0.
\)
Thus, for \(\kappa<\kappa_0\) close to \(\kappa_0\), the policy \(k_0\) is
truly unfair, but its distance to the fairness boundary is
\[
\left|
d_{k_0,g_0g_0'}^\psi
\right|-\kappa
=
\kappa_0-\kappa
=
\rho(\kappa).
\]

Because \(k_0\) has a strictly larger objective value than the incumbent best
fair policy, the ranking event that \(k_0\) beats the current best fair policy has probability
tending to one and does not determine the local \(\PFS_{\mathfrak f}(T)\) decay rate. The false selection
event associated with \(k_0\) is therefore dominated by the event that this
nearly binding violated comparison is empirically classified as fair.
\[
\left|
\psi(\widehat P_{k_0g_0})
-
\psi(\widehat P_{k_0g_0'})
\right|
\le \kappa.
\]
All other fairness and ranking comparisons can be assumed to have
error decay rates bounded away from zero in this local neighborhood. Therefore,
the local \(\PFS_{\mathfrak f}(T)\) decay rate is governed by the active fairness comparison.
By Proposition \ref{prop:local} and Theorem \ref{thm:rate}, the local \(\PFS_{\mathfrak f}(T)\) decay rate is
\[
\frac{
\left(
\left|d_{k_0,g_0g_0'}^\psi\right|-\kappa
\right)^2
}{
2V_{k_0,g_0g_0'}^\psi(\bm{\alpha})
}
+
o\!\left((\kappa-\kappa_0)^2\right).
\]
Since \(\left|d_{k_0,g_0g_0'}^\psi\right|=\kappa_0\), this becomes
\[
\frac{
(\kappa_0-\kappa)^2
}{
2V_{k_0,g_0g_0'}^\psi(\bm{\alpha})
}
+
o\!\left((\kappa-\kappa_0)^2\right)
=
\frac{
\rho(\kappa)^2
}{
2V_{k_0,g_0g_0'}^\psi(\bm{\alpha})
}
+
o\!\left(\rho(\kappa)^2\right).
\]
Hence, for a fixed allocation \(\bm{\alpha}\), the local error decay rate in this
case has the form
\[
C_{\mathfrak f}^{(2)}(\bm{\alpha})\rho(\kappa)^2
+
o\!\left(\rho(\kappa)^2\right),
\qquad
C_{\mathfrak f}^{(2)}(\bm{\alpha})
=
\frac{
1
}{
2V_{k_0,g_0g_0'}^\psi(\bm{\alpha})
},
\]
when \(\alpha_{k_0,g_0}>0\) and \(\alpha_{k_0,g_0'}>0\).

For \(\kappa\) sufficiently close to \(\kappa_0\), the local analysis maintains a stable active comparison and an allocation solution that varies continuously with \(\kappa\). We restrict attention to allocations with positive shares on \((k^\star_{\mathfrak f},g_0)\), \((k^\star_{\mathfrak f},g_0')\), \((k_0,g_0)\), and \((k_0,g_0')\). Optimizing the two one-sided rate expressions over this set yields finite positive constants \(C^{(1)}_{(\psi,\kappa_0)}\) and \(C^{(2)}_{(\psi,\kappa_0)}\). When the same active comparison governs the two local regimes, these constants coincide. We denote their common value by \(C_{(\psi,\kappa_0)}\). This stability condition is the local regularity used in the main text expansion.

The two cases show that, from either side of the critical tolerance, the
\(\mathrm{PFS}_{\mathfrak f}(T)\) decay rate depends on the squared distance between \(\kappa\) and \(\kappa_0\). Hence, locally,
\begin{equation}
\label{eq:C-expansion-proof}
\Phi^\star_{\mathfrak f}
=
C_{(\psi,\kappa_0)}\rho(\kappa)^2
+
o\!\left(\rho(\kappa)^2\right),
\qquad
\rho(\kappa)=|\kappa-\kappa_0|\to0.
\end{equation}
Therefore,
\[
\Phi^\star_{\mathfrak f}
=
C_{(\psi,\kappa_0)}(\kappa-\kappa_0)^2
+
o\!\left((\kappa-\kappa_0)^2\right).
\]

We now use this local rate to derive the implementation cost. Let
\(
\mathrm{PFS}_{\mathfrak f}(T)
=
1-\PCS_{\mathfrak f}(T).
\)
By the exponential approximation for the false selection probability,
\[
-\log \mathrm{PFS}_{\mathfrak f}(T)
=
T\Phi^\star_{\mathfrak f}(1+o(1)).
\]
Equivalently,
\[
1-\PCS_{\mathfrak f}(T)=
\mathrm{PFS}_{\mathfrak f}(T)
=
\exp
\left\{
-
T\Phi^\star_{\mathfrak f}(1+o(1))
\right\}.
\]

By definition, \(B_{\mathfrak f}(\delta)\) is the smallest budget \(T\)
such that
\(
\PCS_{\mathfrak f}(T)\ge 1-\delta,
\)
or equivalently,
\(
\mathrm{PFS}_{\mathfrak f}(T)\le \delta.
\)
Using the exponential approximation, this condition is equivalent to
\[
T\Phi^\star_{\mathfrak f}(1+o(1))
\ge
\log(1/\delta).
\]
Thus, the minimum budget satisfies
\[
B_{\mathfrak f}(\delta)
=
\frac{
\log(1/\delta)
}{
\Phi^\star_{\mathfrak f}
}
(1+o(1)).
\]
Using the local expansion
\[
\Phi^\star_{\mathfrak f}
=
C_{(\psi,\kappa_0)}(\kappa-\kappa_0)^2
+
o\!\left((\kappa-\kappa_0)^2\right),
\]
gives
\[
B_{\mathfrak f}(\delta)
=
\frac{
\log(1/\delta)
}{
C_{(\psi,\kappa_0)}(\kappa-\kappa_0)^2
}
(1+o(1)).
\]

It remains to derive the tightest implementable tolerance above the critical tolerance. By definition,
\[
\kappa_{\mathrm{impl},\mathfrak f}(T,\delta)
=
\inf
\left\{
\kappa>\kappa_0\mid
B_{\mathfrak f}(\delta)\le T
\right\}.
\]
Using the asymptotic expression for \(B_{\mathfrak f}(\delta)\), the
budget constraint becomes
\[
\frac{
\log(1/\delta)
}{
C_{(\psi,\kappa_0)}(\kappa-\kappa_0)^2
}
(1+o(1))
\le
T.
\]
Equivalently,
\[
(\kappa-\kappa_0)^2
\ge
\frac{
\log(1/\delta)
}{
C_{(\psi,\kappa_0)}T
}
(1+o(1)).
\]
Since \(\kappa>\kappa_0\), taking the positive square root gives
\[
\kappa-\kappa_0
\ge
\sqrt{
\frac{
\log(1/\delta)
}{
C_{(\psi,\kappa_0)}T
}
}
(1+o(1)).
\]
The tightest implementable tolerance above the critical tolerance is the smallest \(\kappa\) satisfying this
inequality. Therefore,
\[
\kappa_{\mathrm{impl},\mathfrak f}(T,\delta)-\kappa_0
=
\sqrt{
\frac{
\log(1/\delta)
}{
C_{(\psi,\kappa_0)}T
}
}
(1+o(1)).
\]
The local expansion is valid in the regime
\(
\kappa_{\mathrm{impl},\mathfrak f}(T,\delta)\downarrow\kappa_0,
\)
equivalently when
\(
{\log(1/\delta)}/{T}
\) is sufficiently close to zero.
This completes the proof.
\Halmos
\endproof

\subsection{Proof of Corollary \ref{cor:separation}}

\proof{}
The result follows from the asymptotic implementation cost formula in Theorem
\ref{thm:wedge}. For each parity \((\psi_i,\kappa_i)\), \(i=1,2\), we have
\[
B_{(\psi_i,\kappa_i)}(\delta)
=
\frac{
\log(1/\delta)
}{
C_{(\psi_i,\kappa_0)}(\kappa_i-\kappa_0)^2
}
(1+o(1)),
\]
in the local regime where \(\kappa\) is sufficiently close to \(\kappa_0\).

Since
\(
C_{(\psi_1,\kappa_0)}>C_{(\psi_2,\kappa_0)}>0,
\)
we have
\(
{1}/{C_{(\psi_1,\kappa_0)}}
<
{1}/{C_{(\psi_2,\kappa_0)}}.
\)
Therefore, for any \(T>0\) and \(\delta\in(0,1)\),
\[
\sqrt{
\frac{\log(1/\delta)}{C_{(\psi_1,\kappa_0)}T}
}
<
\sqrt{
\frac{\log(1/\delta)}{C_{(\psi_2,\kappa_0)}T}
}.
\]
It follows that the interval
\[
I(T,\delta)
=
\left[
\kappa_0+
\sqrt{
\frac{\log(1/\delta)}{C_{(\psi_1,\kappa_0)}T}
},
\,
\kappa_0+
\sqrt{
\frac{\log(1/\delta)}{C_{(\psi_2,\kappa_0)}T}
}
\right)
\]
has positive length.

Now take any \(\kappa_1\in I(T,\delta)\). From the lower endpoint of the interval,
\[
\kappa_1-\kappa_0
\ge
\sqrt{
\frac{\log(1/\delta)}{C_{(\psi_1,\kappa_0)}T}
}.
\]
Squaring both sides gives
\[
(\kappa_1-\kappa_0)^2
\ge
\frac{\log(1/\delta)}{C_{(\psi_1,\kappa_0)}T},
\]
or equivalently,
\[
\frac{
\log(1/\delta)
}{
C_{(\psi_1,\kappa_0)}(\kappa_1-\kappa_0)^2
}
\le
T.
\]
By the asymptotic expression for \(B_{(\psi_1,\kappa_1)}(\delta)\), this implies
\(
B_{(\psi_1,\kappa_1)}(\delta)
\le
T(1+o(1)).
\)

Similarly, take any \(\kappa_2\in I(T,\delta)\), from the upper endpoint of the interval,
\[
\kappa_2-\kappa_0
<
\sqrt{
\frac{\log(1/\delta)}{C_{(\psi_2,\kappa_0)}T}
}.
\]
Hence
\[
(\kappa_2-\kappa_0)^2
<
\frac{\log(1/\delta)}{C_{(\psi_2,\kappa_0)}T},
\]
and therefore
\[
\frac{
\log(1/\delta)
}{
C_{(\psi_2,\kappa_0)}(\kappa_2-\kappa_0)^2
}
>
T.
\]
Using the asymptotic expression for \(B_{(\psi_2,\kappa_2)}(\delta)\), we obtain
\(
B_{(\psi_2,\kappa_2)}(\delta)
>
T(1+o(1)).
\)

Thus, for any \(\kappa\in I(T,\delta)\),
\[
B_{(\psi_1,\kappa_1)}(\delta)
\le
T(1+o(1))
<
B_{(\psi_2,\kappa_2)}(\delta).
\]
Therefore, under the same finite sampling budget,
\((\psi_1,\kappa_1)\) is statistically implementable while \((\psi_2,\kappa_2)\) is not
implementable at \(\mathrm{PCS}_{\mathfrak f}(T)\) target \(1-\delta\), even though the two specifications share the same fair policy set near \(\kappa_0\). Since \(\kappa_1\) and \(\kappa_2\) are arbitrary, setting \(\kappa_1=\kappa_2\) does not affect the conclusion. This proves the corollary.
\Halmos
\endproof

\subsection{Proof of Lemma \ref{lemma2.2}}
\proof{}
Our goal can be reformulated into 
\begin{equation}
\max_{N_{11}, N_{12}, \dots, N_{KG}} \mathrm{PCS}^{\pi}_{\mathfrak f}(T) \quad \text{subject to} \quad N_{11} + N_{12} + \cdots + N_{KG} = T. \label{equ2}
\end{equation}
where \(\mathrm{PCS}^{\pi}_{\mathfrak f}(T)\) is the probability of a correct selection under sampling budget \(T\) and allocation rule \(\pi\).

To streamline computations, we derive an asymptotic budget allocation strategy by considering the regime \(T \to \infty\). Let \(\alpha_{kg}\) denote the proportion of total budget allocated to policy-group pair \((k,g)\), so \(N_{kg}=\alpha_{kg}T\) with \(\sum \alpha_{kg}=1\). Problem \eqref{equ2} can be reformulated as follows.
\begin{equation}
\max_{\alpha_{11},\dots,\alpha_{KG}} \mathrm{PCS}^{\pi}_{\mathfrak f}(T) \quad \text{s.t.} \quad \sum \alpha_{kg}=1.
\label{equ4}
\end{equation}

When picking policies based on their estimated aggregate performance and empirical distribution, successfully choosing the best fair policy $k^\star_{\mathfrak f}$ as the right selection depends on two critical conditions being satisfied at once. First, policy $k^\star_{\mathfrak f}$ is estimated to be fair; second, there are no other policies that perform as well as or better than $k^\star_{\mathfrak f}$ in the empirical fair policy set.

Thus
\begin{equation}
\mathrm{PCS}^{\pi}_{\mathfrak f}(T) = \mathbb{P}\left\{
    \{k^\star_{\mathfrak f}
\in\widehat\cF_{\mathfrak f}\}\cap 
    \bigcap_{\substack{k=1 \\ k \neq k^\star_{\mathfrak f}}}^K
    \left\{
        \left[
           (k
\in\widehat\cF_{\mathfrak f}) \cap (\widehat\mu_{k^\star_{\mathfrak f}} \leq \widehat\mu_k)
        \right]^c
    \right\}
\right\}
\end{equation}

One key challenge when tackling \eqref{equ2} is that $\mathrm{PCS}_{\mathfrak f}(T)$ does not have an explicit closed-form formulation. By introducing the stated rate approximations, we obtain a tractable allocation strategy and a direct view of the comparisons that govern the budget decision in \eqref{equ2}.

\begin{equation}
\PFS^{\pi}_{\mathfrak f}(T) =1-\PCS^{\pi}_{\mathfrak f}(T)=  \mathbb{P}\left((k^\star_{\mathfrak f}
\notin\widehat\cF_{\mathfrak f})\bigcup_{\substack{ k=1\\k \neq {k^\star_{\mathfrak f}}}}^K (  k^\star_{\mathfrak f}
\in\widehat\cF_{\mathfrak f} ,k
\in\widehat\cF_{\mathfrak f}, \widehat\mu_{k^\star_{\mathfrak f}} \leq  \widehat\mu_k )\right) 
\end{equation}
\[
\begin{aligned}
\max\left( \mathbb{P}(E_0), \max_{ k \neq {k^\star_{\mathfrak f}}}  \mathbb{P}(E_k) \right)\leq \PFS^{\pi}_{\mathfrak f}(T) \leq  K\max\left( \mathbb{P}(E_0), \max_{ k \neq {k^\star_{\mathfrak f}}}  \mathbb{P}(E_k) \right).
\end{aligned}
\]
We take the negative logarithm of both sides, divide by $T$, and then take the limit as $T$ tends to infinity, while
\[
\lim_{T\to \infty} -\frac{1}{T}\log\left(\max \mathbb{P}(E_k)\right)= \min\left( - \lim_{T\to\infty} \frac{1}{T} \log \mathbb{P}(E_k) \right) ,
\]
we get
\[\lim_{T\to \infty} -\frac{1}{T}\log \PFS^{\pi}_{\mathfrak f}(T) = \min\left( - \lim_{T\to\infty} \frac{1}{T} \log \mathbb{P}(E_0), \min_{ k \neq {k^\star_{\mathfrak f}}} \left( - \lim_{T\to\infty} \frac{1}{T} \log \mathbb{P}(E_k) \right) \right).\]
\Halmos
\endproof

\subsection{Proof of Theorem \ref{thm1}}

\proof{}
For notational simplicity, write
\(
k^\star=k^\star_{\mathfrak f}.
\)
The false selection event associated with the best fair policy is the event
that \(k^\star\) is incorrectly classified as unfair.
\[
E_0
=
\left\{
k^\star\notin \widehat{\mathcal F}_{\mathfrak f}(\kappa)
\right\}.
\]
By the definition of the empirical fair policy set, this event can be written as
\[
E_0
=
\left\{
\max_{g<g'}
\left|
\psi(\widehat P_{k^\star g})
-
\psi(\widehat P_{k^\star g'})
\right|
>
\kappa
\right\}.
\]
Equivalently,
\(
E_0
=
\bigcup_{g<g'}
A_{gg'},
\)
where
\[
A_{gg'}
=
\left\{
\left|
\psi(\widehat P_{k^\star g})
-
\psi(\widehat P_{k^\star g'})
\right|
>
\kappa
\right\}.
\]

Since there are \(G(G-1)/2\) group pairs, the finite union bound gives
\[
\max_{g<g'}
\mathbb P(A_{gg'})
\le
\mathbb P(E_0)
\le
\frac{G(G-1)}{2}
\max_{g<g'}
\mathbb P(A_{gg'}).
\]
Taking logarithms, dividing by \(T\), and letting \(T\to\infty\), the constant
factor \(G(G-1)/2\) vanishes at the exponential scale. Therefore,
\[
\lim_{T\to\infty}
\frac{1}{T}
\log
\mathbb P(E_0)
=
\max_{g<g'}
\lim_{T\to\infty}
\frac{1}{T}
\log
\mathbb P(A_{gg'}).
\]

For each group pair \((g,g')\), Assumption \ref{ass:local-rate} gives the local verification error decay rate
\[
-\lim_{T\to\infty}
\frac{1}{T}
\log
\mathbb P(A_{gg'})
=
\frac{1}{2H_{k^\star,gg'}^\psi(\kappa;\bm{\alpha})}.
\]
Hence,
\[
\lim_{T\to\infty}
\frac{1}{T}
\log
\mathbb P(E_0)
=
-
\min_{g<g'}
\frac{1}{2H_{k^\star,gg'}^\psi(\kappa;\bm{\alpha})}.
\]
Equivalently,
\[
-\lim_{T\to\infty}
\frac{1}{T}
\log
\mathbb P(E_0)
=
\min_{g<g'}
\frac{1}{2H_{k^\star,gg'}^\psi(\kappa;\bm{\alpha})}.
\]
This proves the theorem.
\Halmos
\endproof

\subsection{Proof of Theorem \ref{thm:mean-rate1}}

\proof{}
For notational simplicity, write
\(
k^\star = k^\star_{\mathfrak f}.
\)
We prove the large deviation rate of the ranking error event
\(
\left\{
\widehat\mu_k \geq \widehat\mu_{k^\star}
\right\}
\)
under the true condition
\(
\mu_k < \mu_{k^\star}.
\)

For each policy-group pair \((k,g)\), let
\[
\widehat\mu_{kg}
=
\frac{1}{N_{kg}}
\sum_{i=1}^{N_{kg}}Y_{kg,i}
\]
denote the sample mean. We assume that, for each relevant distribution
\(R_{kg}\), the log moment-generating function
\[
\Lambda_{R_{kg}}(\theta)
=
\log \mathbb E_{R_{kg}}\!\left[e^{\theta Y}\right]
\]
is finite in an open neighborhood of zero. By Cramér's theorem,
\(\widehat\mu_{kg}\) satisfies a large deviation principle with rate function
\[
I_{R_{kg}}(m)
=
\sup_{\theta\in\mathbb R}
\left\{
\theta m-\Lambda_{R_{kg}}(\theta)
\right\}.
\]
Since \(N_{kg}/T\to\alpha_{kg}\), the contribution of group \(g\) to the
\(T\)-scale rate is
\(
\alpha_{kg}I_{R_{kg}}(m_{kg}).
\)

Because samples are independent across policy-group pairs, the joint vector of
group sample means for policies \(k\) and \(k^\star\),
\[
\left(
\{\widehat\mu_{kg}\}_{g\in\mathcal G_+},
\{\widehat\mu_{k^\star g}\}_{g\in\mathcal G_+}
\right),
\]
satisfies a joint large deviation principle with \(T\)-scale rate function
\[
I
\left(
\{m_{kg}\},
\{m_{k^\star g}\}
\right)
=
\sum_{g\in\mathcal G_+}
\alpha_{kg}I_{R_{kg}}(m_{kg})
+
\sum_{g\in\mathcal G_+}
\alpha_{k^\star g}I_{R_{k^\star g}}(m_{k^\star g}).
\]

Now define the continuous linear map
\[
\Psi
\left(
\{m_{kg}\},
\{m_{k^\star g}\}
\right)
=
\sum_{g\in\mathcal G_+}w_g m_{kg}
-
\sum_{g\in\mathcal G_+}w_g m_{k^\star g}.
\]
This map is the empirical difference between the aggregate weighted objective
of policy \(k\) and that of policy \(k^\star\). By the contraction principle,
the scalar random variable
\(
\widehat\mu_k-\widehat\mu_{k^\star}
\)
satisfies an LDP with rate function
\[
J_{k,k^\star}(t;\bm\alpha)
=
\inf_{\substack{
\{m_{kg}\},\{m_{k^\star g}\}\\
\sum_{g\in\mathcal G_+}w_g m_{kg}
-
\sum_{g\in\mathcal G_+}w_g m_{k^\star g}
=
t
}}
\left[
\sum_{g\in\mathcal G_+}
\alpha_{kg}I_{R_{kg}}(m_{kg})
+
\sum_{g\in\mathcal G_+}
\alpha_{k^\star g}I_{R_{k^\star g}}(m_{k^\star g})
\right].
\]

The ranking error event is
\(
\left\{
\widehat\mu_k-\widehat\mu_{k^\star}\geq0
\right\}.
\)
By the LDP, its exponential decay rate is determined by the smallest rate over
the error region. Under the true condition
\(
\Delta
=
\mu_k-\mu_{k^\star}
< 0,
\)
the rate function \(J_{k,k^\star}(t;\bm\alpha)\) attains its minimum value zero
at \(t=\Delta\). Moreover, under the standard convexity and nondegeneracy
conditions for Cramér transforms, \(J_{k,k^\star}(t;\bm\alpha)\) is convex and
is nondecreasing to the right of its minimizer. Therefore, over the error region
\(t\ge0\), the smallest rate is attained at the boundary \(t=0\).
\[
\inf_{t\ge0}J_{k,k^\star}(t;\bm\alpha)
=
J_{k,k^\star}(0;\bm\alpha).
\]
Equivalently, the strict event \(t>0\) has the same exponential rate by the
standard boundary approximation for this ranking error event.

Substituting \(t=0\) into the expression above gives
\[
J_{k,k^\star}(0;\bm\alpha)
=
\inf_{\substack{
\{m_{kg}\},\{m_{k^\star g}\}\\
\sum_{g\in\mathcal G_+}w_g m_{kg}
=
\sum_{g\in\mathcal G_+}w_g m_{k^\star g}
}}
\left[
\sum_{g\in\mathcal G_+}
\alpha_{kg}I_{R_{kg}}(m_{kg})
+
\sum_{g\in\mathcal G_+}
\alpha_{k^\star g}I_{R_{k^\star g}}(m_{k^\star g})
\right].
\]
Hence,
\[
-\lim_{T\to\infty}
\frac{1}{T}
\log
\mathbb P
\left(
\widehat\mu_k\geq\widehat\mu_{k^\star}
\right)
=
J_{k,k^\star}(0;\bm\alpha),
\]
which is the desired result.

Recall that, for \(k\neq k^\star\),
\[
E_k
=
\left\{
k^\star\in\widehat{\mathcal F}_{\mathfrak f},\,
k\in\widehat{\mathcal F}_{\mathfrak f},\,
\widehat\mu_{k^\star}\leq\widehat\mu_k
\right\}.
\]
Since \(k^\star\) is truly fair and all policy-group pairs are sampled
consistently under the maintained sampling assumptions, the law of large
numbers implies
\[
\mathbb P\left\{
k^\star\in\widehat{\mathcal F}_{\mathfrak f}
\right\}
\to 1.
\]
Similarly, for any truly fair comparison
\((g,g')\in\mathcal C_F^k\),
\[
\mathbb P\left\{
\left|
\psi(\widehat P_{kg})
-
\psi(\widehat P_{kg'})
\right|
\le \kappa
\right\}
\to 1.
\]
Therefore, at the exponential scale, the comparisons in
\(\mathcal C_F^k\) do not contribute to the decay rate of event \(E_k\).

Thus, we ignore the effects of events whose probabilities converge to one,
\(
\mathbb P\{E_k\}
\)
is controlled by the joint event
\[
\left\{
\bigcap_{(g,g')\in\mathcal C_I^k}
\left\{
\left|
\psi(\widehat P_{kg})
-
\psi(\widehat P_{kg'})
\right|
\le \kappa
\right\},
\,
\widehat\mu_{k^\star}\leq\widehat\mu_k
\right\}.
\]
For any finite collection of events \(A_1,\dots,A_m\),
\[
\mathbb P\left(\bigcap_{i=1}^m A_i\right)
\le
\min_{1\le i\le m}\mathbb P(A_i).
\]
Applying this bound to the joint event 
gives the conservative decay rate bound for the decay rate of error event \(E_k\)
\[
\begin{split}
-\liminf_{T\to\infty}
\frac{1}{T}
\log
\mathbb P\{E_k\}
\ge
\max
\Bigg\{
&
\max_{(g,g')\in\mathcal C_I^k}
\left[
-\lim_{T\to\infty}
\frac{1}{T}
\log
\mathbb P
\left(
\left|
\psi(\widehat P_{kg})
-
\psi(\widehat P_{kg'})
\right|
\le \kappa
\right)
\right],
\\
&
-\lim_{T\to\infty}
\frac{1}{T}
\log
\mathbb P
\left(
\widehat\mu_{k^\star}\leq\widehat\mu_k
\right)
\Bigg\}=R_k(\bm{\alpha}).
\end{split}
\]
For \((g,g')\in\mathcal C_I^k\), the true comparison violates the fairness
constraint, so the event
\[
\left|
\psi(\widehat P_{kg})
-
\psi(\widehat P_{kg'})
\right|
\le \kappa
\]
is a false fairness event. By Assumption \ref{ass:local-rate}, its local decay rate is
\[
-\lim_{T\to\infty}
\frac{1}{T}
\log
\mathbb P
\left(
\left|
\psi(\widehat P_{kg})
-
\psi(\widehat P_{kg'})
\right|
\le \kappa
\right)
=
\frac{1}{2H_{k,gg'}^\psi(\kappa;\bm{\alpha})}.
\]

For the ranking event, if \(\mu_k< \mu_{k^\star}\), Theorem
\ref{thm:mean-rate1} gives
\[
-\lim_{T\to\infty}
\frac{1}{T}
\log
\mathbb P
\left(
\widehat\mu_{k^\star}\leq \widehat\mu_k
\right)
=
J_{k,k^\star}(0;\bm{\alpha}).
\]
If \(\mu_k\ge\mu_{k^\star}\), the event
\(\widehat\mu_{k^\star}\leq\widehat\mu_k\) has probability converging to one, and its
decay rate is zero.

We now recover the three policy classes.

For \(k\in\Gamma\), policy \(k\) is fair but suboptimal. Hence
\(\mathcal C_I^k=\emptyset\), and the conservative decay rate bound reduces to
the ranking event term.
\[
R_k(\bm{\alpha})
=
J_{k,k^\star}(0;\bm{\alpha}).
\]

For \(k\in\mathcal S_b\), policy \(k\) is unfair but has a no worse objective
value than \(k^\star\). Therefore the ranking event has decay rate zero. The conservative decay rate bound is driven by the false fairness events.
\[
R_k(\bm{\alpha})
=
\max_{(g,g')\in\mathcal C_I^k}
\frac{1}{2H_{k,gg'}^\psi(\kappa;\bm{\alpha})}.
\]

For \(k\in\mathcal S_w\), policy \(k\) is unfair and worse than
\(k^\star\). The objective ranking error event \(E_k\) requires both false fairness and a
ranking error, and the intersection bound yields
\[
R_k(\bm{\alpha})
=
\max
\left(
\max_{(g,g')\in\mathcal C_I^k}
\frac{1}{2H_{k,gg'}^\psi(\kappa;\bm{\alpha})},
\,
J_{k,k^\star}(0;\bm{\alpha})
\right).
\]
Thus, for each \(k\neq k^\star\), \(R_k(\bm{\alpha})\) used in the allocation problem provides the conservative local lower bound on the decay rate of the objective ranking error event
\(\mathbb P\{E_k\}\).
\Halmos
\endproof

\subsection{Proof of Theorem \ref{thm:rate}}

\proof{}
For notational simplicity, write
\(
k^\star = k^\star_{\mathfrak f}.
\)
Under Assumptions \ref{ass:lal} and \ref{ass:local-rate}, and under the Cram\'er regularity condition, we establish the result by combining the preceding decay rate characterizations with the event decomposition of the false selection probability.

By Lemma \ref{lemma2.2}, the exponential decay rate of the probability of false selection equals the minimum decay rate across all error events.
\[
\lim_{T\to\infty} -\frac{1}{T}\log \PFS^{\pi}_{\mathfrak f}(T)
=
\min\left\{
-\lim_{T\to\infty}\frac{1}{T}\log \mathbb P(E_0),
\;
\min_{k\neq k^\star}
\left(
-\lim_{T\to\infty}\frac{1}{T}\log \mathbb P(E_k)
\right)
\right\}.
\]
We substitute the decay rate results from the prior theorems into this decomposition. By Theorem \ref{thm1}, the fairness verification error decay rate admits the expression
\[
-\lim_{T\to\infty}\frac{1}{T}\log \mathbb P(E_0) = R_{k^\star}(\bm{\alpha}),
\]
and by Theorem \ref{thm:mean-rate1}, for every policy \(k\neq k^\star\), the objective ranking error decay rate satisfies the lower bound
\[
-\lim_{T\to\infty}\frac{1}{T}\log \mathbb P(E_k) \ge R_k(\bm{\alpha}).
\]
By the monotonicity of the minimum operator, if \(a_i \ge b_i\) for all \(i\), then \(\min_i a_i \ge \min_i b_i\), substituting these bounds into the rate decomposition yields
\[
\lim_{T\to\infty} -\frac{1}{T}\log \PFS^{\pi}_{\mathfrak f}(T)
\ge
\min\left\{
R_{k^\star}(\bm{\alpha}),
\;
\min_{k\neq k^\star} R_k(\bm{\alpha})
\right\}
=
\min_{k\in\mathcal K} R_k(\bm{\alpha}),
\]
which verifies Equation \eqref{eq:final}.

To maximize the decay rate of \(\PFS_{\mathfrak f}(T)\) over feasible allocation vectors, we maximize this tractable lower bound, which leads to a max-min optimization problem over the allocation simplex \(\Delta\). The resulting fairness-guided budget allocation rule solves
\[
\max_{\bm{\alpha}\in\Delta} \min_{k\in\mathcal K} R_k(\bm{\alpha}),
\qquad
\Delta=
\left\{
\bm{\alpha}\in\mathbb R_+^{K\times G}\mid
\sum_{k,g}\alpha_{kg}=1
\right\},
\]
as stated in Equation \eqref{eq:optimization}.
\Halmos
\endproof

\section{Additional Numerical Details}
\label{sec: numerical details}

\subsection{Implementation of OCBA-O.}
\label{sec:OCBA-O}

OCBA-O serves as a benchmark that determines the sampling budget allocation solely based on objective value estimates, without incorporating any fairness constraints. Our implementation adopts the same sequential structure as Algorithm~\ref{alg:two-stage}, i.e., each policy-group pair receives \(n_0\) initial samples, and the remaining budget is allocated incrementally in batches of size \(T_0\).

At the start of each batch, OCBA-O computes the estimated objective value \(\widehat\mu_k\) and its corresponding estimation variance for every policy \(k\). The estimated best-performing policy is given by $\widehat k_{\mu}
    \in
    \arg\max_k \widehat\mu_k$.
For each policy \(k\neq \widehat k_{\mu}\), we define the estimated objective gap as
\(
    \widehat\Delta^\mu_k
    =
    \widehat\mu_{\widehat k_{\mu}}-\widehat\mu_k .
\)

In the classical OCBA framework, allocation decisions are fundamentally driven by the variance of the sample mean estimator for each alternative. For a single policy with \(n_k\) independent samples and variance \(\sigma_k^2\), the variance of its sample mean is \(\sigma_k^2 / n_k\), which directly quantifies estimation uncertainty and shapes the optimal budget ratio. In our multi-group setting, we aggregate group-level noise into a policy-level variance to align with the standard OCBA formulation. For policy \(k\), the estimated variance of the objective estimator is
\[
    \widehat{\sigma}_{\mu,k}^2
    =
    \sum_{g\in\mathcal{G}_+}
    w_g^2 \cdot \frac{\widehat{\sigma}_{\mu,kg}^2}{n_{kg}},
\]
where \(n_{kg}\) is the cumulative sample count for policy-group pair \((k,g)\), and \(\widehat\sigma_{\mu,kg}\) is the estimated standard deviation of the objective output for policy-group pair \((k,g)\). This aggregation preserves the core variance structure of classical OCBA while adapting to our groupwise policy definition.

Policy-level sampling proportions are then calculated following the standard OCBA rule \citep{chen2000simulation}.

As each policy covers multiple groups, the policy-level allocation share must be further allocated across groups within the policy. For policy \(k\), the within-policy allocation is proportional to the group-weighted estimated standard deviation, i.e., $\omega_{kg}
    \propto
    w_g \widehat\sigma_{\mu,kg}$ and $\sum_g \omega_{kg}=1$, 
where \(w_g\) denotes the weight of group \(g\). The resulting policy-group allocation proportions are rounded to integer sample counts within the current batch, and all performance estimates are updated after the batch of samples is collected.

\subsection{Implementation of OCBA-C.}
\label{sec: OCBA-C}

OCBA-C is a constrained benchmark adapted from constrained R\&S methods \citep{lee2012approximate}. It uses the same sequential
sampling structure as Algorithm~\ref{alg:two-stage}. Each policy-group pair
receives \(n_0\) initial samples, and the remaining budget is allocated in
batches of size \(T_0\).

At the beginning of each batch, OCBA-C constructs the empirical fair policy set
\[
\widehat{\mathcal F}_{\mathfrak f}
=
\left\{
k\mid
\max_{g<g'}
\left|
\psi(\widehat P_{kg})-\psi(\widehat P_{kg'})
\right|
\le \kappa
\right\}.
\]
It then identifies the estimated best fair policy
\(
\widehat k_{\mathfrak f}
\in
\arg\max_{k\in \widehat{\mathcal F}_{\mathfrak f}}
\widehat\mu_k .
\)
If \(\widehat{\mathcal F}_{\mathfrak f}\) is empty, OCBA-C uses the
policy with the smallest estimated fairness violation as the provisional best
fair policy. For each policy \(k\), define the estimated fairness slack
\[
    \widehat s_k^\psi
    =
    \kappa
    -
    \max_{g<g'}
    \left|
    \psi(\widehat P_{kg})-\psi(\widehat P_{kg'})
    \right|.
\]

Let \(b=\widehat k_{\mathfrak f}\). OCBA-C assigns policy priorities from two
sources, namely objective ranking and constraint verification. For each estimated
fair competitor \(k\neq b\), define the estimated objective gap
$
    \widehat\Delta^\mu_{k,b}
    =
    \widehat\mu_b-\widehat\mu_k .
$
The objective ranking score for comparing \(k\) with the current estimated best fair policy
is
\[
    D^\mu_{k,b}
    =
    \frac{\widehat V^\mu_{k,b}}
    {(\widehat\Delta^\mu_{k,b})^2+\epsilon},
\]
where \(\widehat V^\mu_{k,b}\) is the estimated variance of the objective
comparison between policies \(k\) and \(b\). This score is assigned to both
policies in the comparison. Thus, the objective priority of the current estimated best fair
policy is accumulated over its estimated fair competitors.
\[
    B^\mu_b
    =
    \sum_{\substack{k\in \widehat{\mathcal F}_{\mathfrak f}\\ k\neq b}}
    D^\mu_{k,b}.
\]
For an estimated fair competitor \(k\neq b\), set \(B^\mu_k=D^\mu_{k,b}\).
Estimated infeasible policies receive no objective ranking score.

OCBA-C also assigns a constraint verification score. Let
\(\widehat\sigma^2_{\psi,k}\) be the estimated variance of the fairness gap
estimator for policy \(k\). The constraint score is
\[
    B^\psi_k
    =
    \frac{\widehat\sigma^2_{\psi,k}}
    {(\widehat s_k^\psi)^2+\epsilon}.
\]
Estimated feasible
non-best policies receive no constraint verification score.

The overall policy priority is
$
    B^C_k
    =
    \lambda_\mu B^\mu_k
    +
    \lambda_c B^\psi_k .
$
Policy priorities are mapped to policy-group priorities by splitting the
objective and constraint terms separately. The objective term uses the OCBA-O method within-policy split.
\[
    \omega^\mu_{kg}
    =
    \frac{w_g \widehat\sigma_{\mu,kg}}
    {\sum_{g'} w_{g'} \widehat\sigma_{\mu,kg'}},
    \qquad
    \sum_g \omega^\mu_{kg}=1,
\]
where \(w_g\) is the weight of group \(g\), and
\(\widehat\sigma_{\mu,kg}\) is the estimated standard deviation of the objective
output for policy-group pair \((k,g)\). The constraint term is split by the
estimated uncertainty of the fairness metric.
\[
    \omega^\psi_{kg}
    =
    \frac{\widehat\sigma_{\psi,kg}}
    {\sum_{g'} \widehat\sigma_{\psi,kg'}},
    \qquad
    \sum_g \omega^\psi_{kg}=1,
\]
where \(\widehat\sigma_{\psi,kg}\) is the estimated standard deviation of
\(\psi(P_{kg})\). The resulting policy-group priority is
\[
    S^C_{kg}
    =
    \lambda_\mu B^\mu_k \omega^\mu_{kg}
    +
    \lambda_c B^\psi_k \omega^\psi_{kg}.
\]
Normalizing \(S^C_{kg}\) over all policy-group pairs gives the allocation
proportions for the next batch.
The budget \(T_0\) is
allocated by rounding these proportions to integer sample counts, and all
estimates are updated after the batch. In the experiments, we set
\(\lambda_\mu=\lambda_c=1\).

\subsection{Details of the First Synthetic Experiment}
\label{sec: first synthetic}
The location parameters are $\theta_{k1}=a_k+\frac{d_k}{2}$ and $\theta_{k2}=a_k-\frac{d_k}{2}$, 
where \(a_k\) controls average waiting time and \(d_k\) controls absolute group disparity. We set
\begin{align*}
    a
    &=
    (1.20, 1.10, 1.00, 0.90, 0.80, 0.70, 0.60, 0.50, 0.85, 0.75),\\
    d&
    =
    (0.02, 0.05, 0.08, 0.12, 0.16, 0.20, 0.25, 0.30, 0.14, 0.18).
\end{align*}
This design induces an efficiency--fairness tradeoff in which policies with shorter average waiting times tend to exhibit larger group disparities, reflecting service settings in which efficiency gains are unevenly distributed across groups.

The policy-specific operating costs are set to
\[
c
=
(0.000, 0.002, 0.003, 0.004, 0.005, 0.006, 0.008, 0.010, 0.005, 0.007).
\]
For the light-tailed regime, we use Gaussian noise
$
Z_k \sim N(0,\sigma_k^2),
$
with standard deviations
\[
\sigma
=
(0.077, 0.077, 0.081, 0.084, 0.088, 0.091, 0.098, 0.105, 0.088, 0.095).
\]
For the heavy-tailed regime, we let $Z_k
=
Y_k-\mathbb E[Y_k]$ and $Y_k \sim \mathrm{Lognormal}(m_k,s_k^2)$, 
so that the noise is mean-centered but retains a heavy right tail. We fix
\(
s_k^2 = 1
\), for all \(k\), and select \(m_k\) so that the heavy-tailed regime has the same variance as the corresponding Gaussian regime, namely
\(
\mathrm{Var}(Z_k)=\mathrm{Var}(Y_k)=\sigma_k^2.
\)
Using the lognormal variance formula,
\[
\mathrm{Var}(Y_k)
=
\bigl(e^{s_k^2}-1\bigr)e^{2m_k+s_k^2},
\]
this implies
\[
m_k
=
\log\!\left(
\frac{\sigma_k}{\sqrt{e-1}}
\right)
-\frac{1}{2}.
\]
Thus, the light- and heavy-tailed regimes are calibrated to have comparable marginal variance while exhibiting markedly different tail behavior.

In the first synthetic experiment, we fix a common normalized tolerance level
$
\bar{\kappa}=0.225,
$
and convert it into a metric-specific absolute fairness threshold. For each
fairness metric \(\psi\), we define
\[
D_k^\psi = \max_{g<g'} \left|\psi(P_{kg})-\psi(P_{kg'})\right|,
\]
and set the corresponding absolute threshold as
\(
\kappa^\psi = \bar{\kappa}\,\max_k D_k^\psi.
\)
In this way, the three fairness metrics are compared at the same normalized
tolerance level, while allowing their absolute thresholds to differ according
to the scale of \(D_k^\psi\).

We also report upper tail waiting time parity results under both light-tailed and heavy-tailed regimes for Synthetic Experiment I in Figure~\ref{fig:cvar_results}.

\begin{figure}[t]
    \centering
    \includegraphics[width=0.85\linewidth]{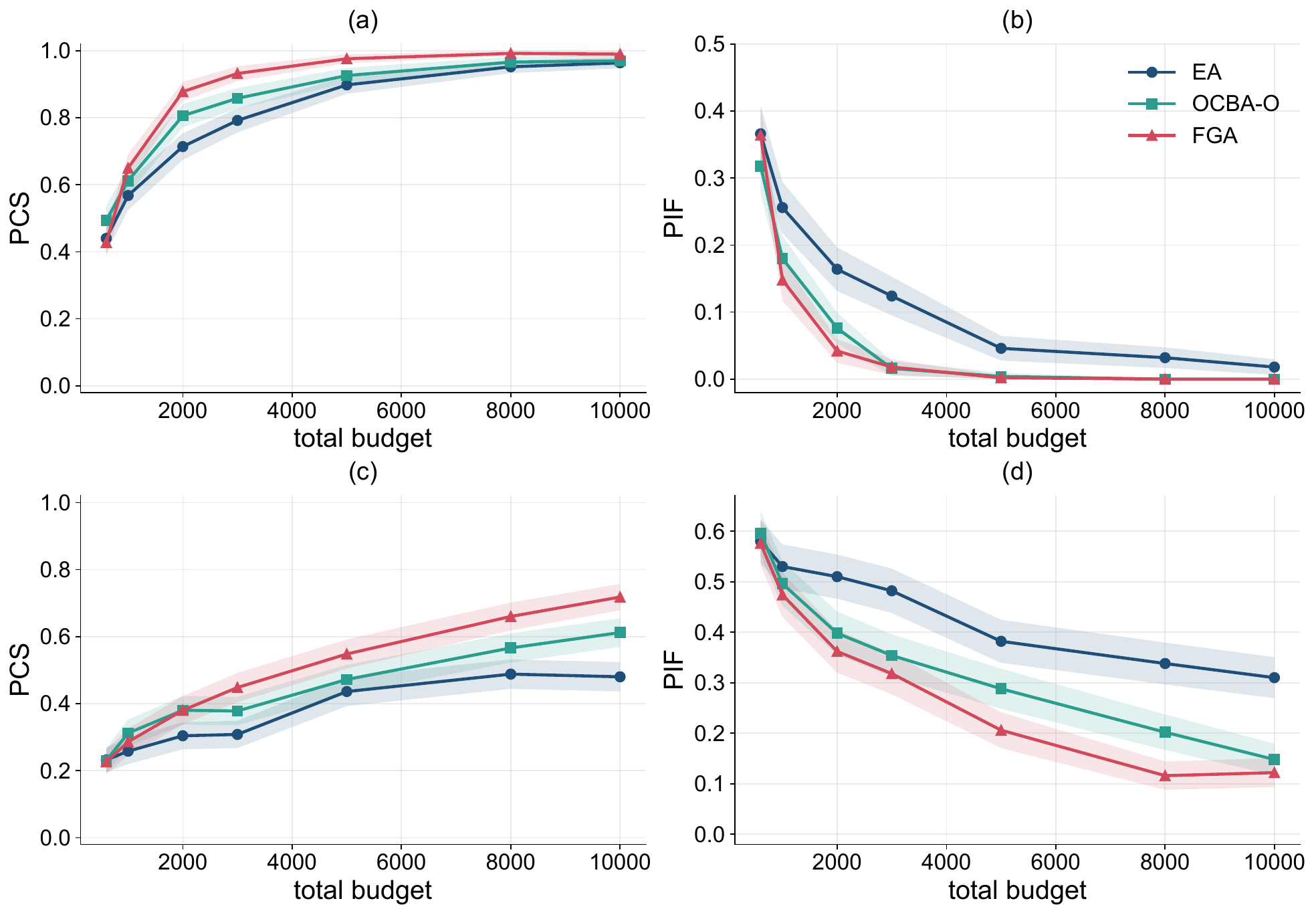}
    \caption{Upper Tail Waiting Time Parity Results for Synthetic Experiment I.
    (a) PCS under the parity in the light-tailed regime.
    (b) PIF under the parity in the light-tailed regime.
    (c) PCS under the parity in the heavy-tailed regime.
    (d) PIF under the parity in the heavy-tailed regime.}
    \label{fig:cvar_results}
\end{figure}

Under light-tailed noise, FGA again
outperforms the benchmarks, achieving higher PCS and lower PIF across the
sampling budget range. Under heavy-tailed noise, all methods exhibit slower improvement, with PCS increasing more slowly and PIF remaining elevated even at larger budgets. This pattern is consistent with the theoretical prediction that tail-sensitive fairness metrics are
particularly difficult to verify when the outcome distribution has heavy tails.

\subsection{Details of the Second Synthetic Experiment}
In the second synthetic experiment, the true objective values and corresponding fairness disparities are
\begin{align*}
    \mu
    &=
    (1.00,0.98,0.96,0.94,0.92,0.90,0.88,0.86,0.84,0.82,0.80,0.78),\\
    D
    &=
    (0.18,0.15,0.11,0.095,0.090,0.070,0.060,0.050,0.040,0.030,0.020,0.010).
\end{align*}

\subsection{Details of the Call Center Experiment}
\label{sec: call center exp}
The
system operates over an \(8\)-hour horizon with \(11\) parallel servers, and
group weights are set to \(w_1=w_2=0.5\). Customer arrivals follow
nonstationary Poisson processes with piecewise-constant hourly rates
\begin{align*}
\lambda_1&=(42,55,70,86,92,82,60,45),\\
\lambda_2&=(38,50,65,78,88,92,70,52).
\end{align*}
Service times are group-dependent and lognormally distributed with $S_1 \sim \mathrm{Lognormal}(5.348496,\,0.264285)$ and $S_2 \sim \mathrm{Lognormal}(5.504394,\,0.398776)$. Patience times are exponential with means $\mathbb{E}[P_1]=420$ and 
$\mathbb{E}[P_2]=360$
(seconds). A customer abandons if service has not started before the patience
time expires.

We consider \(K=12\) policies, formed by combining three congestion thresholds
\(
b \in \{5,10,15\}
\)
with four queue disciplines, namely FCFS, static priority for group 1, static priority
for group 2, and weighted priority. For weighted priority, the group weights are
set to \((1.00,1.35)\). The service performance threshold is \(120\) seconds.

The policy complexity score \(C_k\) is determined by the queue discipline and
the congestion threshold. Specifically, the base complexity is set to \(0\) for
FCFS, \(0.45\) for static priority rules, and \(0.70\) for weighted priority,
and we add a threshold adjustment
\(
0.25 \times \frac{15-b}{20},
\)
where \(b \in \{5,10,15\}\). Thus,
\[
C_k
=
\min\!\left\{
1,\;
C_k^{\mathrm{base}}
+
0.25\frac{15-b_k}{20}
\right\}.
\]

Each algorithm begins with \(n_0=20\) initial replications per policy-group
pair. One replication is an independently seeded simulated day. It returns a
group-level objective contribution and the within-day batch of group-specific
customer waits. The simulated day is the independent sampling unit; customer
records observed within a day are retained as a batch when estimating the
waiting time distribution. For FGA, the variance of each fairness functional
is estimated from day-level influence contributions. This construction
accounts for within-day dependence and random group-specific customer counts.

\subsection{Details of the Emergency Department Experiment}
\label{sec: healthcare exp}

The emergency department operates over a 12-hour day with \(G=2\) patient
groups. Patient arrivals follow a nonstationary Poisson process with
piecewise-constant hourly rates

$$
    (14,22,30,20).
$$

Each patient is assigned to one of the two groups with equal probability and
is high-acuity with probability \(0.35\). The system includes two triage
nurses, three main clinicians, two fast-track clinicians, and one
communication support staff member. Group~2 patients require communication
support with probability \(0.90\).

Triage, treatment, and communication support times are lognormal. Mean triage
times are 5 and 9 minutes for groups 1 and 2, respectively. Mean main
treatment times are 44 and 52 minutes for high-acuity patients and 30 and
38 minutes for low-acuity patients. Mean fast-track treatment times are
14 and 26 minutes, and mean communication support time is 14 minutes.
Patient patience times are exponential, with means of 95 minutes for
high-acuity patients and 70 minutes for low-acuity patients. Patients abandon
if service has not begun before their patience time expires.

The service burden outcome is the waiting time \(W_{kg}\) for group \(g\)
under policy \(k\). For high-acuity patients, waiting time is measured as
door-to-provider time, whereas for low-acuity patients it is measured from
arrival to completion. For patients who abandon, waiting time is the time
spent in the system before abandonment.

One independently seeded simulated day is one outer replication. Each day
produces the objective contribution and a batch of patient-level waiting
times for each group. Fairness estimators pool the within-day batches and
estimate uncertainty from day-level influence contributions. Thus, the
effective independence assumption applies across simulated days rather than
across individual patients within the same day.

In the emergency department experiments, \(C_k\) denotes the normalized implementation complexity of policy \(k\). It is a policy-specific scalar used to penalize operationally more complicated scheduling rules in the objective function. For the 12 candidate policies in the near-boundary emergency department experiment, the complexity values are
\[
(0.00,\ 0.02,\ 0.04,\ 0.06,\ 0.08,\ 0.10,\ 0.12,\ 0.14,\ 0.16,\ 0.16,\ 0.18,\ 0.25),
\]
Larger values of \(C_k\) represent policies that are more operationally complex to implement and maintain.

This subsection also presents supplementary experimental results for the emergency department setting, which complement the main analysis under upper tail parity in Section \ref{sec:ed}. We report results for two additional fairness parities, namely mean waiting time performance parity and 90th quantile of waiting time parity. The supplementary experiments use the same emergency department configuration
and candidate policy set described above.

For mean waiting time performance parity, we set fairness tolerance \(\kappa=13\).

\begin{figure}[htbp]
    \centering
    \includegraphics[width=0.85\textwidth]{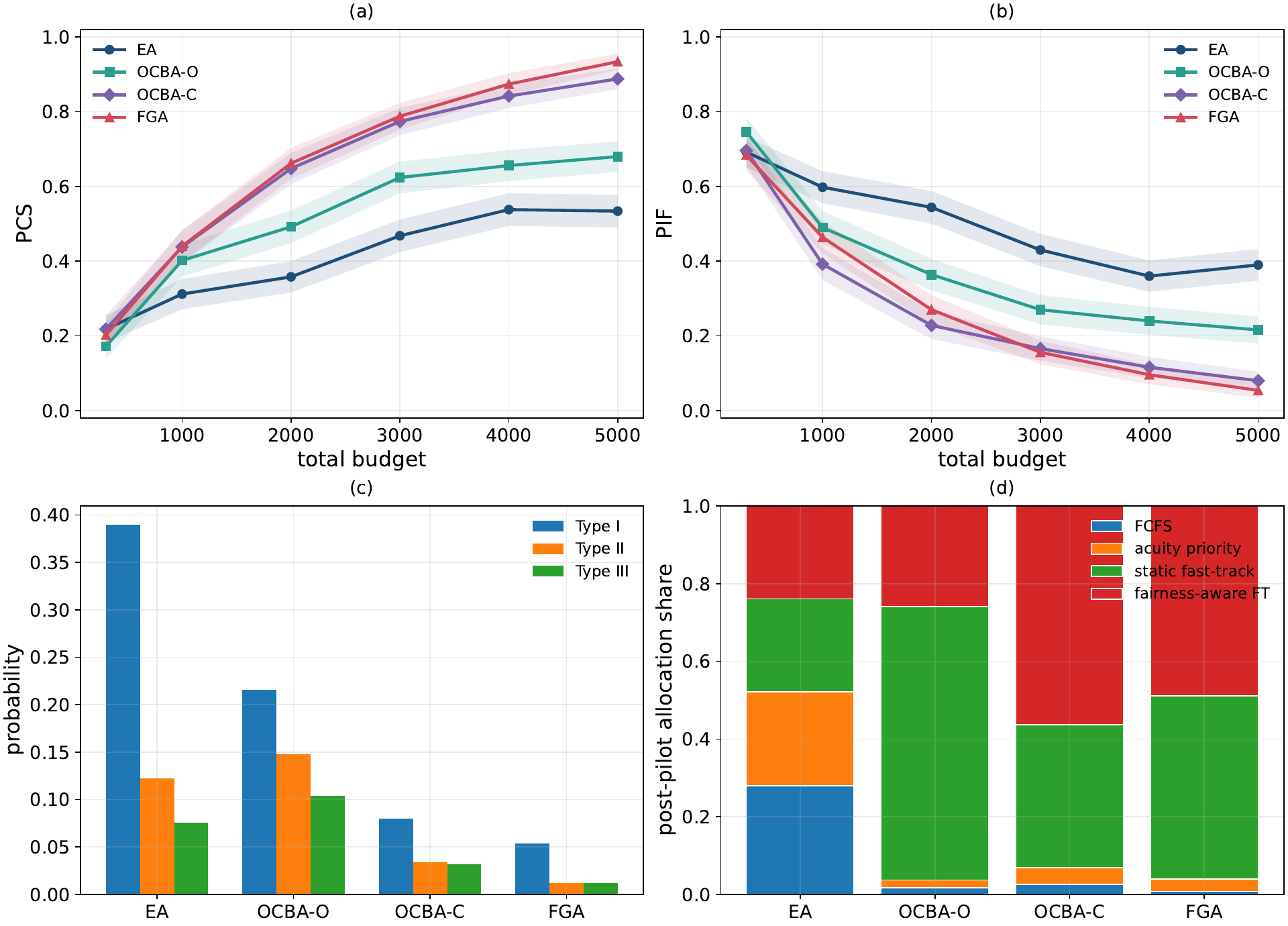}
    \caption{
    Emergency Department Experiment under Mean Waiting Time Performance Parity.
    (a) PCS of the algorithms.
    (b) PIF of the algorithms.
    (c) Error type decomposition across algorithms at \(T=5000\).
    (d) Budget allocation decomposition at \(T=5000\).
    }
    \label{fig:healthcare_mean}
\end{figure}

Figure \ref{fig:healthcare_mean} reports the selection performance under mean performance parity, with the fairness tolerance calibrated to induce the same best fair policy as the upper tail parity case in the main text. Mean performance parity yields higher overall PCS for fixed sampling budget and lower implementation cost than tail-sensitive parity cases, as the mean estimator has smaller influence function variance. This is consistent with our theoretical prediction and implies that operational cost equivalence does not imply implementation cost equivalence.

Panel~(a) and Panel~(b) illustrate that the relative performance ranking of the four algorithms remains unchanged. FGA achieves the highest PCS and lowest PIF across all budget levels, followed by OCBA-C.
Panel~(c) decomposes the selection errors at \(T=5000\). For EA and OCBA-O, errors are dominated by Type~I errors, reflecting frequent false selection of unfair policies. OCBA-C and FGA mitigate this error type, yet Type~I error remains its largest error component relative to Type~II and Type~III errors. FGA achieves balanced reductions across all three error categories and yields the lowest error probability for each type, resulting in the best overall selection performance.
Panel~(d) explains these performance differences through the decomposition of the allocated sampling budget. EA serves as the proportional baseline, where the budget share of each policy class is strictly determined by the number of policies in that category. Both FGA and OCBA-C allocate substantial sampling effort to fairness-aware fast-track policies as well as static fast-track policies. By contrast, OCBA-O concentrates heavily on static fast-track policies alone. These allocation patterns further illustrate that explicitly targeting difficult fairness verification comparisons improves both fairness compliance and finite-budget PCS in the emergency department setting.

For 90th quantile of waiting time parity, we set fairness tolerance \(\kappa=28\).

\begin{figure}[htbp]
    \centering
    \includegraphics[width=0.85\textwidth]{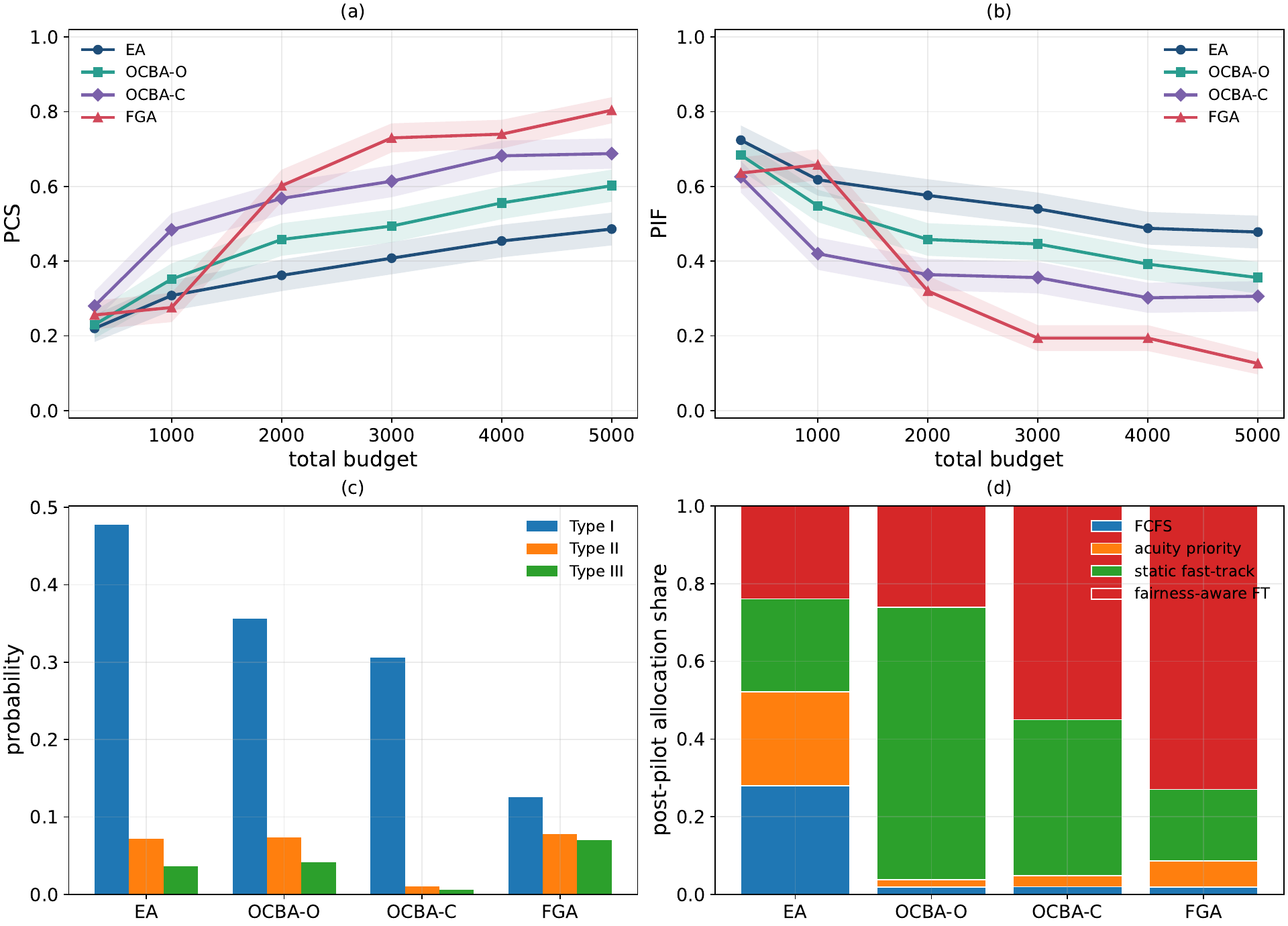}
    \caption{
    Emergency Department Experiment under 90th Quantile of Waiting Time Parity.
    (a) PCS of the algorithms.
    (b) PIF of the algorithms.
    (c) Error type decomposition across algorithms at \(T=5000\).
    (d) Budget allocation decomposition at \(T=5000\).
    }
    \label{fig:healthcare_quantile}
\end{figure}

Figure~\ref{fig:healthcare_quantile} presents the emergency department experiment results under high quantile parity. 

Panel~(a) shows that FGA consistently achieves the highest PCS across all sampling budget levels, followed by OCBA-C. EA and OCBA-O improve more slowly and remain less reliable even at \(T=5000\). 
Compared with the mean performance parity setting, the overall PCS levels are lower for all four algorithms, which shows that the quantile-based fairness specification incurs a higher implementation cost, even though the specifications have the same operational cost.
Panel~(b) reports that FGA yields the steepest decline in PIF. By contrast, EA and OCBA-O have slowly declining elevated PIF even as the sampling budget grows, as they do not strategically allocate samples to fairness verification comparisons.
Panel~(c) decomposes the selection errors at \(T=5000\). For EA and OCBA-O, errors are overwhelmingly dominated by Type~I errors, reflecting frequent false selection of high-performing but unfair policies. OCBA-C and FGA substantially reduce Type~I errors, but it remains the largest error component relative to Type~II and Type~III errors. FGA achieves the lowest Type~I error probability, while OCBA-C achieves the lowest Type~II and Type~III error probability.
Panel~(d) illustrates the post-pilot budget allocation decomposition at \(T=5000\). EA serves as the proportional baseline, where the budget share of each policy class is strictly determined by the number of policies in that category. OCBA-O concentrates most of its sampling budget on static fast-track policies. In comparison, both OCBA-C and FGA allocate a much larger share of budget to fairness-aware fast-track policies.

Results for mean performance parity and high-quantile parity follow the same qualitative pattern as the upper-tail parity case. Across all three fairness metrics, fairness-guided allocation consistently has better selection performance than other algorithms. The intuition is straightforward: explicitly prioritizing the hardest verification comparisons puts sample budget where it matters most, which directly boosts selection accuracy when sampling budgets are finite.

\end{document}